\documentclass[11pt,a4paper]{article}
\usepackage[utf8]{inputenc}
\usepackage{amsmath}
\usepackage{amsfonts}
\usepackage{amssymb}
\usepackage{graphicx}
\usepackage{xcolor}
\usepackage{comment}
\usepackage{amsthm}
\usepackage[shortlabels]{enumitem}
\usepackage{mathrsfs}
\newcommand{\mathbbm}[1]{\text{\usefont{U}{bbm}{m}{n}#1}}
\usepackage{hyperref}
\usepackage{amssymb}
\usepackage{algorithm}
\usepackage{algorithmic}
\usepackage[letterpaper,
left = 1.1truein,  
right = 1.1truein, 
top = 1.1truein, 
bottom = 1.1truein]{geometry} 

\usepackage{definitions}
\newtheorem{theorem}{Theorem}
\newtheorem{corollary}{Corollary}
\newtheorem{lemma}{Lemma}
\theoremstyle{remark}
\newtheorem{remark}{Remark}
\newtheorem{definition}{Definition}

\newcommand{\bm}{\boldsymbol}
\title{Optimal  Federated Learning for Functional Mean Estimation under Heterogeneous Privacy Constraints}
\author{Tony Cai, Abhinav Chakraborty, Lasse Vuursteen}
\date{\today}
\begin{document}
	
\maketitle

\begin{abstract}
Federated learning (FL) is a distributed machine learning technique designed to preserve data privacy and security, and it has gained significant importance due to its broad range of applications. This paper addresses the problem of optimal functional mean estimation from discretely sampled data in a federated setting.

We consider a heterogeneous framework where the number of individuals, measurements per individual, and privacy parameters vary across one or more servers, under both common and independent design settings. In the common design setting, the same design points are measured for each individual, whereas in the independent design, each individual has their own random collection of design points. Within this framework, we establish minimax upper and lower bounds for the estimation error of the underlying mean function, highlighting the nuanced differences between common and independent designs under distributed privacy constraints.

We propose algorithms that achieve the optimal trade-off between privacy and accuracy and provide optimality results that quantify the fundamental limits of private functional mean estimation across diverse distributed settings. These results characterize the cost of privacy and offer practical insights into the potential for privacy-preserving statistical analysis in federated environments.

%	This paper investigates the cost of privacy in the problem of estimating the mean of random functions based on discretely sampled data, where the data is subject to differential privacy constraints. We consider a heterogeneous framework that accommodates for federated settings where the numbers of individuals, measurements per individual, and privacy parameters can vary across one or more servers, under both common and independent design settings. In the common design scenario, the same design points are measured for each individual, whereas the independent design allows each individual to have its own random collection of design points. Within this framework, we establish minimax upper and lower bounds for the estimation error of the underlying mean function, delineating the nuanced differences between common and independent designs under distributed privacy constraints. Our work provides algorithms that achieve the optimal trade-off between privacy and accuracy, as well as theoretical results that quantify the fundamental limits of private functional mean estimation in various distributed settings. These findings offer practical insights into the performance possibilities of conducting privacy-preserving statistical analysis in federated environments.

	\end{abstract}

\section{Introduction}

In the era of big data, sensitive information is often distributed across multiple sources in fields such as healthcare and financial services. Privacy concerns prevent direct data pooling, making it essential to develop efficient statistical inference methods that preserve privacy while leveraging the collective power of distributed data. Federated learning addresses this challenge by enabling organizations or groups to collaboratively train and improve a shared global machine learning model without sharing raw data externally, ensuring that privacy is maintained at each data source. Differential privacy (DP) \cite{dwork2006calibrating} has emerged as the leading framework for providing rigorous privacy guarantees, offering a principled approach to quantifying the extent to which individual privacy is protected within a dataset.

The problem of estimating the mean of random functions based on discretely sampled data arises naturally in functional data analysis and has significant applications in areas such as signal processing, biomedical imaging, environmental monitoring, and financial modeling \cite{ramsay2005principal,ferraty2006nonparametric}. Classical methods have explored this problem extensively, with adaptive algorithms and theoretical results establishing optimal rates of convergence under varying sampling schemes \cite{cai2011meanfunction, horiguchi2023sampling, wang2016functional, rice1991estimating}. These methods emphasize a critical balance: increasing the number of subjects reduces variability across the population, while increasing the number of measurements per subject improves the precision of individual function estimates. This trade-off is fundamental to designing efficient sampling strategies.

Incorporating differential privacy into functional mean estimation introduces unique challenges. DP requires a trade-off between privacy and accuracy by adding noise to the data or the estimation process. For functional data, this trade-off is further complicated by the relationship between sampling intensity and privacy risk -- taking more measurements per individual increases privacy risks, even though it enhances estimation accuracy. Developing DP-compliant estimators requires careful consideration of both privacy constraints and the intrinsic properties of functional data, diverging significantly from traditional non-private methods.

The federated learning paradigm introduces additional complexities to functional mean estimation under privacy constraints. In federated settings, data is distributed across multiple servers or entities, such as hospitals, mobile devices, or autonomous vehicles, each with distinct privacy requirements, see e.g. \cite{beaufays2019federated} and references therein for an overview. These settings often involve heterogeneous data characterized by varying numbers of individuals, measurements, and privacy budgets across servers. This heterogeneity exacerbates the challenges of balancing privacy and statistical accuracy.

This paper investigates the cost of privacy in functional mean estimation within the framework of Federated Differential Privacy (FDP). FDP extends DP principles to distributed environments, enabling privacy-preserving statistical analysis at various levels, including individual data holders, servers, or aggregated data.  Private functional data analysis, especially when the private outputs are functions, has received recent  attention, primarily within the central DP framework \cite{hall2013differential,mirshani2019formal}. 

Our work makes several key contributions. First, we establish minimax upper and lower bounds for the estimation error of the mean function under both common and independent design settings, highlighting key differences under FDP constraints. Second, we propose novel algorithms that achieve the optimal trade-off between privacy and accuracy, accommodating heterogeneity in the number of individuals, measurements, and privacy budgets across servers. Third, we extend lower bound techniques by introducing innovative data processing methods that address information bottlenecks arising from server heterogeneity. Lastly, we quantify the fundamental limits of private functional mean estimation across various distributed environments, offering practical insights for privacy-preserving statistical analysis of functional data under diverse conditions, from local to central DP frameworks.

\subsection{Problem Formulation}

We consider the problem of distributed functional mean estimation under privacy constraints in a setting where data is distributed across multiple servers. Let \(N = \sum_{s=1}^S n_s\) denote the total number of observations, with server \(s = 1, \ldots, S\) holding \(n_s\) independent observations.

Let $X(.)$ be a random function defined on $[0,1]$, and for $s=1,\ldots,S$, let $X_1^{(s)},\ldots,X^{(s)}_{n_s}$ be independent copies of $X$. The goal is to estimate the mean function $f(\cdot) = \E(X(\cdot))$ based on noisy observations from discrete  locations on these curves distributed across $S$ servers:
\[
Y_{ij}^{(s)}= X_i(\zeta^{(s)}_{ij}) + \xi_{ij}^{(s)}, i=1,\ldots n_s \text{ and } j = 1,\ldots, m_s,
\]
where $\zeta^{(s)}_{ij}$ are the sampling points and $ \xi_{ij}^{(s)} \sim N(0, \sigma^2)$ (without loss of generality we assume $\sigma =1$). The sample path
of  \(X\) is assumed to be smooth in that it belongs to $\cH^\alpha(R)$ which also implies that the mean \(f \in \cH^\alpha(R)\) (where \(\E X_i(\cdot) = f(\cdot)\)). Here, \(\cH^\alpha(R)\) denotes the Hölder ball of radius \(R\), within the class of \(\alpha\)-Hölder continuous functions.  
In the classical setting without privacy constraints and distributed data, this problem has been extensively studied in the statistical literature (e.g., \cite{ferraty2006nonparametric, cai2011meanfunction}). In this article, we focus on quantifying the cost of privacy in a distributed setup for this problem. 

 So far, we have not discussed the design of the experiment. We will consider two cases: the {common design} and the {independent design}. In the \emph{common design} case, all individuals are observed on the same locations (grid points) i.e all $s\in[S]$ and $i\in[n_s]$, $\zeta^{(s)}_{ij} = \zeta_j$ for all $j \in [m_s]$. Note that for the common design the design points are not considered to be private since it is shared among all the individuals. In the \emph{independent design} case, the grid points are random and distinct for each individual. In this case, the design points are considered to be private. This reflects real-world situations where individuals could be identified based on specific measurement time points. We make the simplifying assumption that they are i.i.d. draws from a uniform distribution on \([0,1]\). This can be relaxed to non-uniform distributions that are e.g. bounded above and below by a constant, and the estimator presented in this paper can be easily adapted to non-uniform designs using the privacy preserving technique presented in the Supplementary Material of \cite{cai2024optimal}. 

To account for privacy constraints, we adopt the general framework of distributed estimation under privacy constraints introduced in \cite{cai2024optimal}. Let \(Z^{(s)} = \{Z_i^{(s)}\}_{i=1}^{n_s}\) denote the dataset held by server \(s\), with \(Z_i^{(s)} = \{Y_{ij}^{(s)}, \zeta_{ij}^{(s)}\}_{j=1}^{m_s}\). Each server \(s\) outputs a (randomized) transcript \(T^{(s)}\) based on \(Z^{(s)}\), with the law of the transcript determined conditionally on \(Z^{(s)}\), i.e., \(\P(\cdot | Z^{(s)})\). The collection of transcripts \(T = (T^{(s)})_{s=1}^S\) must satisfy \((\bm \varepsilon, \bm \delta) = (\varepsilon_s, \delta_s)_{s=1}^S\)-federated differential privacy (FDP), defined as follows:

\begin{definition}\label{def:differential_privacy}
	The transcript $T = (T^{(s)})_{s=1}^S$ is $(\bm \varepsilon,\bm \delta)$-federated differentially private (FDP) if for all $s\in[S]$, $A \in \mathscr{T}$ and $z,z' \in \mathcal{Z}^{n_j}$ differing in one individual datum it holds that
	\begin{equation*}
		\P \left( T^{(s)} \in A | Z^{(s)} = z \right) \leq  e^{\varepsilon_s} \P \left( T^{(s)} \in A | Z^{(s)} = z' \right)  + \delta_s.
	\end{equation*}
\end{definition}

In the above definition,``differing in one datum" refers to being Hamming distance ``neighbors." Specifically, local datasets $Z^{(s)}$ and $\tilde{Z}^{(s)}$ are \emph{neighboring} if their Hamming distance is at most 1, calculated over $\mathcal{Z}^{n_s} \times \mathcal{Z}^{n_s}$. In other words, $\tilde{Z}^{(s)}$ can be derived from $Z^{(s)}$ by modifying at most one observation among $Z^{(s)}_1,\dots,Z^{(s)}_{n_s}$. In the common design setup, where the design points are shared among all individuals, neighboring datasets share the same design points, and only the measurements 
$Y$ differ..  The smaller the values of $\varepsilon_s$ and $\delta_s$, the stricter the privacy constraint. We consider $\varepsilon_s \leq C_{\varepsilon}$ for $s=1,\dots,S$, with a fixed constant $C_{\varepsilon} > 0$ that does not affect the derived rates.

The Federated Differential Privacy (FDP) framework addresses scenarios where sensitive data is distributed across multiple parties, each generating outputs while maintaining differential privacy. In this distributed protocol, each server's transcript depends solely on its local data, with no direct communication or data exchange between servers. This framework is particularly relevant in settings like multi-site studies or trials conducted on the same population, where institutions (e.g., hospitals) aim to collaborate without sharing raw data due to privacy concerns. The FDP framework generalizes commonly studied settings, including the local DP model $(n_s =1)$, where privacy mechanisms operate at the individual level, and the central DP model $(S=1)$, where data is aggregated at a single location.

Each server transmits its transcript to a central server. Using all transcripts \(T = (T^{(1)}, \ldots, T^{(S)})\), the central server computes an estimator \(\hat{f} : \mathcal{T}^S \to \mathcal{F}\). We refer to the pair \((\hat{f}, \{\P(\cdot \mid Z^{(s)})\}_{s=1}^S)\) as a distributed estimation protocol. The class of distributed estimation protocols satisfying Definition~\ref{def:differential_privacy} is denoted by \(\mathcal{M}(\bm \varepsilon, \bm \delta)\). We let $\P_f$ denote the joint law of transcripts and the $N = \sum_{s=1}^S n_s$  observations. We let $\E_f$ denote the expectation corresponding to $\P_f$.

The aim is to estimate the function $f$ based on the distributed data. The difficulty of this estimation task arises from both the distributed nature of the data and privacy constraints that limit the sharing of information between servers.
As in the conventional decision-theoretical framework, the estimation accuracy of a distributed estimator $\hat f \equiv \hat f(T)$ is measured by the integrated mean squared error (IMSE), $\E_{f} \|\hat f - f\|_2^2$, where the expectation is taken over the randomness in both the data and construction of the transcripts, and the quantity of particular interest is the \emph{minimax risk} for the distributed private protocols over function class $\cF$ which we take to be the Hölder class \(\cH^\alpha(R)\) with $\alpha>1/2$,
\begin{equation}\label{eq:def-minimax-risk}
	\inf_{\hat f \in \cM(\bm \varepsilon,\bm\delta)} \sup_{f \in \cF} \, \E_{f} \|\hat f - f\|_2^2.
\end{equation}
This `global' risk characterizes the difficulty of the federated learning problem over the function class when trying to infer the entire function underlying the data whilst adhering to the heterogeneous privacy constraints.

%The minimax risk characterizes the difficulty of the distributed learning problem over the function class $\cF$ when trying to infer the entire function underlying the data whilst adhering to the heterogeneous privacy constraints. 

%\subsection{Federated Differential Privacy}

% explain the problem

% explain classically the situation

% novel lowerbound

% novel upperbound

% problem formulation (including privacy?)

% literature review
\subsection{Related work}

Federated Differential Privacy (FDP), as considered in this paper, applies DP at the level of local samples consisting of multiple observations. Optimal statistical inference under this framework has been studied in the context of nonparametric models for estimation \cite{cai2024optimal} and goodness-of-fit testing \cite{cai2024privateTesting}, with \cite{cai2024optimal} addressing server heterogeneity. The homogeneous setting has been explored for discrete distributions \cite{liu2020learning, pmlr-v206-acharya23a_user_level_LDP} and parametric mean estimation \cite{levy2021learning, narayanan2022tight}. Additionally, \cite{canonne2024private} examines discrete distribution testing in a two-server setting ($m=2$) with differing DP constraints.

Upon completion of this work, we came across the paper \cite{xue2024optimalestimationprivatedistributed}, which considers a setting similar to ours and claims some partially overlapping results. There are important differences between the papers, however. Foremost, both their general upper- and lower-bounds in \cite{xue2024optimalestimationprivatedistributed} exhibit a polynomial gap, whereas our upper bound provides a strict improved in certain regimes and our lower bound matches our upper bound up to polylogarithmic factors. Another important difference is that  in \cite{xue2024optimalestimationprivatedistributed}, only the random design case is considered, which means the contrast with the fixed design setup are not treated in their work. Furthermore, the methods considered in the two papers differ: the method in \cite{xue2024optimalestimationprivatedistributed} requires a sequential federated learning algorithm, whereas our method is a one-shot algorithm.

In terms of lower bound techniques, several approaches have been specifically developed for private settings. \cite{lalanne2023statistical} and \cite{acharya2021differentially} explore private adaptations of general methods, such as Fano, Assouad, and Le Cam techniques, for establishing lower bounds in the central DP setting. In contrast, \cite{barnes2020fisher} develops Van Trees-based lower bounds tailored to local DP, which do not directly extend to central DP. Similarly, the techniques in \cite{acharya2023optimal, acharya2024unified} are designed for the local DP setting and lack straightforward extensions to central DP.

Building on these foundations, \cite{cai2024optimal} extends the Van Trees-based lower bound techniques of \cite{barnes2020fisher} to the federated DP setting, addressing the complexities introduced by distributed data and federated privacy constraints. In this paper, we further advance these methodologies by applying them to the functional mean estimation problem. We introduce novel data processing techniques that account for the diverse information bottlenecks arising from user, measurement, and server heterogeneity, enabling the derivation of optimal lower bounds for heterogeneous configurations.

\subsection{Organization of the paper}

The rest of the paper is organized as follows. In Section~\ref{sec:main-results}, we present our main results, which summarize the fundamental performance limits of private functional mean estimation in the common and independent design settings. In Section~\ref{sec:methods}, we describe the methods for differentially private functional mean estimation in the common and independent design settings. In Section~\ref{sec:lower-bounds}, we present the lower bound for the common design setting. Proofs of the main results are provided in the Supplementary Material.

\subsection{Notation, definitions and assumptions}

Throughout the paper, we shall write $N := \sum_{s=1}^S m_s n_s$ and consider asymptotics in $S$, $m_s$, the $n_s$'s and the privacy budget $(\bm \epsilon, \bm \delta) := \{\varepsilon_s,\delta_s\}_{s=1}^S$. We assume that $N \to \infty$ and $\max_{s} \varepsilon_s = O(1)$. For two positive sequences $a_k$, $b_k$ we write $a_k\lesssim b_k$ if the inequality $a_k \leq Cb_k$ holds for some universal positive constant $C$. Similarly, we write $a_k\asymp b_k$ if $a_k\lesssim b_k$ and $b_k\lesssim a_k$ hold simultaneously and let $a_k\ll b_k$ denote that $a_k/b_k = o(1)$. 

We use $a\vee b$ and $a\wedge b$ for the maximum and minimum, respectively, between $a$ and $b$. For $k \in \N$, $[k]$ shall denote the set $\{1,\dots,k\}$. Throughout the paper $c$ and $C$ denote universal constants whose value can differ from line to line. The Euclidean norm of a vector $v \in \mathbb{R}^d$ is denoted by $\|v\|_2$. For a matrix $M \in \R^{d \times d}$, the norm $M \mapsto \| M \|$ is the spectral norm and $\text{Tr}(M)$ is its trace. Furthermore, we let $I_d$ denote the $d \times d$ identity matrix.

%Throughout this paper, we shall let $\alpha >1/2$, which is a required assumption for estimation in H\"older spaces (see e.g. \cite{Ibragimov1997}). 

\section{Main results}\label{sec:main-results}

In most practical settings, different servers often contain differing numbers of samples in terms of individuals, measurements per individual, and privacy parameters. We consider the most general setting where each server $s = 1,\dots,S$ can have a different number of individuals $n_s$, a different number of measurements $m_s$ for each individual in the server, and varying privacy parameters $\varepsilon_s$ and $\delta_s$. In the common design setting, our theory allows for heterogeneity in all of the aforementioned characteristics except for the design points.

Our main results are summarized in Theorem \ref{thm:rate-common-design} and \ref{thm:rate-independent-design} below, which capture the fundamental performance limits of private functional mean estimation in the common and independent design settings, respectively.

\subsection{Optimal performance when the design is common}\label{ssec:optimal-performance-common-design}

In this section, we consider the case of common design, where \( m_s = m \) for all \( s = 1, \dots, S \). To maintain simplicity of the theoretical statement here, we will assume the design to be a equispaced for the main result of this section; $\zeta_{j} = j/m$. We note that our upper bounds for the common design setting hold for more general design point setups (see Section \ref{ssec:common-design-upper-bound}).

Theorem \ref{thm:rate-common-design} below describes the minimax risk for differentially private estimators in the common design setting. The rate-determining quantity in this setting is governed by the parameter \( D_* \geq 1 \), which is the solution to the following equation:  
\begin{equation}\label{eq:rate-determining-relation:common-design}
	D^{2\alpha} =  m^{2\alpha} \bigwedge \sum_{s=1}^S \min \left\{ n_s, D^{-1} n_s^2 \varepsilon_s^2  \right\}.
\end{equation} 
The existence and uniqueness of the solution \( D_* > 0 \) to equation \eqref{eq:rate-determining-relation:common-design} can be established under standard assumptions on the parameters \( m, n_s, \varepsilon_s \), and \( \alpha \). Intuitively, \( D_* \) represents the balance between the design complexity, captured by \( m^{2\alpha} \), and the effective sample size across all samples. The left-hand side of the equation, \( D^{2\alpha} \), increases monotonically with \( D \), while the right-hand side combines terms that are bounded and decrease monotonically with \( D_* \), guaranteeing a unique solution.

\begin{theorem}\label{thm:rate-common-design}
	Suppose that for $\delta ' = \min_s \delta_s$, it holds that $\delta' \log(1/\delta') \lesssim \left(\frac{n_s}{m_s} \wedge \sqrt{\frac{n_s}{m_s}}\right)\varepsilon_s^2/N$. Let $D_*$ solve~\eqref{eq:rate-determining-relation:common-design} then,
	\begin{equation*}
		\inf_{\hat f \in \cM_{\varepsilon,\delta}} \sup_{f \in \cH^\alpha(R)} \E_f \|\hat f - f\|_2^2 \asymp \eta_N D^{-2\alpha}_* .
	\end{equation*}
	where $\eta_{N}$ is a polylogarithmic factor in $N = \sum_{s=1}^S n_s$ and $1/ \delta'$. 
\end{theorem}

It is instructive to consider two specific limiting cases to better understand the implications of this result.  
First, in the \textit{non-private case}, where the privacy constraints are removed by setting \( \varepsilon_s \to \infty \) for all \( s \in [S] \), the rate simplifies to the known minimax rates established in \cite{cai2011meanfunction}. This scenario corresponds to the classical non-private setting where privacy considerations do not impose any restrictions. The effective rate is determined solely by the design complexity and sample sizes, without the additional penalties introduced by privacy constraints.
Second, in the \textit{large-sample limit}, where \( n_s \to \infty \) for all \( s \in [S] \), the rate reduces to \( m^{-2\alpha} \). Here, the noise becomes negligible due to the abundance of samples, and the discretization error from the common design dominates. This reflects a regime where the design complexity \( m \) dictates the achievable performance, independent of the privacy constraints or noise levels.
In order to gain more fine grained insight it would be helpful to look at the homogeneous setting which we do section \ref{ssec:comparison-common-and-independent-design}. 

\subsection{Optimal performance when the design is independent}\label{ssec:optimal-performance-independent-design}

The convergence rate in the independent design setting is, just as in the common design setting, given by the solution to an equation. Let $D_* \geq 1$ be the largest number such that 
\begin{equation}\label{eq:rate-determining-relation:independent-design}
	D^{2\alpha}_* = \inf_{1 \leq D \leq D_*} \sum_{s=1}^S \min \left\{ D^{-1} n_s m_s, D^{-2}m_s n_s^2 \varepsilon_s^2, D^{2\alpha} n_s , D^{2\alpha - 1} n_s^2 \varepsilon_s^2  \right\}.  
\end{equation}
The rate of convergence for differentially private estimators in the independent design setting is given by the following theorem.

\begin{theorem}\label{thm:rate-independent-design}
	Suppose that for $\delta ' = \min_s \delta_s$, it holds that $\delta' \log(1/\delta') \lesssim \left(\frac{n_s}{m_s} \wedge \sqrt{\frac{n_s}{m_s}}\right)\varepsilon_s^2/N$. 
	Let $D_*$ solve \eqref{eq:rate-determining-relation:independent-design}. 
	
	Then,
	\begin{equation*}
		\inf_{\hat f \in \cM_{\varepsilon,\delta}} \sup_{f \in \cH^\alpha(R)} \E_f \|\hat f - f\|_2^2 \asymp \eta_N D^{-2\alpha}_*.
	\end{equation*}
	where $\eta_{N}$ is a polylogarithmic factor in $N = \sum_{s=1}^S n_s$ and $1/\delta'$. 
\end{theorem}

The optimization problem in \eqref{eq:rate-determining-relation:independent-design} can be interpreted 
minimum the four terms, which correspond to the sampling error, the measurement error, and their respective private counterparts. Our upper bounds rely on solving the optimization problem \eqref{eq:rate-determining-relation:independent-design}, which determines, roughly speaking, the effective dimension of the problem. The effective dimension \( D_* \) balances the trade-offs between design complexity, sample sizes, noise levels, and the smoothness of the underlying function class.

In the \textit{non-private case}, where the privacy constraints are removed by setting  \( \varepsilon_s \to \infty \) for all \( s \in [S] \), the rate simplifies to the classical non-private minimax rate. Notably, minimax optimal rates in the heterogeneous setup, where design points differ across samples, were not well understood prior to this work. The optimization problem in this case reveals how differences in sample sizes and design complexities across datasets interact with the smoothness of the underlying function to determine the rate.

\subsection{Comparison of the common and independent design settings}\label{ssec:comparison-common-and-independent-design}

In this section, we compare the minimax rates of convergence for the common and independent design settings under the homogeneous case, where \( n_s = n \), \( m_s = m \), \( \varepsilon_s = \varepsilon \), and \( \delta_s = \delta \). While in practice, settings are more often heterogeneous than not, the homogenous case allows for a more straightforward formulation minimax rate than the implicit ones of Theorems \ref{thm:rate-common-design} and \ref{thm:rate-independent-design}. In this simplified formulation, the rates between the independent and common designs are more easily compared. The comparison highlights the fundamental differences in statistical efficiency and the impact of differential privacy constraints across these two design settings.

We begin by stating the minimax risks for the homogeneous setup under the two design settings, starting with the independent design setting.

\begin{corollary}\label{cor:rate-independent-design}
	Suppose that it holds that $\delta \log(1/\delta) \lesssim \left(\frac{n}{m} \wedge \sqrt{\frac{n}{m}}\right)\varepsilon^2/N$. 

	It holds that
\begin{equation}\label{eq:rate-independent-design}
	\inf_{\hat f \in \cM_{\varepsilon,\delta}} \sup_{f \in \cH^\alpha(R)} \E_f \|\hat f - f\|_2^2 \asymp \eta_{N} \left( (Sn)^{-1} + (Smn)^{-\frac{2\alpha}{2\alpha + 1}} + (Smn^2 \varepsilon^2)^{-\frac{2\alpha}{2\alpha + 2}} +(Sn^2\varepsilon^2)^{-1} \right),
\end{equation}
where $\eta_{N}$ is a polylogarithmic factor in $N = S n$ and $1/\delta$.
\end{corollary} 

For the common design setting, the minimax rate of convergence in the homogeneous setup is given by the following corollary.

\begin{corollary}\label{cor:rate-common-design}
	Suppose that it holds that $\delta \log(1/\delta) \lesssim \left(\frac{n}{m} \wedge \sqrt{\frac{n}{m}}\right)\varepsilon^2/N$. 

	It holds that
	\begin{equation}\label{eq:rate-common-design}
		\inf_{\hat f \in \cM_{\varepsilon,\delta}} \sup_{f \in \cH^\alpha(R)} \E_f \|\hat f - f\|_2^2 \asymp \eta_{N}\left( (Sn)^{-1}  + m^{-2\alpha} + (Sn^2 \varepsilon^2)^{-\frac{2\alpha}{2\alpha + 1}} \right),
	\end{equation}
	where $\eta_{N}$ is a polylogarithmic factor in $N = S n$ and $1/\delta$.
\end{corollary}

In the common design case, the rate of convergence in \eqref{eq:rate-common-design} is dominated by the privacy-related term \( (Sn^2\varepsilon^2)^{-\frac{2\alpha}{2\alpha +1 }} \) when \( \varepsilon \lesssim \min\left( S^{-1/2}n^{-1/2} m^{\alpha +1/2}, S^{1/4\alpha}n^{(1-2\alpha)/4\alpha}\right) \). In contrast, under the independent design setting, the rate is given by the combination \( (Smn^2 \varepsilon^2)^{-\frac{2\alpha}{2\alpha + 2}} + (Sn^2\varepsilon^2)^{-1} \) when privacy constraints dominate. 

If the rate is dominated by the second term $(Sn^2\varepsilon^2)^{-1}$, it is strictly faster than the common design privacy rate. If the first term $(Smn^2 \varepsilon^2)^{-\frac{2\alpha}{2\alpha + 2}} $ dominates, it is faster than the common design rate whenever
$$
(Smn^2 \varepsilon^2)^{-\frac{2\alpha}{2\alpha + 2}} \lesssim (Sn^2 \varepsilon^2)^{-\frac{2\alpha}{2\alpha + 1}} \quad \text{iff} \quad \varepsilon\lesssim  S^{-1/2}n^{-1/2} m^{\alpha +1/2}. 
$$
Thus, whenever the ``privacy-related" terms dominate, an independent design achieves a faster rate of convergence, meaning the cost of privacy is consistently lower in this setting.

What is particularly noteworthy is that under the independent design, when privacy constraints are binding, there are benefits to having both more individuals and more measurements per individual. In contrast, under the common design, increasing the number of measurements per individual offers no benefit when the privacy constraint is binding. This can be interpreted as a privacy-related penalty for the common design. This might seem counterintuitive at first, since in the common design setting, the design points do not need to be private, whereas in the independent design setting, they do. The key insight here is about information leakage: In the common design where all individuals share the same measurement points, adding more measurements doesn't help because the privacy cost is already maximized - everyone's data is perfectly correlated through the shared design points. In contrast, with independent designs, while more measurements do leak more information about each individual (requiring more noise for privacy), this is offset by the statistical benefits of having more diverse measurement points across the population.

\section{Methods for Differentially Private Functional Mean Estimation}\label{sec:methods}

We now describe the methods for differentially private functional mean estimation in the independent and common design settings. As revealed by Theorems \ref{thm:rate-common-design} and \ref{thm:rate-independent-design}, these settings exhibit distinct principal phenomena and require different approaches. We present an optimal algorithm for each setting, both of which provably achieve the minimax rate of convergence. The theorems accompanying the algorithms in this section provide the upper bounds for the respective theorems mentioned earlier, starting with the independent design case.

\subsection{An optimal algorithm for the independent design setting}\label{ssec:algorithm-independent-design}

In the independent design setting, we consider a scenario where all design points are distinct and treated as private information. The challenge in this setup, compared to e.g. private nonparametric regression as studied in \cite{cai2024optimal}, stems from the fact that the model itself is effectively subject to three different sources of noise: the per-subject $X_i^{(s)}$-function noise, the noise stemming from the randomness in the design points, and the measurement noise. An optimal privacy mechanism is calibrated to account for all three sources of noise, ensuring that the privacy budget for each server is allocated efficiently across the different components. This means that the procedure must differ from server to server, depending on the number of individuals, measurements per individual, and privacy parameters.

Our approach is based on a private, truncated projection estimator, where we project onto compactly supported orthonormal basis functions. This approach allows effective privatization of both the design points and the outcome values. As basis functions, we use $A>\alpha$-regular wavelets, which form an orthonormal basis of $L_2[0,1]$. Wavelet bases allow characterization of H\"older spaces, with $\alpha$ capturing the decay of wavelet coefficients. For further details, see e.g. \cite{hardle2012wavelets}. 

For any $A \in\mathbb{N}$ one can follow Daubechies' construction of the father $\phi(\cdot)$ and mother $\psi(\cdot)$ wavelets with $A$ vanishing moments and bounded support on $[0,2A-1]$ and $[-A+1,A]$, respectively, for which we refer to \cite{daubechies1992ten}. The basis functions are then obtained as
\begin{align*}
	\big\{ \phi_{l_0+1,m},\psi_{lk}:\, m\in\{0,...,2^{l_0+1}-1\},\quad l \geq l_0 +1,\quad k\in\{0,...,2^{l}-1\} \big\},
\end{align*}
with $\psi_{lk}(x)=2^{l/2}\psi(2^lx-k)$, for $k\in [A-1,2^l-A]$, and $\phi_{l_0+1, k}(x)=2^{l_0+1}\phi(2^{l_0 + 1}x-m)$, for $m\in [0,2^{l_0 + 1}-2A]$, while for other values of $k$ and $m$, the functions are specially constructed, to form a basis with the required smoothness property. A function \(f \in L_2[0,1]\) can then be expressed as
\[
f = \sum_{l=l_0}^\infty \sum_{k=0}^{2^l-1} f_{lk} \psi_{lk},
\]
where $f_{lk} = \int f \psi_{lk}$. A well known estimator for the wavelet coefficients \(f_{lk}\) -- in the absence of privacy constraints -- is given by
\[
\widehat{f}_{lk}^{(s)} = \frac{1}{n_s} \sum_{i=1}^{n_s} \frac{1}{m_s} \sum_{j=1}^{m_s} Y_{ij}^{(s)} \psi_{lk}\left(\zeta_{ij}^{(s)}\right),
\]
where \(Y_{ij}^{(s)}\) are the observations and \(\zeta_{ij}^{(s)}\) are the corresponding design points. 

For notational simplicity, let us define the contribution of a single individual \(i \in[n_s]\), on server $s\in[S]$ to the estimator as:
\[
U_{i,lk}^{(s)} = \frac{1}{m_s} \sum_{j=1}^{m_s} Y_{ij}^{(s)} \psi_{lk}(\zeta_{ij}^{(s)}).
\]
Thus, the standard (unclipped) estimator can be rewritten as:
\(
\widehat{f}_{lk}^{(s)} = \frac{1}{n_s} \sum_{i=1}^{n_s} U_{i,lk}^{(s)}.
\)
In order to make an informed decision on how to clip requires sharp concentration bounds on \(U_{i,lk}^{(s)}\). The following lemma establishes the concentration of \(U_{i,lk}^{(s)}\) for a fixed \(l\) and \(k\), using Bernstein's inequality.

\begin{lemma}\label{lemma:concentration-of-U}
	Fix $c>0$, \(l\) and \(k\), the sequence \(\{U_{i,lk}^{(s)}\}_{i \in [n_s]}\) satisfies:
	\[
	\P\left(\forall i, \, |U_{i,lk}^{(s)}| < \tau_l^{(s)}\right) \geq 1 - \gamma,
	\]
	with the threshold \(\tau_l^{(s)}\) given by
	\begin{equation}\label{eq:def-clipping}
	\tau_l^{(s)} = 2 (2c\log N)^{3/2}\left(\sqrt{\frac{1}{m_s}} + \frac{1}{3}\|\psi\|_\infty \frac{2^{l/2}}{m_s}\right)+ R2^{-l(\alpha+1/2)},
	\end{equation}
	and \(\gamma = N^{-c}\), where the constant $c > 0$  can be arbitrary.
\end{lemma}

A proof of the lemma is given in Section \ref{sssec:proof-of-independent-design-upper-bound} of the Supplementary Material. We define what is effectively the local clipped estimator for the $l,k$-th wavelet coefficient as
\begin{equation}
	\label{eq:fhat_lk-independent-design}
\widehat{f}_{lk}^{\tau,s} = \frac{1}{n_s} \sum_{i=1}^{n_s} \left[ U_{i,lk}^{(s)} \right]_{\tau_l^{(s)}},
\end{equation}
where \(\left[ U_{i,lk}^{(s)} \right]_{\tau_l^{(s)}}\) denotes the clipping of \(U_{i,lk}^{(s)}\) to the range \([- \tau_l^{(s)}, \tau_l^{(s)}]\). Clipping at this magnitude ensures that the estimator has bounded sensitivity, enabling it to be privatized effectively.

To achieve differential privacy, we aim to employ the Gaussian mechanism. However, adding Gaussian noise to the clipped averages of \eqref{eq:fhat_lk-independent-design} directly does not balance the amount of signal versus sensitivity present in each wavelet coefficient. Instead, we employ a modified Gaussian mechanism, akin to ideas used in anisotropic private mean estimation of e.g. \cite{dagan2024dimensionfreeprivatemeanestimation}. 

The coordinates of \(\widehat{f}_{lk}^\tau\) exhibit varying greatly varying sensitivities across the various resolution levels, proportional to \(\tau_l^{(s)}\). In order to obtain the right signal-sensitivity trade-off, we rescale each coordinate, yielding at the \(l\)-th resolution level:
\[
\tilde{f}_{lk}^{\tau,s} = (\tau_l^{(s)} \sqrt{2^l \wedge m_s})^{-1} \widehat{f}_{lk}^{\tau,s}.
\]
The \(\ell_2\) sensitivity of the vector \( \{\tilde{f}_{lk}^{\tau,s}, k = 0, \ldots, 2^l-1, l_0 \leq l \leq L\}\) is bounded as shown in Lemma \ref{lemma:L2-sensitivity-independent-design}.

\begin{lemma}\label{lemma:L2-sensitivity-independent-design}
	Let \(Z^{(j)}\) and \(\tilde{Z}^{(j)}\) be any neighboring datasets. Let $\tilde{f}^{\tau,s}(Z^{(j)}), \tilde{f}^{\tau,s}(\tilde{Z}^{(j)})$ denote the vectors of coefficient estimators \(\{\tilde{f}_{lk}^{\tau,s}, k = 0, \ldots, 2^l-1, l_0 \leq l \leq L\}\) computed on the basis of \(Z^{(j)}\) and \(\tilde{Z}^{(j)}\) respectively. Then:
	\[
	\left\|\tilde{f}^{\tau,s}(Z^{(j)}) - \tilde{f}^{\tau,s}(\tilde{Z}^{(j)})\right\|_2 \leq c_A \frac{\sqrt{L}}{n_s},
	\]
	where $c_A$ is a computable constant depending only the choice of wavelet basis.
\end{lemma}

Using the Gaussian mechanism (for a details, see Appendix A of \cite{dwork2014algorithmic}), we obtain a \((\varepsilon_s, \delta_s)\)-differentially private estimator
\[
\tilde{f}_{lk}^{P,s} = \tilde{f}_{lk}^{\tau,s} + \widetilde{W}_{lk}^{(s)}, \quad \widetilde{W}_{lk}^{(s)} \overset{i.i.d.}{\sim} \mathcal{N}\left(0, \frac{4 c_A^2 L \log(2/\delta)}{n_s^2 \varepsilon_s^2}\right).
\]

To construct the final private estimator \(\widehat{f}_{lk}^{P,s}\) in the central server, we rescale \(\tilde{f}_{lk}^{P,s}\) back to the original scale, yielding
\[
\widehat{f}_{lk}^{P,s} = \widehat{f}_{lk}^{\tau,s} + W_{lk}^{(s)}, \quad W_{lk}^{(s)} \overset{i.i.d.}{\sim} \mathcal{N}\left(0, \frac{4 c_A^2 L (2^l \wedge m_s) (\tau_l^{(s)})^2 \log(2/\delta_s)}{n_s^2 \varepsilon_s^2}\right).
\]
Each server \(s\) then outputs a transcript
\[
T_L^{(s)} = \{\widehat{f}_{lk}^{P,s}: k = 0, \ldots, 2^l-1, l = l_0, \ldots, L\}.
\]

The final estimator \(\widehat{f}^P\) is obtained via a post-processing step, where the transcripts are aggregated using weights that account for the heterogeneity of the servers. These weights depend on the resolution level \(l\), the number of local observations \(n_s\), the number of design points \(m_s\), the local privacy constraints \(\varepsilon_s\), and the smoothness level \(\alpha\). Specifically, the final estimator is:
\begin{equation}\label{eq:federated-independent-design-estimator}
	\widehat{f}^{P}_L = \sum_{l=l_0}^L\sum_{k=0}^{2^l -1} \widehat{f}^{P} _{lk}\psi_{lk} \quad \text{where} \quad \widehat{f}^{P} _{lk} = \sum_{s=1}^S  w^s_l \widehat{f}^{P,s} _{lk}
\end{equation}
with \(\sum_{s=1}^S w_l^{(s)} = 1\) for all \(l \geq l_0\). The weights \(w_l^{(s)}\) are chosen to be inversely proportional to the variance of \(\widehat{f}_{lk}^{P,s}\):
\begin{equation}\label{eq:choice-of-weights}
	w_l^{(s)} = \frac{u_l^{(s)}}{\sum_{s=1}^S u_l^{(s)}}, \quad \text{where} \quad u_l^{(s)} = n_s m_s \wedge m_s 2^{l(2\alpha+1)} \wedge 2^{-l} m_s n_s^2 \varepsilon_s^2 \wedge 2^{2l\alpha} n_s^2 \varepsilon_s^2.
\end{equation}
\begin{remark}
	The weights \(w_l^{(s)}\) account for the heterogeneity in data distribution across servers, including variations in the number of observations (\(n_s\)), the number of design points (\(m_s\)), and the privacy budgets (\(\varepsilon_s\)). By being inversely proportional to the variance of \(\widehat{f}_{lk}^{P,s}\), the weights prioritize servers that contribute more reliable estimates. The structure of \(u_l^{(s)}\) incorporates terms that reflect the impact of resolution level \(l\), smoothness parameter \(\alpha\), and privacy constraints.
\end{remark}

\begin{algorithm}[h]
	\caption{Private Estimation via Truncated Projection in Independent Design}
	\begin{algorithmic}[1]\label{alg:independent-design-estimator}
		\STATE \textbf{Input:} Observations \( Y_{ij}^{(s)} \), design points \( \zeta^{(s)}_{ij} \), wavelet basis \(\psi_{lk}\), truncation level \( L \), privacy budgets \(\varepsilon_s, \delta_s\)
		\STATE \textbf{Output:} Final estimator \( \widehat{f}^P_L \)
		\STATE \textbf{Clipped Coefficients:} Each server computes:
		\[
		\widehat{f}^{\tau,s}_{lk} = \frac{1}{n_s} \sum_{i=1}^{n_s} \left[ \frac{1}{m_s} \sum_{j=1}^{m_s} Y_{ij}^{(s)} \psi_{lk}(\zeta^{(s)}_{ij}) \right]_{\tau_l^{(s)}} \,\forall \,l=l_0,\ldots,L,\, k =0,\ldots,2^l -1 ,
		\]
		where $\tau_l^{(s)}$ is given by~\eqref{eq:def-clipping}.
		\STATE \textbf{Privatized Coefficients:} Apply the modified Gaussian mechanism:
		\[
		\widehat{f}^{P,s}_{lk} = \widehat{f}^{\tau,s}_{lk} + W^{(s)}_{lk}, \quad W^{(s)}_{lk} \overset{i.i.d.}{\sim} N\left(0, \frac{4 c_A^2L(2^l \wedge m) (\tau^{(s)}_l)^2 \log(2/\delta)}{n_s^2 \varepsilon_s^2} \right).
		\]
		\STATE \textbf{Weighted Aggregation:} Aggregate across servers:
		\[
		\widehat{f}^P_{lk} = \sum_{s=1}^S w_l^{(s)} \widehat{f}^{P,s}_{lk}, \quad w_l^{(s)} = \frac{u_l^{(s)}}{\sum_{s=1}^S u_l^{(s)}}, \quad u_l^{(s)} = n_s m_s \wedge m_s 2^{l(2\alpha+1)} \wedge 2^{-l}m_s n_s^2 \varepsilon_s^2 \wedge 2^{2l\alpha} n_s^2 \varepsilon_s^2.
		\]
		\STATE \textbf{Final Estimator:} Combine wavelet coefficients:
		\[
		\widehat{f}^P_L(x) = \sum_{l=l_0}^L \sum_{k=0}^{2^l-1} \widehat{f}^P_{lk} \psi_{lk}(x).
		\]
	\end{algorithmic}
\end{algorithm}

We summarize our final estimator in Algorithm \ref{alg:independent-design-estimator}.  
 
\begin{theorem}\label{thm:independent-design-upper-bound}
	Let \(\widehat{f}^P_L\) be the estimator defined in Algorithm \ref{alg:independent-design-estimator}.  Let $D^*$  satisfy~\eqref{eq:rate-determining-relation:independent-design} and let $L = [\log_2D^*]$ we obtain the following upper bound on the minimax rate,
	\begin{equation*}%\label{eq:rate-independent-design}
		\sup_{f \in \cH^\alpha(R)} \E_f \|\widehat f^P_L - f\|_2^2 \lesssim  (\log N)^5 \log(2/\delta') (D^*)^{-2\alpha}.
	\end{equation*}
	where $\delta' = \min_s \delta_s$.
\end{theorem}
\begin{remark}
	The parameter \(D^*\), which satisfies the rate-determining relation~\eqref{eq:rate-determining-relation:independent-design}, can be interpreted as the effective dimension of the problem. It reflects the interplay between the resolution level \(L\) and the sample size \(N = \sum_{s=1}^S n_s\). Larger \(D^*\) corresponds to higher resolution levels and greater flexibility in representing the function \(f\), while also increasing the complexity of the estimation task. The optimal choice of \(D^*\) balances this trade-off to achieve the minimax rate.
\end{remark}

\begin{remark}
	In the absence of privacy constraints (\(\varepsilon_s \to \infty\)), the result in Theorem~\ref{thm:independent-design-upper-bound} obtains the minimax rates for the setting where different individuals may have different numbers of design points \(m_s\). To the best of our knowledge, these rates were not previously established in the literature. 
\end{remark}

\subsection{An optimal algorithm for the common design setting}\label{ssec:common-design-upper-bound}

In the common design setup, the design points \( \zeta_1, \zeta_2, \dots, \zeta_m \) are shared among all individuals and servers. These points are assumed to be publicly known, requiring no privatization. Here, the challenge lies in the fact that the design points are shared across all servers, which means that an optimal private estimator must balance the discretization error due to the shared design points with the privacy leakage.

Our estimation strategy follows two main steps:

\begin{enumerate}
	\item First, we obtain privatized averages of observations at each design point, effectively creating private, noisy estimates of the function at these specific points.
	\item Second, if the number of design points is large, we introduce a bagging step that divides the design points into groups.
	\item Third, these privatized averages present an inverse problem at the central server, which we solve using local polynomial regression. This regression effectively interpolates across the design points.
\end{enumerate}

Private averages at the design points of first step effectively pose an inverse problem at the central server; solved by local polynomial regression in the third step. The noise in this inverse problem at the central server is does not only depend on measurement noise, but also on the noise added by the privacy mechanism. If the number of design points is large, the noise added by the privacy mechanism needs to increase to avoid the privacy leakage due to revealing more measurements per individual. 

The second step employing bagging mitigates this by dividing the design points into groups and computes smoothed estimators that interpolate across groups. Intuitively, the purpose of the bagging step is used to effectively reduce the number of measurements $m$, in the scenario where the measurements reveal too much information about the individuals if they were all communicated to the central server. In terms of attaining the minimax rate, the bagging step is superfluous, in the sense that one could also effectively use a smaller number of measurements by taking a selection of design points. Bagging, however, is a more practical approach, as it a straightforward implementation that is more robust to the choice of design points.

Next, we describe the algorithm for the common design formally.

\paragraph{Privatized averages at the design points}

Given \( S \) servers and a set of observations from a function \( f \) at deterministic points \( \zeta_1, \zeta_2, \dots, \zeta_m \), the privatized mean \(\bar{Y}_{j}^{P,s}\) at design point \( \zeta_{j} \) on server \( s \) is computed as:
\[
\bar{Y}_{j}^{P,s} = \frac{1}{n_s} \sum_{i=1}^{n_s} \left[ Y^{(s)}_{i,j} \right]_{\tau} + W^{(s)}_j,
\]
where \( \left[ Y^{(s)}_{i,j} \right]_{\tau} \) represents the clipped observations with clipping threshold \(\tau = \sqrt{2\log N} + C_{\alpha,R}\) ($C_{\alpha,R}$ is as defined in Lemma \ref{lem:holder-norm-sup-norm}), and \( W^{(s)}_j \sim \mathcal{N}\left( 0, \frac{4 \tau^2 m \log (2/\delta_s)}{n_s^2 \varepsilon_s^2} \right) \) is a Gaussian noise term ensuring differential privacy. Here, \( N \) is the total sample size across all servers, \( n_s \) is the sample size on server \( s \), and \(\varepsilon_s\) is the privacy budget for server \( s \).

Each server outputs a transcript \( T^{(s)} = (\bar{Y}_{j}^{P,s})_{j \in [m]}\). The central server aggregates these transcripts using inverse variance weighting:
\[
\bar{Y}_{j}^{P} = \sum_{s=1}^S w_s \bar{Y}_{j}^{P,s},
\]
where the weights \( w_s \) are designed to account for server heterogeneity:
\[
w_s = \frac{u_s}{\sum_{s=1}^S u_s}, \quad u_s = D^{-1} n_s^2 \varepsilon_s^2 \wedge n_s.
\]
where $D$ solves~\eqref{eq:rate-determining-relation:common-design}.

\paragraph{Bagging}

The design points \( \zeta_1, \zeta_2, \dots, \zeta_m \) are ordered and divided into \( B \) disjoint sets \(\{G_b\}_{b \in [B]}\), each containing \( m_0 \) points, such that $\lfloor m/m_0 \rfloor \leq B \leq \lceil m/m_0 \rceil$. We choose groups $(G_b)_{b \in [B]}$ of sizes $m_0 = [D]$ where $D$ solves~\eqref{eq:rate-determining-relation:common-design}. In case of equispaced design (i.e $\zeta_j = j/m$), this amounts to $G_b = \left\{\frac{B(j-1)+b}{m}\,:\, j=1,\ldots,m_0\right\} $. %The points within each group are selected to be uniformly spread over the domain.

% and set the bandwidth $h$ to be $1/m_0$

For each group \( G_b \), a local polynomial estimator \( \widehat{f}^{P}_b(x) \) is computed based on the privatized means \(\bar{Y}_{j}^{P}\) corresponding to the design points in \( G_b \). Finally, the overall estimator \( \widehat{f}^{P}_B(x) \) is obtained by averaging across all groups:
\[
\widehat{f}^{P}_B(x) = \frac{1}{B} \sum_{b=1}^B \widehat{f}^{P}_b(x).
\]

\paragraph{Function approximation via local polynomials}

The local polynomial estimator approximates the function \( f \in \mathcal{H}^\alpha(R) \) by expanding it around \( x \) as:
\[
f(z) = f(x) + f'(x)(z - x) + \dots + \frac{f^{(p)}(x)}{p!} (z - x)^p,
\]
where \( p = \lfloor \alpha \rfloor \). This expansion can be written compactly as:
\[
f(z) = \Theta^\top(x) V\left( \frac{z - x}{h} \right),
\]
where \( V(u) = \begin{pmatrix} 1, u, \frac{u^2}{2!}, \dots, \frac{u^p}{p!} \end{pmatrix}^\top \) is the vector of scaled monomials, and \( \Theta(x) = \begin{pmatrix} f(x), f'(x)h, \dots, f^{(p)}(x) h^p \end{pmatrix}^\top \) represents the function's coefficients at \( x \), scaled by powers of the bandwidth \( h \).

To estimate \( f(x) \), we solve a weighted least squares problem for each group:
\[
\widehat{\Theta}_b(x) = \arg\min_{\Theta \in \mathbb{R}^{l+1}} \sum_{j \in G_b} \left[ \bar{Y}_{j}^{P} - \Theta^\top V\left( \frac{\zeta_{j} - x}{h} \right) \right]^2 K\left( \frac{\zeta_{j} - x}{h} \right).
\]
where $K$ can be any smooth enough kernel function satisfying
\[
\operatorname{supp}(K) \subseteq [-1,1], \quad \|K\|_\infty < \infty, \quad \text{and} \quad \int K(u) \, du = 1.
\]

The first component of \( \widehat{\Theta}_b(x) \), corresponding to \( \hat{f}^{P}_b(x) \), is the estimate of \( f(x) \) for group \( b\). 

The local polynomial estimator \( \hat{f}^{P}_b(x) \) for the \( b\)-th group can be shown to be equal to:
\[
\hat{f}_b(x) = \sum_{j \in G_b} \bar{Y}_{j}^p W_{b,j}^*(x),
\]
where the weights \( W_{b,j}^*(x) \) are determined by the kernel and the design points of the $b$th group:
\begin{equation}\label{eq:LP-weights}
	W_{b,j}^*(x) = \frac{1}{m_0 h} V^\top(0) B_{b,x}^{-1} V\left( \frac{\zeta_{j} - x}{h} \right) K\left( \frac{\zeta_{j} - x}{h} \right).
\end{equation}
The matrix \( B_{b,x} \) is the ``effective'' weighted design matrix for group \( b \) induced by $K$, given by
\[
B_{b,x} = \frac{1}{m_0 h} \sum_{j \in G_b} V\left( \frac{\zeta_{j} - x}{h} \right) V\left( \frac{\zeta_{j} - x}{h} \right)^\top K\left( \frac{\zeta_{j} - x}{h} \right).
\]

For the subsequent analysis, we rely on the following assumptions, where the choice of kernel and groups allows for additional flexibility. We remark that both assumptions (LP1) and (LP2) hold in particular for regular equi-spaced design under mild conditions on the Kernel $K$. See for example Lemma 1.5 in \cite{tsybakov_introduction_2009}.

\begin{itemize}
	\item [(LP1)] \( \lambda_{\min}(B_{b,x}) \geq \lambda_0 \) for all \( b \in [B] \), and \( x \in [0,1] \). 
	
	\item [(LP2)] For $b\in[B]$ , for some constant \( a \), and any measurable set \( A \), we have
	\[
	\frac{1}{m_0} \sum_{j\in G_b} \mathbb{I}(\zeta_{j} \in A) \leq a \max\left( \text{Leb}(A), \frac{1}{m_0} \right).
	\]
\end{itemize}

Assumption (LP1) ensures that the local design points in each group provide sufficient information for local polynomial estimation. Together with (LP2), these assumption ensures that within each group the design points \( (\zeta_{j})_{j\in G_b} \) are sufficiently spread over the interval \([0,1]\). %In other words when dividing into groups \(\{G_b\}_{b \in [B]}\), each subset inherits this uniform spread property.

\begin{algorithm}[h]
	\caption{Private Estimation via Smoothing and Bagging for Common Design}
	\begin{algorithmic}[1]\label{alg:common-design-estimator}
		\STATE \textbf{Input:} Design points \( T_1, \dots, T_m \), \( S \) servers, observations \( Y_{i,j}^{(s)} \), privacy budgets \(\varepsilon_s\), groups \(\{G_b\}_{b \in [B]}\) each of size $m_0$ , bandwidth \( h \), kernel \( K \)
		\STATE \textbf{Output:} Final estimator \( \hat{f}^P_B(x) \)
		\STATE \textbf{Privatized Means:} Each server computes privatized means:
		\[
		\bar{Y}_{j}^{P,s} = \frac{1}{n_s} \sum_{i=1}^{n_s} \left[ Y_{i,j}^{(s)} \right]_{\tau} + W_j^{(s)}, \quad W_j^{(s)} \sim \mathcal{N}\left( 0, \frac{4 \tau^2 m \log (2/\delta_s)}{n_s^2 \varepsilon_s^2} \right),
		\]
		where $\tau = \sqrt{2 \log N}+ C_{\alpha,R}$.
		\STATE \textbf{Aggregate Means:} Central server aggregates:
		\[
		\bar{Y}_j^P = \sum_{s=1}^S w_s \bar{Y}_{j}^{P,s}, \quad w_s = \frac{u_s}{\sum_{s=1}^S u_s}, \quad u_s = m_0^{-1} n_s^2 \varepsilon_s^2 \wedge n_s.
		\]
		
		\STATE \textbf{Group-wise Estimation:} For each group \( G_b \), compute local polynomial estimator:
		\[
		\widehat{f}_b^P(x) = \sum_{j \in G_b} \bar{Y}_j^P W_{b,j}^*(x),
		\]
		where \( W_{b,j}^*(x) \) are kernel-based weights given by~\eqref{eq:LP-weights}.
		\STATE \textbf{Final Estimator:} Aggregate across groups:
		\begin{equation}\label{eq:final-estimator-bagged}
			\widehat{f}^P_B(x) = \frac{1}{B} \sum_{b=1}^B \widehat{f}_b^P(x).
		\end{equation}
	\end{algorithmic}
\end{algorithm}

Algorithm \ref{alg:common-design-estimator} summarizes the differentially private protocol. By averaging across all groups in \eqref{eq:final-estimator-bagged}, the bagging estimator as \( \hat{f}^{P}_B(x) \) is computed in the central server. The following theorem establishes the convergence rate of this estimator.

\begin{theorem}\label{thm:common-design-upper-bound}
	Let \(\widehat{f}^P_B\) be the estimator defined in Algorithm \ref{alg:common-design-estimator}. We choose groups $(G_b)_{b \in [B]}$ of sizes $m_0 = D$ where $D$ solves~\eqref{eq:rate-determining-relation:common-design} and set the bandwidth $h$ to be $1/m_0$ . We also assume that for each of the groups $(G_b)_{b \in [B]} $, the design points in the groups $(\zeta_{j})_{j \in G_b} $ satisfy (LP1) and (LP2).
	Then,
	\begin{equation*}%\label{eq:rate-independent-design}
		\sup_{f \in \cH^\alpha(R)} \E_f \|\hat f - f\|_2^2 \lesssim\log(N)\log(2/\delta') \left(D^{-2\alpha}\vee m^{-2\alpha}\right),
	\end{equation*}
	where $D$ solves~\eqref{eq:rate-determining-relation:common-design} $\delta' = \min_s \delta_s$.
\end{theorem}
\begin{remark}
	The same minimax risk as in Theorem~\ref{thm:common-design-upper-bound} can be achieved by using data from only one of the groups of size \(m_0\) instead of aggregating across all \(B\) groups. This is because, when using a single group, the amount of noise added to ensure differential privacy scales with \(m_0\), which is smaller than or equal to \(m\). Consequently, the effective privacy noise is reduced, and the overall estimation error remains comparable. While using all \(B\) groups can improve practical stability/performance by averaging across groups, it does not offer a theoretical advantage in terms of the minimax rate.
\end{remark}

%\subsection{Comparison of the two designs}
%
%In case of common design, the rate of convergence in \eqref{eq:rate-common-design} is dominated by the term $(n^2\varepsilon^2)^{-\frac{2\alpha}{2\alpha + 1}}$ when $\varepsilon \lesssim \min\left(  n^{\frac{1 - 2\alpha}{4\alpha}}, \ n^{-1} m^{\alpha+1/2} \right)$. 
%
%In the case of independent design, when $\varepsilon \lesssim \min\left( m^{\frac{1}{4\alpha + 2}} n^{-\frac{\alpha}{2\alpha + 1}}, \ n^{\frac{1 - \alpha}{2 \alpha}} m^{-1/2} \right)$, $(mn^2\varepsilon^2)^{-\frac{2\alpha}{2\alpha + 2}}$ becomes the dominating term in \eqref{eq:rate-independent-design}.
%
%Comparison of these two terms, we see that
%\begin{equation*}
%	(mn^2\varepsilon^2)^{-\frac{2\alpha}{2\alpha + 2}} \lesssim (n^2\varepsilon^2)^{-\frac{2\alpha}{2\alpha + 1}} \quad \iff \quad \varepsilon \gtrsim m^{\alpha + 1/2} n^{-1}. 
%\end{equation*}
%That is, whenever the ``privacy-related'' terms dominate, the rate of convergence is faster in the independent design case / the cost of privacy is always lower in the independent design case. What is also remarkable, is that the rate of convergence in the independent design case, whenever the privacy constraint is binding, there is benefit of having more individuals but also more measurements per individual. In case of common design, there is no benefit of having more measurements per individual when the privacy constraint is binding.
%

\section{Lower bounds}\label{sec:lower-bounds}

As a complement to the upper bounds, we provide lower bounds, which together establish the minimax rates for the federated learning problem; resulting in Theorems  \ref{thm:rate-common-design} and \ref{thm:rate-independent-design}. We first present the lower bound for the independent design setting, followed by the lower bound for the common design setting.

\subsection{Lower bound for independent design}\label{ssec:lower-bound-independent-design}

We first present the lower bound, before discussing the proof. 

\begin{theorem}\label{thm:lower-bound-independent-design}
	Suppose that for $\delta ' = \min_s \delta_s$, it holds that $\delta' \log(1/\delta') \lesssim \left(\frac{n_s}{m_s} \wedge \sqrt{\frac{n_s}{m_s}}\right)\varepsilon_s^2/N$. 
	Let $D_*$ solve \eqref{eq:rate-determining-relation:independent-design}.  Then,
	\begin{equation}\label{eq:rate-lower-bound-independent-design}
		\inf_{\hat f \in \cM_{\varepsilon,\delta}} \sup_{f \in \cH^\alpha(R)} \E_f \|\hat f - f\|_2^2 \gtrsim D^{-2\alpha}_*.
	\end{equation}
\end{theorem}

Below, we briefly sketch the proof of Theorem \ref{thm:lower-bound-independent-design}. The challenge in proving the lower bound stems from the heterogeneity of the servers. Essentially, for each respective server, either the sample size, the number of design points, or the privacy constraints can form the bottleneck. 

The proof involves a careful construction of a data generating process that allows to capture these different bottlenecks for each server. From a technical standpoint, the heterogeneous nature of the servers is captured by considering the Fisher information of the model in an appropriate finite dimensional sub-model. The Fisher information tensorizes along the servers, which allows us provide separate information theoretic arguments for each server.

To provide a more detailed description of the argument, we introduce the aformentioned notions formally. Consider the $2^L$-dimensional sub-model given by
	\begin{equation}\label{eq:holder-submodel}
	\cH^\alpha_{R,L} : \left\{f \in \cH_\alpha^{R} \,:\, f =\sum_{k=1}^{2^L}f_k \phi_{Lk},\, f_k \in [-2^{-L (\alpha +1/2) - 1}R, 2^{-L (\alpha +1/2) - 1}R]  \right\},
	\end{equation}
	where $\phi_k$ are the wavelet basis functions at resolution level $L$.  Let $\mu$ and $\nu$ denote dominating measures for $\P_f^{Y^{(s)}}$ and $\P_f^{(Y^{(s)}X^{(s)})}$, respectively, for $f \in \cH^\alpha_{R,L}$.
	Define $f_L = (f_k)_{k=1}^{2^L}$ to be the vector of coefficients of $f$ in the wavelet basis. The transcript $T^{(s)}$ induces a ``Fisher information'', which can be computed to be equal
	\begin{equation}\label{eq:fisher-information-Ys}
		I^{Y^{s}|T^{(s)}}_f = \E_f \, \E_f\left[ S_f^{Y^{(s)}} \bigg| T^{(s)} \right] \E_f\left[ S_f^{Y^{(s)}} \bigg| T^{(s)} \right]^\top
	\end{equation}
	where $S_f^{Y^{(s)}}$ is given by 
	\begin{equation}\label{eq:effective-score-Ys}
		S_f^{Y^{(s)}_i} = \sum_{j=1}^m \left(Y_{ij}^{(s)} - \sum_{k=1}^{2^L}f_k \phi_l(\zeta^{(s)}_{ij})\right) \phi(\zeta^{(s)}_{ij}),
	\end{equation}
	which can be seen as a ``score function'' of the observational model for $Y^{(s)}$ in the sub-model $\cH^\alpha_{R,L}$, if $X_i = \E_f X_i$ almost everywhere. Let $I_f^{Y^{(s)}}$ denote the Fisher information of the observational model for $Y^{(s)}$ in the sub-model $\cH^\alpha_{R,L}$, which equals $\E_f S_f^{Y^{(s)}} {S_f^{Y^{(s)}}}^\top$.

	When $m_s$ tends to infinity, the observational model for $Y^{(s)}$ effectively contains close to the amount of information as the observational model in which $X^{(s)}$ is directly observed. Following this intuition, the right information quantity to consider for servers with large $m_s$ is the Fisher information of the observational model for $X^{(s)} = (X^{(s)})_{i\in [n]}$. Through data processing arguments, we can pass to the Fisher information of the observational model for $X^{(s)}$ given the transcript in the sub-model $\cH^\alpha_{R,L}$. We consider the following dynamics for $X^{(s)}$:
	\begin{equation}\label{eq:holder-smooth-X-generation}
		X_i^{(s)} = f + R 2^{-L(\alpha + 1/2) - 1} \sum_{k=1}^{2^L} (1/2 - B_{ki}^{(s)}) \phi_{Lk}, \quad i \in [n],
	\end{equation}
	for i.i.d. Beta$(2,2)$ random variables $B_{ki}^{(s)}$. It is easy to see that each $X^{(s)}_i$ takes values in a H\"older ball of radius $R$ and has mean $f$. Letting $S^{X^{(s)}}$ denote the score along the vector $f_L$, we let $I_f^{X^{(s)}|T^{(s)}}$ denote the Fisher information of the observational model for $X^{(s)}$ given the transcript in the sub-model $\cH^\alpha_{R,L}$, 
	\begin{equation}\label{eq:fisher-information-Ys}
		I^{X^{(s)}|T^{(s)}}_f = \E_f \, \E_f\left[ S_f^{X^{(s)}} \bigg| T^{(s)} \right] \E_f\left[ S_f^{X^{(s)}} \bigg| T^{(s)} \right]^\top.
	\end{equation}
	The Fisher information of $X_i^{(s)}$ is defined as $I^{X^{(s)}}_f= \E_f S_f^{X^{(s)}} {S_f^{X^{(s)}}}^\top$. 

	The crux of the proof lies in the following lemma, which provides a lower bound on the minimax risk by expressing it in terms of a combination of Fisher information quantities—each capturing a different bottleneck (sample sizes, number of design points, privacy constraints).
	\begin{lemma}\label{lem:functional-von-trees}
		Let $L \in \N$. The minimax risk $\inf_{\hat f \in \cM_{\varepsilon,\delta}} \sup_{f \in \cH^\alpha(R)} \E_f \|\hat f - f\|_2^2$ is lower bounded by
		\begin{equation}\label{eq:von-trees-lb-in-lemma}
			 \frac{2^{2L}}{ \sup_{f \in \cH^\alpha_{R,L}} \sum_{s=1}^{S}  \min \left\{ \text{Tr}(I_f^{Y^{(s)}|T^{(s)}}), \text{Tr}(I_f^{Y^{(s)}}), \text{Tr}(I_f^{X^{(s)}|T^{(s)}}), \text{Tr}(I_f^{X^{(s)}}) \right\} + 2^{2L(\alpha + 2)}}.
		\end{equation}

	\end{lemma}
	It can be seen as a generalization of the technique introduced by \cite{cai2024optimal}, which allows us to account for the situation where the observational model for $Y^{(s)}$ leads to close to equal performance for estimating $f$ as one would have if $X^{(s)}$ were to be observed directly. Its proof relies on a multivariate version of the von Trees inequality of \cite{gill1995applications} and data processing arguments exploiting the Markov chain structure $f \to X^{(s)} \to Y^{(s)}$ of the model.

	The minimum of the summand in \eqref{eq:von-trees-lb-in-lemma} effectively correspond to four different scenarios that can arise for each server, depending on which information theoretic property forms the bottleneck. 

\begin{enumerate}[(i)]
	\item In the first scenario we consider, both the privacy constraint, the limited number of observations as well the limited number of measurements affect the information that the transcript $T^{(s)}$ contains on $f$. This corresponds with the quantity $I_f^{Y^{(s)}|T^{(s)}}$, which we in turn bound by a constant multiple of $m_s n_s^2 \varepsilon_s^2$.
	\item In the second scenario, the privacy constraint is relatively lenient for the $s$-th server, in which case the Fisher information of its transcript satisfies $I_f^{Y^{(s)}|T^{(s)}} \asymp I_f^{Y^{(s)}}$, the latter of which can be computed to equal $2^L m_s n_s$.
	\item In the scenario that the $s$-th server contains relatively many measurements per individual, the sample size $n_s$ and privacy constraint form the bottleneck for the $s$-th server. In this case, $I_f^{X^{(s)}|T^{(s)}}$ is the relevant quantity to consider, which we bound by a constant multiple of $ 2^{L (2\alpha+1)} n_s^2 \varepsilon_s^2$ using data processing arguments.
	\item In the last scenario, both measurements and privacy budget are no longer part of the bottleneck, and just the (number of) $X_i$ observations locally determine information transmitted by machine $s$. In this case, the $I_f^{X^{(s)}|T^{(s)}} \asymp I_f^{X^{(s)}} \leq 2^{2L (\alpha+1)} n_s$.
\end{enumerate}

The bounds on the traces of the various Fisher information are given in the Lemma \ref{lem:fisher-info-bounds} below, for which we give a proof in Section \ref{sec:proofs-lower-bound-independent-design} of the Supplementary Material.

\begin{lemma}\label{lem:fisher-info-bounds}
	Suppose that for $\delta ' = \min_s \delta_s$, it holds that $\delta' \log(1/\delta') \lesssim \left(\frac{n_s}{m_s} \wedge \sqrt{\frac{n_s}{m_s}}\right)\varepsilon_s^2/N$. 

	Then,
	\begin{align*}
		\min \left\{ \text{Tr}(I_f^{Y^{(s)}|T^{(s)}}), \text{Tr}(I_f^{Y^{(s)}}), \text{Tr}(I_f^{X^{(s)}|T^{(s)}}), \text{Tr}(I_f^{X^{(s)}}) \right\} \lesssim \\ \min \left\{ m_s n_s^2 \varepsilon_s^2 , 2^L m_s n_s, 2^{L(2\alpha + 1)} n_s^2 \varepsilon_s^2, 2^{2L(\alpha + 1)} n_s \right\}.
	\end{align*}
	\end{lemma}

The proof of Theorem \ref{thm:lower-bound-independent-design} follows a choice of $L$ as a function of $D_*$, where $D_*$ is the solution to equation \eqref{eq:rate-determining-relation:independent-design}.

\subsection{Lower bound for common design}\label{ssec:lower-bound-common-design}

For the common design setting, we prove the following lower bound on the minimax risk. 

\begin{theorem}\label{thm:common-design-lower-bound}
	Assume that, for $\delta ' = \min_s \delta_s$, we have $\delta' \log(1/\delta') \lesssim \left(\frac{n_s}{m_s} \wedge \sqrt{\frac{n_s}{m_s}}\right)\varepsilon_s^2/N$. 
	
	Then, given a common equispaced design $(\zeta_{j})_{j \in [m]}$ across the $S$ servers, the minimax estimation risk satisfies
	\begin{equation}\label{eq:rate-independent-design}
		\inf_{\hat f \in \cM_{\varepsilon,\delta}}	\sup_{f \in \cH^\alpha(R)} \E_f \|\hat f - f\|_2^2 \gtrsim D^{-2\alpha}_*,
	\end{equation}
where $D_*$ solves~\eqref{eq:rate-determining-relation:common-design}.
\end{theorem}

Below, we provide sketch of the proof, highlighting the differences with the independent design lower boun. A detailed proof is given in Section \ref{ssec:lower-bound-common-design} of the Supplementary Material. The lower bound relies on a similar construction as in the independent design case, but with a different data generating process. Under this common design setup, the minimax rate is either completely governed by the approximation error due to a limited number of design points, which is a cross-server phenomenon, or by the privacy constraints, which are server-specific. 

Formally, \eqref{eq:rate-determining-relation:common-design} indicates the occurance of two scenarios. When $D_* = m^{-2\alpha}$, the rate is governed by the deterministic approximation error (see e.g. \cite{devore1993constructive}). Here, the privacy budget is of no concern. The challenging regime in terms of its proof strategy is when $D_*$ takes its value in the second argument of the minimum of \eqref{eq:rate-determining-relation:common-design}. The server specific cost is captured by a similar argument as that of Lemma \ref{lem:functional-von-trees}, bounding the minimax risk by a combination of Fisher information quantities using a multivariate version of the von Trees inequality:
\begin{equation*}%\label{eq:von-trees-lb-in-lemma}
	\frac{2^{2L}}{ \sup_{f \in \cH^\alpha_{R,L}} \sum_{s=1}^{S}  \min \left\{ \text{Tr}(I_f^{X^{(s)}|T^{(s)}}), \text{Tr}(I_f^{X^{(s)}}) \right\} + 2^{2L(\alpha + 2)}},
\end{equation*}
where the ``sub-model Fisher information'' matrices $I_f^{X^{(s)}|T^{(s)}}$ and $I_f^{X^{(s)}}$ relate to  different data generating process that target specifically the cost due to limited local sample sizes and privacy constraints. Combining a geomertric argument with a data processing argument, these matrics have trace bounded by constant multiples of $2^L n_s^2 \varepsilon_s^2$ and $n_s 2^{2L}$, respectively. Optimizing over the choice of $L$ then yields the result.

\section{Simulation Study}

To highlight the differences between the independent and common design settings, we consider the setup of central differential privacy. In such a setting, where data is collected by a central server, the difference between common and independent design just boils down to whether each individual has their measurements taken at the same design points, or whether the design points are random independent across individuals. 

In Section \ref{ssec:simulation-design}, we describe the simulation design. 

\subsection{Simulation Design}\label{ssec:simulation-design}

Specifically, the random curves \( X_i \) for each individual are generated as follows:
\begin{equation}
	X(t) = R \cdot \sum_{l=0}^{L^*} \sum_{k=0}^{2^l - 1} s_{lk} 2^{-l(\alpha + 1/2)} \psi_{lk}(t), \quad \text{where} \quad s_{lk} \overset{\text{i.i.d.}}{\sim} \text{Rad}(p).
\end{equation}
Here, the wavelets \( \psi_{lk} \) are selected from the Daubechies' extremal phase wavelet family. The parameter \( R \) represents the radius of the H\"older ball in which the random function resides, and \( \alpha \) denotes the smoothness parameter. The term \(\text{Rad}(p)\) refers to the Rademacher distribution, which takes values \( +1 \) and \( -1 \) with probabilities \( p \) and \( 1-p \), respectively.

This construction ensures that the random functions \( X_i \) exhibit the desired smoothness properties and fall within the H\"older ball of radius \( R \), while the use of Daubechies' wavelets allows for compact support and flexibility in approximation.

We are interested in estimating the mean of the these curves given by $f(\cdot ) = \E X(\cdot)$ which evaluates to
\begin{equation}
	f(t) = R \cdot \sum_{l=0}^{L^*} \sum_{k=0}^{2^l - 1} (2p -1)\cdot 2^{-l(\alpha + 1/2)} \psi_{lk}(t).
\end{equation}
For both of the simulation studies below we set $R =2$, ${L^*} = 15$ and $p = 0.9$.
\subsection{Independent Design}

\begin{figure}[t]
	\centering
	\includegraphics[scale = 0.8]{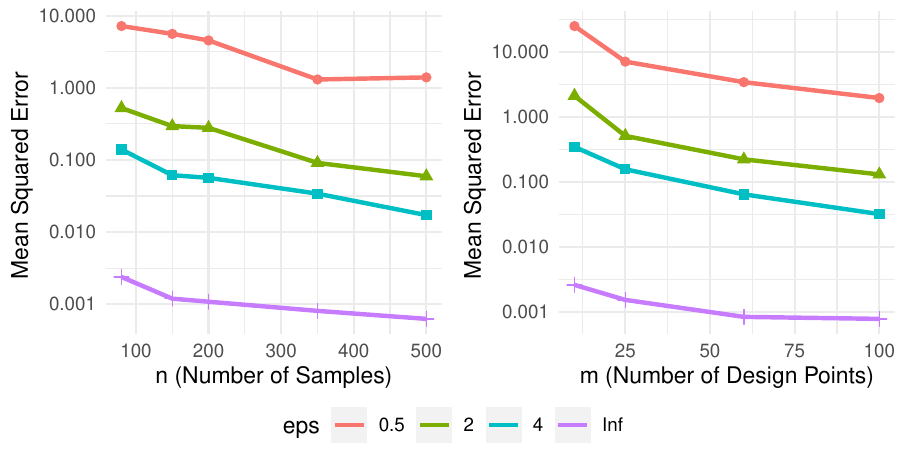}
	\caption{ MSE vs $n$ (with fixed $m =64$) and MSE vs $m$ (with fixed $n= 200$) for Independent design}\label{fig:independent-design}
\end{figure}

In the independent design setting, each individual has their sampling points \( \zeta_{ij} \) drawn independently from a uniform distribution over \([0,1]\). This accounts for greater variability in the observed data, as both the sampling locations and the random curves \( X_i \) vary across individuals. To estimate the mean function \( f(t) = \mathbb{E}[X(t)] \), we employ the wavelet-based private estimator of Section \ref{ssec:algorithm-independent-design} (Algorithm \ref{alg:independent-design-estimator}). We study the impact of varying \( n \), the number of individuals, and \( m \), the number of sampling points per individual, for different privacy budgets \( \epsilon \). Throughout, we set the privacy parameter \( \delta = 1/n^2 \). The performance is evaluated using the mean squared error (MSE), averaged over multiple runs, to capture the estimator's accuracy across varying settings. Figure~\ref{fig:independent-design} presents the MSE (log-scale) against the number of samples ($n$) and design points ($m$). 

The intuitively obvious finding that can be observed from both panels in Figure \ref{fig:independent-design}, is that estimation error decreases as the privacy budget ($\varepsilon$) increases and / or as the number of individuals in the sample increase $n$. Similarly, for an increase in the number of measurements, we see that an increase in the number of design points $m$ leads to a decrease in the estimation error. However, for very large $m$, the error plateaus, indicating a phase transition (depending on $\varepsilon$) where the MSE becomes independent of $m$, which in line with the main results (see Corollary \ref{cor:rate-independent-design}).

\subsection{Common Design}\label{ssec:simulation-common-design}

In the common design setting, all individuals share the same set of fixed sampling points, chosen as \( m \) equally spaced points in the interval \([0,1]\), whilst otherwise the simulation setting remains the same: For each individual, we observe noisy evaluations of their random curve \( X_i \) at these fixed design points. To this end, we employ the privatized local polynomial regression estimator of Section \ref{ssec:common-design-upper-bound} (Algorithm \ref{alg:common-design-estimator}), where we choose $K$ to be Gaussian kernel. 

We study the effect of varying \( n \), the number of individuals, and \( m \), the number of design points, under different privacy budgets \( \epsilon \). Throughout, we set \( \delta = 1/n^2 \) as the privacy parameter for all simulations. The performance of the estimators is evaluated using the integrated mean squared error.  Figure~\ref{fig:common-design} presents the MSE (log-scale) against the number of samples ($n$) and design points ($m$), providing a comparison of the impact of \( n \), \( m \), and \( \epsilon \) on estimation accuracy. 

Similarly to the dependent design case, the estimation error decreases as the privacy budget ($\varepsilon$) increases and / or as $n$ increases, with a plateau occuring for very large $n$ and $\varepsilon$. This plateau corresponds to the phase transition where the estimation error is essentially fully determined by the approximation error due to the limited number of design points, and is not observed in the independent design case. This corroborates the main results, comparing Corollaries \ref{cor:rate-independent-design} and \ref{cor:rate-common-design}.

Another stark difference compared to the independent design case, is that the estimation error plateaus much quicker as $m$ increases, transitioning to a regime where the rate becomes independent of $m$. This is in line with the main results (see Corollary \ref{cor:rate-common-design}), which state that there is only a benefit of having more measurements per individual if the privacy constraints are not forming a bottleneck. In the regimes where the differential privacy constraint is sufficiently lenient, there is a benefit of additional measurements up to the point that the number of individuals forms the bottleneck, as can be seen in the case where $\varepsilon = \infty$ in Figure \ref{fig:common-design}.

\begin{figure}[t]
	\centering
	\includegraphics[scale = 0.8]{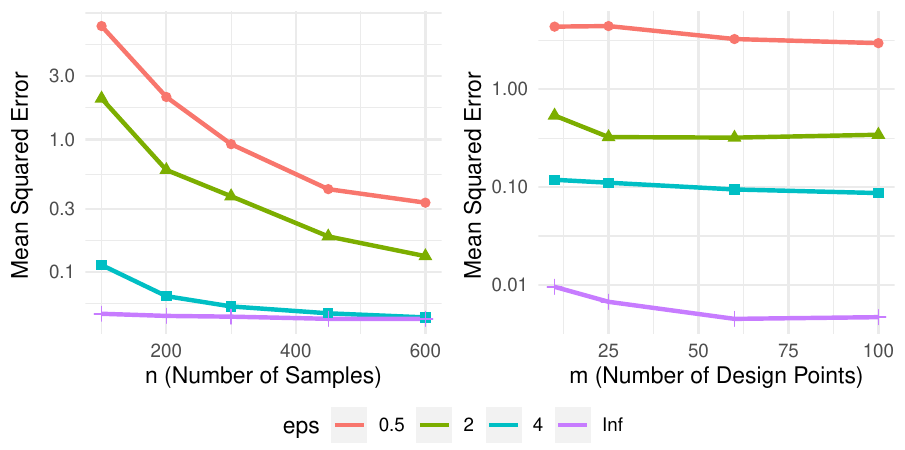}
	\caption{Log-scale MSE vs $n$ (with fixed $m =64$) and vs $m$ (with fixed $n= 200$) for the common design setting.}\label{fig:common-design}
\end{figure}
%\section{Discussion}

% reference
\bibliographystyle{plain}
\bibliography{references}

\section*{Supplementary Material}

\section{Proofs related to the upper bound theorems}

\subsection{Proof for the independent design upper bound}
We first prove a upper bound on the $\ell_2$-sensitivity i.e Lemma
 \ref{lemma:L2-sensitivity-independent-design}.
 \subsubsection{Proof of Lemma \ref{lemma:L2-sensitivity-independent-design}}\label{sec:proof-of-L2-sensitivity}
 \begin{proof}%[Proof of Lemma \ref{lemma:L2-sensitivity-independent-design}] 
 	We aim to bound the $\ell_2$-sensitivity of the statistic $\tilde{f}^{\tau,s}_L$ under neighboring datasets \( Z^{(s)} \) and \( Z^{'(s)} \). Neighboring datasets differ in at most one individual:
 	\[
 	Z^{(s)}_i = \big(Y_{ij}^{(s)}, \zeta^{(s)}_{ij}\big)_{j \in [m_s]} \neq Z^{'(s)}_i = \big(Y_{ij}^{'(s)}, \zeta_{ij}^{'(s)}\big)_{j \in [m_s]}.
 	\]
 	Define the privatized coefficient for \( Z^{'(s)} \):
 	\[
 	\tilde{f}^{\prime \tau, s}_{lk} := \frac{1}{\sqrt{2^l \wedge m_s} n_s} \sum_{i=1}^{n_s} \left[ \frac{1}{m_s} \sum_{j=1}^{m_s} Y^{'(s)}_{ij} \psi_{lk}(T^{'(s)}_{ij}) \right]_{\tau_l^{(s)}}.
 	\]
 	The $\ell_2$ sensitivity is given by:
 	\[
 	\left\|\tilde{f}^{\tau,s}_L(Z^{(s)}) - \tilde{f}^{\tau,s}_L(Z^{'(s)}) \right\|_2^2 = \sum_{l=l_0}^L \sum_{k=0}^{2^l - 1} \big(\tilde{f}^{\tau,s}_{lk} - \tilde{f}^{\prime \tau,s}_{lk}\big)^2.
 	\]
 	Expanding $\big(\tilde{f}^{\tau,s}_{lk} - \tilde{f}^{\prime \tau,s}_{lk}\big)^2$, it can be seen to equal
 	\[
 	 \frac{1}{(2^l \wedge m_s) (\tau_l^{(s)})^2 n_s^2} \left(\left[\frac{1}{m_s} \sum_{j=1}^{m_s} Y^{(s)}_{ij} \psi_{lk}(\zeta^{(s)}_{ij}) \right]_{\tau_l^{(s)}} - \left[\frac{1}{m_s} \sum_{j=1}^{m_s} Y^{'(s)}_{ij} \psi_{lk}(\zeta^{'(s)}_{ij}) \right]_{\tau_l^{(s)}}\right)^2.
 	\]
 	Since the wavelet basis \(\psi_{lk}\) is compactly supported, with its support of the order $2^{-l}$ at resolution level $l$, at most a constant number of wavelet functions overlap at any design point at each resolution level. Thus, the summation over \(k\) is bounded by the number of overlapping supports. Let \(c_A\) denote this constant, which gives the upper bound 
 	\[
 	\sum_{l=l_0}^L \sum_{k=0}^{2^l - 1} \frac{1}{(2^l \wedge m_s) n_s^2} \mathbbm{1}\left\{\exists j : \{\zeta^{(s)}_{ij}, \zeta_{ij}^{'(s)}\} \cap \text{supp}(\psi_{lk}) \neq \emptyset\right\} 
 	\leq \sum_{l=l_0}^L \frac{c_A (2^l \wedge m_s)}{(2^l \wedge m_s) n_s^2}.
 	\]
 	Combining terms, the $\ell_2$ sensitivity satisfies:
 	\[
 	\left\|\tilde{f}^{\tau,s}_L(Z^{(s)}) - \tilde{f}^{\tau,s}_L(Z^{'(s)})\right\|_2^2 \leq \frac{c^2_A L}{n_s^2},
 	\]
 	for some constant \(c_A > 0\). Taking the square root completes the proof.
 \end{proof}
 
\subsubsection{Proof of upper bound for independent design}\label{sssec:proof-of-independent-design-upper-bound}

We next prove the upper bound i.e Theorem \ref{thm:independent-design-upper-bound}.

\begin{proof}[Proof of Theorem \ref{thm:independent-design-upper-bound}]
	Let \( U^{(s)}_i = \big(U^{(s)}_{i,lk}\big)_{k=0,\ldots,2^l-1, \, l_0 \leq l \leq L} \), where:
	\begin{equation}\label{eq:def-U_il}
	U^{(s)}_{i,lk} := \frac{1}{m_s} \sum_{j=1}^{m_s} Y^{(s)}_{ij} \psi_{lk}(\zeta^{(s)}_{ij}).
	\end{equation}
	The randomized \((\varepsilon_s, \delta_s)\)-DP algorithm outputs a noisy version of the clipped coefficients:
	\[
	\widehat{f}^{P,s}_{lk} = \widehat{f}^{\tau,s}_{lk} + W^{(s)}_{lk}, \quad \widehat{f}^{\tau,s}_{lk} = \frac{1}{n_s} \sum_{i=1}^{n_s} \big[ U^{(s)}_{i,lk} \big]_{\tau_l^{(s)}},
	\]
	where \( W^{(s)}_{lk} \overset{\text{i.i.d.}}{\sim} N\left(0, \frac{4L(2^l \wedge m_s)(\tau_l^{(s)})^2 \log(2/\delta_s)}{n_s^2 \varepsilon_s^2}\right). \)

	Given that the wavelets form and orthonormal basis, Plancharel's theorem yields that the squared $L_2$-risk is given by
	\begin{equation}\label{eq:Plancharel}
	 \sum_{l\geq l_0}^{\infty} \sum_{k=1}^{2^l}	\E_f\big(\widehat{f}^P_{lk} - f_{lk}\big)^2.
	\end{equation}
	where $\widehat{f}^P_{lk} = 0$ for $l > L$. For $l \leq L$, each summand equals 
	\[
		\E_f\left(\sum_{s=1}^S w_l^{(s)} \widehat{f}^{P,s}_{lk} - f_{lk}\right)^2  = \underbrace{\E_f\left(\sum_{s=1}^S w_l^{(s)} \widehat{f}^{\tau,s}_{lk} - f_{lk}\right)^2}_{\text{bias + sampling variance}} + \underbrace{\E_f\left(\sum_{s=1}^S w_l^{(s)} W^{(s)}_{lk}\right)^2}_{\text{Privacy Noise}}.
	\]
	Below, we bound each of these terms separately.

	\paragraph{Bounding the sampling variance and bias}
	Under the event \( A^{(s)}_{lk} \equiv \{\forall i \in [n_s] : [U^{(s)}_{i,lk}]_{\tau_l^{(s)}} = U^{(s)}_{i,lk} \} \), the clipped estimator \(\widehat{f}^{\tau,s}_{lk}\) equals the unbiased sample mean \(\widehat{f}^{(s)}_{lk}\). By Lemma~\ref{lemma:concentration-of-U}, this occurs with high probability:
	\[
	\P(A^{(s)}_{lk}) \geq 1 - \frac{1}{N^c}.
	\]
	Decomposing the expectation into two parts, we have
		\begin{align*}
		\E_f&\left( \sum_{s=1}^Sw_l^{(s)}\widehat f^{\tau,s}_{lk}- f_{lk}\right)^2\\ &= \E_f\left( \sum_{s=1}^Sw_l^{(s)}\widehat f^{\tau,s}_{lk}- f_{lk}\right)^2\mathbbm{1}\{\cap_{s=1}^SA^{(s)}_{lk}\} +   \E_f\left( \sum_{s=1}^Sw_l^{(s)}\widehat f^{\tau,s}_{lk}- f_{lk}\right)^2\mathbbm{1}\left\{\cap_{s=1}^SA^{(s)}_{lk}\right\}^c\\
		&\leq \E_f\left( \sum_{s=1}^Sw_l^{(s)}\widehat f^{(s)}_{lk}- f_{lk}\right)^2\mathbbm{1}\{\cap_{s=1}^SA^{(s)}_{lk}\} + c_R\sum_{s=1}^S\E_f \mathbbm{1}\{(A^{(s)}_{lk})^c\} \\
		&\leq\sum_{s=1}^S(w_l^{(s)})^2 \E_f(\widehat f^{(s)}_{lk}- f_{lk})^2  + c_R\frac{S}{N^c}.
	\end{align*}
	Here, $\widehat{f}^{(s)}_{lk} - f_{lk}$ has $\E_f$-expectation zero and satisfies 
	\[
	\E_f \big(\widehat{f}^{(s)}_{lk} - f_{lk}\big)^2 \lesssim \frac{1}{m_s n_s} + \frac{2^{-l(2\alpha+1)}}{n_s},
	\]
	by Lemma \ref{lemma:MSE-of-non-private-estimator}. Combining these bounds, we have that 
	\[
	\E\left( \sum_{s=1}^Sw_l^{(s)}\widehat f^{\tau,s}_{lk}- f_{lk}\right)^2 \lesssim \sum_{s=1}^S(w_l^{(s)})^2 \left(\frac{1}{m_sn_s} + \frac{2^{-l(2\alpha+1)}}{n_s}\right) + \frac{S}{N^c}.
	\]

	\paragraph{Bounding the Privacy Noise}
	We have
	\[
	\E_f\left(\sum_{s=1}^S w_l^{(s)} W^{(s)}_{lk}\right)^2 = \sum_{s=1}^S \big(w_l^{(s)}\big)^2 \cdot \frac{4L(2^l \wedge m_s)(\tau_l^{(s)})^2 \log(4/\delta)}{n_s^2 \varepsilon_s^2}.
	\]
	Recalling the definition of the clipped estimator, we have that
	\begin{equation*}%\label{eq:def-clipping}
		\tau_l^{(s)} = 2 (2c\log N)^{3/2}\left(\sqrt{\frac{1}{m_s}} + \frac{1}{3}\|\psi\|_\infty \frac{2^{l/2}}{m_s}\right)+ R2^{-l(\alpha+1/2)},
		\end{equation*}
	from which it follows that
	\[
	(2^l \wedge m_s) (\tau_l^{(s)})^2 \lesssim (\log N)^3 \left(\frac{1}{m_s} + \frac{2^l}{m_s^2}\right)(2^l \wedge m_s) + 2^{-2l(\alpha+1/2)}(2^l \wedge m_s).
	\]
	Substituting this in the earlier expression, we obtain that
	\[
	\E\left(\sum_{s=1}^S w_l^{(s)} W^{(s)}_{lk}\right)^2 \lesssim \sum_{s=1}^S \big(w_l^{(s)}\big)^2 \cdot \left(\frac{L2^l (\log N)^3 \log(4/\delta)}{m_s n_s^2 \varepsilon_s^2} + \frac{L2^{-2\alpha l} \log(4/\delta)}{n_s^2 \varepsilon_s^2}\right).
	\]
	
	\paragraph{Combining the bounds}
	Combining the previously derivaed bounds, we find that $\E_f\big(\widehat{f}^P_{lk} - f_{lk}\big)^2$ is bounded by a constant multiple of
	\[
	 \sum_{s=1}^S \big(w_l^{(s)}\big)^2 \left[ \frac{1}{m_s n_s} + \frac{2^{-l(2\alpha+1)}}{n_s} + \frac{L2^l (\log N)^3 \log(4/\delta)}{m_s n_s^2 \varepsilon_s^2} + \frac{L2^{-2\alpha l} \log(4/\delta)}{n_s^2 \varepsilon_s^2} \right] +  \frac{S}{N^c},
	\]
	for each $l \leq L$. 
	
	The optimal choice of weights \(w_l^{(s)}\) is revealed to be inversely proportional to the variance;
	\begin{equation*}%\label{eq:choice-of-weights}
		w_l^{(s)} = \frac{u_l^{(s)}}{\sum_{s=1}^S u_l^{(s)}}, \quad \text{where} \quad u_l^{(s)} = n_s m_s \wedge m_s 2^{l(2\alpha+1)} \wedge 2^{-l} m_s n_s^2 \varepsilon_s^2 \wedge 2^{2l\alpha} n_s^2 \varepsilon_s^2.
	\end{equation*}
	e.g.
	\[
	w_l^{(s)} \propto \frac{1}{m_s n_s \wedge 2^{l(2\alpha+1)} n_s \wedge 2^{-l} m_s n_s^2 \varepsilon_s^2 \wedge 2^{2\alpha l} n_s^2 \varepsilon_s^2}.
	\]
	This implies that $\E_f(\widehat f^{P}_{lk}- f_{lk})^2$ is bounded above by
	\begin{align*}
		& L \log(N)^3 \log(4/\delta)\sum_{s=1}^S(w_l^{(s)})^2 \left(\frac{1}{m_sn_s} + \frac{2^{-l(2\alpha+1)}}{n_s}+ \frac{ 2^l }{m_sn_s^2 \varepsilon_s^2} + \frac{   2^{-2\alpha l} }{n_s^2 \varepsilon_s^2}\right) + \frac{S}{N^c}\\
		&\lesssim L \log(N)^3 \log(4/\delta) \sum_{s=1}^S(w_l^{(s)})^2\left(\frac{1}{m_sn_s \wedge 2^{l(2\alpha+1)}n_s \wedge2^{-l} m_sn_s^2 \varepsilon_s^2 \wedge 2^{2\alpha l}n_s^2 \varepsilon_s^2}\right) + \frac{S}{N^c}\\
		&=  L \log(N)^3 \log(4/\delta) \frac{1}{\sum_{s=1}^S m_sn_s \wedge 2^{l(2\alpha+1)}n_s \wedge2^{-l} m_sn_s^2 \varepsilon_s^2 \wedge 2^{2\alpha l}n_s^2 \varepsilon_s^2}+ \frac{S}{N^c}
	\end{align*}

	Expanding \eqref{eq:Plancharel}, we obtain
	\[
	\E_f\|\widehat{f}^P_L - f\|_2^2 = \sum_{l=l_0}^L \sum_{k=0}^{2^l - 1} \E_f\big(\widehat{f}^P_{lk} - f_{lk}\big)^2 + \sum_{l > L} \sum_{k=0}^{2^l - 1} f_{lk}^2.
	\]
	Hence we have that the estimation risk $\E_f\|\widehat f^{P}_L - f\|_2^2$ is bounded as
	\begin{align*}
		&  L \log(N)^3 \log(4/\delta) \sum_{l=l_0}^L\frac{2^l}{\sum_{s=1}^S m_sn_s \wedge 2^{l(2\alpha+1)}n_s \wedge2^{-l} m_sn_s^2 \varepsilon_s^2 \wedge 2^{2\alpha l}n_s^2 \varepsilon_s^2}+ \frac{S2^L}{N^c} + 2^{-2L\alpha}\\
		&\lesssim   L \log(N)^3 \log(4/\delta) \sum_{l=l_0}^L\frac{2^l}{\sum_{s=1}^S m_sn_s \wedge 2^{l(2\alpha+1)}n_s \wedge2^{-l} m_sn_s^2 \varepsilon_s^2 \wedge 2^{2\alpha l}n_s^2 \varepsilon_s^2} + 2^{-2L\alpha}\\
		&\lesssim   L^2 \log(N)^3 \log(4/\delta) \max_{l=l_0,\ldots,L}\frac{2^l}{\sum_{s=1}^S m_sn_s \wedge 2^{l(2\alpha+1)}n_s \wedge2^{-l} m_sn_s^2 \varepsilon_s^2 \wedge 2^{2\alpha l}n_s^2 \varepsilon_s^2} + 2^{-2L\alpha}
	\end{align*}
	The second last line follows from the fact that $S2^L/N = O(1/N)$ (since $2^L <N$ and $c > 3$).
	By setting $d = 2^l$  and $D = 2^L$, we obtain
	\begin{align*}
		\E_f\|\widehat f^{P}_L - f\|_2^2 \lesssim \frac{(\log D)^2 (\log(N))^3 \log(4/\delta)}{\min_{d \leq D} \sum_{s=1}^S d^{-1}m_sn_s \wedge d^{2\alpha}n_s \wedge d^{-2} m_sn_s^2 \varepsilon_s^2 \wedge d^{2\alpha -1}n_s^2 \varepsilon_s^2} + D^{-2\alpha}
	\end{align*}
	Using $D^*$ which satisfies~\eqref{eq:rate-determining-relation:independent-design} and letting $L = [\log_2D^*]$ we obtain the minimax rate of $$(\log D^*)^2 \log(N)^3 \log(4/\delta) (D^*)^{-2\alpha}$$. 
	Since $D^* \leq N$ this implies the minimax risk is further upper bounded by
	$$(\log(N)^5 \log(4/\delta) (D^*)^{-2\alpha}.$$
\end{proof}

\subsubsection{Proof of Lemma \ref{lemma:concentration-of-U}}

\begin{lemma}\label{lemma:concentration-of-U-supplement}
	
	For a fixed $l$ and $k$, $\{U^{(s)}_{i,lk}\}_{i\in[n]}$ as defined in~\eqref{eq:def-U_il} satisfies
	\begin{equation*}
		\P_f(\forall \, i\; |U^{(s)}_{i,lk}| < \tau_l^{(s)}) \geq 1 - \gamma
	\end{equation*}
	with 
	\begin{equation}
		\tau^{(s)}_l = 2 (2c\log N)^{3/2}\left(\sqrt{\frac{1}{m_s}} + \frac{1}{3}\|\psi\|_\infty \frac{2^{l/2}}{m_s}\right)+ R2^{-l(\alpha+1/2)} \quad \text{and} \quad \gamma = (N)^{-c}.
	\end{equation}
\end{lemma}

\begin{proof}%[Proof of Lemma \ref{lemma:concentration-of-U}]
	By a union bound, it is enough to show that $\P_f( |U^{(s)}_{i,lk}| > \tau_l^{(s)}) \leq \gamma/n$. By the triangle inequality a union bound,
	\begin{align}
		\P_f( |U^{(s)}_{i,lk}| > \tau_l^{(s)}) &\leq \P_f\left( |U^{(s)}_{i,lk} - \E(U^{(s)}_{i,lk} \mid X^{(s)}_i) | > \sqrt{\frac{c_1 \log(N)}{m_s}}\right) \label{eq:Uij-concentration-proof-initial-bound} \\ 
		&+ \P_f\left( |\E(U^{(s)}_{i,lk} \mid X^{(s)}_i) | > c_2 2^{-l(\alpha+1/2)}\right). \nonumber
	\end{align}
	We first bound the second term by observing that:
	\[
	\E(U^{(s)}_{i,lk} \mid X^{(s)}_i) = \E(X^{(s)}_i(T_{ij})\psi_{lk}(T_{ij})\mid X^{(s)}_i) = \int X^{(s)}_i(t)\psi_{lk}(t) dt.
	\]
	Since \(X^{(s)}_i\) lies in the \(\alpha\)-Hölder ball with radius \(R\), we have:
	\[
	\big|\E(U^{(s)}_{i,lk} \mid X^{(s)}_i)\big| \leq R 2^{-l(\alpha+1/2)}.
	\]
	For \(c_2 \geq R\), this implies:
	\[
	\P\big(|\E(U^{(s)}_{i,lk} \mid X^{(s)}_i)| > c_2 2^{-l(\alpha+1/2)}\big) = 0.
	\]
	We now turn to the first term in \eqref{eq:Uij-concentration-proof-initial-bound}. We have
	\[
	U^{(s)}_{i,lk} - \E(U^{(s)}_{i,lk} \mid X^{(s)}_i) = \frac{1}{m_s} \sum_{j=1}^{m_s} \Big(Y^{(s)}_{ij} \psi_{lk}(\zeta^{(s)}_{ij}) - \E \big[Y^{(s)}_{ij} \psi_{lk}(\zeta^{(s)}_{ij}) \mid X^{(s)}_i\big]\Big),
	\]
	which decomposes further as
	\begin{equation}\label{eq:proof-Uij-concentration-lemma-decomposition}
	\frac{1}{m_s} \sum_{j=1}^{m_s} \Big(X^{(s)}_i(\zeta^{(s)}_{ij}) \psi_{lk}(\zeta^{(s)}_{ij}) - \int X^{(s)}_i(t)\psi_{lk}(t) dt\Big) 
	+ \frac{1}{m_s} \sum_{j=1}^{m_s} \xi^{(s)}_{ij} \psi_{lk}(\zeta^{(s)}_{ij}).
	\end{equation}
	We apply Bernstein's inequality (Lemma \ref{lemma:bernstein}) to the first summand. Let
	\[
	V_i = X^{(s)}_i(\zeta^{(s)}_{ij}) \psi_{lk}(\zeta^{(s)}_{ij}) - \int X^{(s)}_i(t)\psi_{lk}(t) dt.
	\]
	Then:
	\[
	\mathrm{Var}(V_i) \leq \|X\|_\infty^2 = C^2_{\alpha,R}, \quad \max_i \|V_i\|_\infty \leq C_{\alpha,R} \|\psi\|_\infty 2^{l/2},
	\]
	where $C_{\alpha,R}$ is a constant as in Lemma \ref{lem:holder-norm-sup-norm}. Set \(\lambda = 2c \log N \left(\frac{C_{\alpha,R}}{\sqrt{m_s}} + \frac{1}{3}C_{\alpha,R}\|\psi\|_\infty\frac{2^{l/2}}{m_s}\right)\). By Lemma \ref{lemma:bernstein}, it holds with probability \(1 - \kappa_1 N^{-c}\) that
	\[
	\frac{1}{m_s} \sum_{j=1}^{m_s} \Big(X^{(s)}_i(\zeta^{(s)}_{ij}) \psi_{lk}(\zeta^{(s)}_{ij}) - \int X^{(s)}_i(t)\psi_{lk}(t) dt\Big) 
	\leq2c \log N \left(\frac{C_{\alpha,R}}{\sqrt{m_s}} + \frac{1}{3}C_{\alpha,R}\|\psi\|_\infty\frac{2^{l/2}}{m_s}\right),
	\]
	where \(K(\alpha, R)\) is a suitable constant.

	Next, we consider the second sum in \eqref{eq:proof-Uij-concentration-lemma-decomposition}. Write $\tilde{V}_i := \xi^{(s)}_{ij} \psi_{lk}(\zeta^{(s)}_{ij})$. Conditioning on \({\xi}_i^{(s)} := \{\xi_{ij}^{(s)}\}_{j \in [m_s]}\), we again apply Bernstein's inequality: We have
	\[
	\mathrm{Var}(\tilde{V}_i \mid {\xi}_i^{(s)}) \leq \|{\xi}_i^{(s)}\|^2_\infty, \quad \max_i \|\tilde{V}_i\|_\infty \leq \|{\xi}_i^{(s)}\|_\infty 2^{l/2} \|\psi\|_\infty,
	\]
	so by Lemma \ref{lemma:bernstein}, with probability \(1 - 1/N^c\), it holds that
	\[
	\frac{1}{m_s} \sum_{j=1}^{m_s} \xi^{(s)}_{ij} \psi_{lk}(\zeta^{(s)}_{ij}) \leq 2c \|{\xi}_i^{(s)}\|_\infty \log N \left(\frac{1}{\sqrt{m_s}} +\frac{1}{3}\|\psi\|_\infty \frac{2^{l/2}}{m_s}\right).
	\]
	By standard calculations, \(\|{\xi}_i^{(s)}\|_\infty \leq  \sqrt{2c\log N}\) with probability at least $1-2/N^c$. 
	% Combining these tail bounds, we obtain that 
	% \begin{align*}
	% 	\P_f \left( \frac{1}{m_s} \sum_{j=1}^{m_s} \xi^{(s)}_{ij} \psi_{lk}(\zeta^{(s)}_{ij}) < ...  \right) &\leq \P_f \left( \frac{1}{m_s} \sum_{j=1}^{m_s} \xi^{(s)}_{ij} \psi_{lk}(\zeta^{(s)}_{ij}) < ...  \; \text{ and } \; \|{\xi}_i^{(s)}\|_\infty \leq  \sqrt{2c\log N} \right)  \\ & \quad \quad + \P \left( \|{\xi}_i^{(s)}\|_\infty > \sqrt{2c\log N} \right) \\ 
	% 	&\leq 
	% \end{align*}
	Putting everything together, we obtain that
	\[
	U^{(s)}_{i,lk} - \E(U^{(s)}_{i,lk} \mid X^{(s)}_i) \leq 2 (2c\log N)^{3/2} \left(\frac{1}{\sqrt{m_s}} + \frac{1}{3}\|\psi\|_\infty  \frac{2^{l/2}}{m_s}\right),
	\]
	with probability $1- \kappa/N^c$ with $\kappa >0$. 
	
	%Here the noise term dominates over the first summand and can be ignored for large $N$.

\end{proof}

\begin{lemma}\label{lemma:MSE-of-non-private-estimator}
	Consider $\{U^{(s)}_{i,lk}\}_{i\in[n]}$ as defined in~\eqref{eq:def-U_il} and let $\widehat f^{(s)}_{lk}$ denote the quantity $ \frac{1}{n} \sum_{i=1}^n U^{(s)}_{i,lk}$. 

	It holds that
	\begin{align}
		\E(\widehat f^{(s)}_{lk} -f_{lk})^2 
		\lesssim \frac{1}{m_sn_s} + \frac{2^{-l(2\alpha+1)}}{n_s}.
	\end{align}
\end{lemma}
\begin{proof}[Proof of Lemma \ref{lemma:MSE-of-non-private-estimator}]
	Note
	$\E(\widehat f^{(s)}_{lk} -f_{lk})^2 = \mathrm{Var}(\widehat f^{(s)}_{lk}) = \frac{1}{n_s}\mathrm{Var}( U^{(s)}_{1,lk}) $. Next using properties of conditional expectations and variance we have
	$$
	\mathrm{Var}( U^{(s)}_{1,lk}) = \mathrm{Var}( \E (U^{(s)}_{1,lk}\mid X_i^{(s)})) + \E(\mathrm{Var}(  U^{(s)}_{1,lk}\mid X_i^{(s)}))
	$$
	Note that since $X \in \cH_\alpha(R)$ we have  $\E(U^{(s)}_{i,lk}\mid X^{(s)}_i) = \E(X^{(s)}_i(\zeta^{(s)}_{ij})\psi_{lk}(\zeta^{(s)}_{ij})\mid X^{(s)}_i) = \int X^{(s)}_i(t)\psi_{lk}(t) dt \lesssim 2^{-l(\alpha+1/2)}$ which implies $\mathrm{Var}( \E (U^{(s)}_{1,lk}\mid X_i^{(s)})) \lesssim 2^{-l(2\alpha+1)}$.
	
	Also note that $\mathrm{Var}(  U^{(s)}_{1,lk}\mid X_i^{(s)}) = \frac{1}{m_s} 	\mathrm{Var}(Y^{(s)}_{ij} \psi_{lk}(\zeta^{(s)}_{ij}) \mid X_i^{(s)})$ using the fact that $X_i$ is uniformly bounded by a constant (which implies that the mean of $Y^{(s)}_{ij}$ is also uniformly bounded) we have that $\E(\mathrm{Var}(  U^{(s)}_{1,lk}\mid X_i^{(s)})) \lesssim 1/m_s$.
	Putting everything together we have the desired bound.
\end{proof}

\subsection{Auxiliary lemmas}

The following lemma is a well-known form of Bernstein's inequality for bounded random variables. 

\begin{lemma}\label{lemma:bernstein}
	If \( V_1, \dots, V_n \) are independent bounded random variables such that \( \mathbb{E}[V_i] = 0 \), \( \mathbb{E}[V_i^2] \leq 1 \), and \( \max_i \sup |V_i| \leq M \), for some constant $M > 0$. 
	
	Then,
	\[
	\mathbb{P}\left(\left|n^{-1} \sum_{i=1}^n V_i\right| > \lambda \right) 
	\leq 2 \exp\left(- \frac{n\lambda^2}{2\left(1 + \frac{M \lambda}{3}\right)}\right).
	\]
\end{lemma}	

The following lemma is a standard result, see e.g. Chapter 9 in \cite{johnstone2019manuscript}.

\begin{lemma}\label{lem:holder-norm-sup-norm}
	There exists a constant $C_{\alpha,R}$ such that $\|f\|_\infty \leq C_{\alpha,R}$ for all $f \in \cH^\alpha(R)$ with $\alpha >1/2$.
\end{lemma}

\section{Proofs for the common design upper bound}
This appendix provides the detailed proofs, assumptions, and auxiliary results that support the analysis of the local polynomial estimator used in the main text.

\subsection{Proof of Theorem \ref{thm:common-design-upper-bound}}
\subsubsection{Aggregated mean estimator}

%For the analysis, we rely on the following assumptions:
%\begin{itemize}
%	\item [(LP1)] \( \lambda_{\min}(B_{m,x}) \geq \lambda_0 \) for \( n \geq n_0 \) and \( x \in [0,1] \), where \( B_{m,x} \) is the weighted design matrix used in the local polynomial estimator.
%	\item [(LP2)] For some constants \( a, \lambda_0 \), and any measurable set \( A \), we have
%	\[
%	\frac{1}{m} \sum_{j=1}^m \mathbb{I}(\zeta_{j} \in A) \leq a \max\left( \text{Leb}(A), \frac{1}{m} \right)
%	\]
%	with probability approaching 1. This assumption ensures a uniform spread of design points \( \zeta_{j} \) over the interval.
%	\item [(LP3)] The kernel \( K \) used in the estimator satisfies \( \operatorname{supp}(K) \subseteq [-1,1] \) and \( \|K\|_\infty < \infty \).
%\end{itemize}

For the $s$th server, $s\in[S]$,  $j \in [m]$ the privatized mean \(\bar{Y}_{j}^{P,s}\) at each point \( \zeta_{j} \) is given by:
\[
\bar{Y}_{j}^{P,s} = \frac{1}{n} \sum_{i=1}^n \left[ Y^{(s)}_{i,j} \right]_{\tau} + W^{(s)}_{j},
\]
where \( W^{(s)}_j \sim \mathcal{N}\left( 0, \frac{4 \tau^2 m\log (2/\delta)}{n_s^2 \varepsilon_s^2} \right) \) is a noise term added to ensure differential privacy, and the clipping parameter \( R = \sqrt{2 \log n_s} \) is used to bound the data.

The aggregated estimator is given by
$$
\bar{Y}_{j}^{P} = \sum_{s=1}^S w_s \bar{Y}_{j}^{P,s} 
$$
The weights are chosen as
$$
w_s = \frac{u_s}{\sum_{s=1}^Su_s}\quad \text{where} \quad u_s = m_0^{-1}n^2_s\varepsilon^2_s \wedge n_s.
$$
\subsubsection{Bias and Variance Analysis of the Local Polynomial Estimator}
Our estimator \(\hat{f}^P(x)\) is given by an average of $B$ many local polynomial estimator \(\{\hat{f}^P_b(x)\}_{b\in[B]}\). Our data is divided based on $\{\zeta_{j}\}_{j\in[m]}$ into $B$ many disjoint groups $\{G_b\}_{b\in [B]}$ each of size $m_0$.
The local polynomial estimator \( \widehat{f}^P_b(x) \) for the $b$th group is defined as:
\[
\hat{f}^P_b(x) = \sum_{j\in G_b} \bar{Y}^P_j W_{b,j}^{*}(x),
\]
where \( W_{b,j}^{*}(x) \) are the weights, given by:
\[
W_{b,j}^{*}(x) = \frac{1}{m_0 h} V^\top(0) B_{b,x}^{-1} V\left( \frac{\zeta_{j} - x}{h} \right) K\left( \frac{\zeta_{j} - x}{h} \right).
\]
and
\[
B_{b,x} = \frac{1}{m_0 h} \sum_{j\in G_b} V\left( \frac{\zeta_{j} - x}{h} \right) V\left( \frac{\zeta_{j} - x}{h} \right)^\top K\left( \frac{\zeta_{j} - x}{h} \right)
\]
To analyze the performance of \( \widehat{f}^P(x) \), we decompose the estimation error into bias and variance components. The choice of bandwidth \( h \) affects both components, allowing us to achieve the minimax risk rate by balancing these terms.

\paragraph{Bias}

The bias of the estimator \( \widehat{f}^P(x) \) is given by:
\[
bias(x) = \mathbb{E} \widehat{f}^P(x) - f(x) = \frac{1}{B}\sum_{b=1}^B bias_b(x)
\]
where $bias_b(x) = \mathbb{E} \widehat{f}^P_b(x) - f(x)$.
\begin{equation}\label{eq:bias-common-design}
	bias_b(x) = \mathbb{E} \widehat{f}^P_b(x) - f(x)
	= \sum_{j\in G_b} f(\zeta_{j}) W_{b,j}^*(x) - f(x) + \sum_{j\in G_b} \left( \mathbb{E} (\bar{Y}_{j}^P) - f(\zeta_{j}) \right) W_{b,j}^*(x).
\end{equation}
Lets look at the first term, using a Taylor expansion of \( f \) around \( x \), we have:
\begin{align*}
	\sum_{j\in G_b} f(\zeta_{j}) W_{b,j}^*(x) - f(x) &=  \sum_{j\in G_b}(f(\zeta_{j}) -f(x)) W_{b,j}^*(x) \\
	&= \sum_{j\in G_b} \frac{1}{p!}\left(f^{(p)}(x + \xi_j(\zeta_{j}-x)) -f^{(p)}(x)\right) W_{b,j}^*(x)
\end{align*}
where \( \xi_j \in [0,1] , p = \lfloor \alpha \rfloor \) where we have used the fact that $\sum_{j=1}^{m_0} W^*_{l,j}(x) =1$ and $\sum_{j\in G_b} (\zeta_{j}- x)^k W_{l,j}^*(x)= 0$ for $k\in[p]$ from Proposition \ref{prop:poly-interpolation}. Under Assumption (LP1) and Lemma \ref{lemma:properties-of-weights} we can bound the above term as:
\begin{align*}
	\sum_{j\in G_b} f(\zeta_{j}) W_{b,j}^*(x) - f(x) &\leq \sum_{j\in G_b} \frac{1}{p!} C_R|\zeta_{j} -x|^\alpha |W^*_{b,j}(x)|\\
	&= \sum_{j \in G_b} \frac{1}{p!} C_R|\zeta_{j} -x|^\alpha |W^*_{b,j}(x)| \mathbbm{1}(|\zeta_{j} - x| \leq h)\\
	&\leq C_R \sum_{j \in G_l} \frac{1}{p!} h^\alpha|W_{b,j}^*(x)| \leq \frac{C_R C_fh^\alpha}{p!}
\end{align*}
where $C_R$ is the lipchitz constant for the $p$th derivative of $f$ and depends on the holder radius $R$.

We next look at the second term in~\eqref{eq:bias-common-design}
\begin{equation}\label{eq:global-bias-to-server-bias}
	\mathbb{E}\bar{Y}_{j}^P - f(\zeta_{j}) = \sum_{s=1}^S w_s(\mathbb{E}\bar{Y}_{j}^{P,s} - f(\zeta_{j}) )
\end{equation}
Using the fact that noise $W$'s are mean zero we have that
\[
\mathbb{E}\bar{Y}_{j}^{P,s} - f(\zeta_{j}) = \E[Y_{1j}^{(s)}]_\tau -f(\zeta_{j})
\]
We can decompose the above bias as
\begin{align*}
	\E[Y_{1j}^{(s)}]_\tau -f(\zeta_{j}) &=\E \left(\E\left[[Y_{1j}^{(s)}]_\tau \mid X_1^{(s)}\right] - X_1^{(s)}(T_{1j})\right)\\
	&\leq \E \left|\E\left[[Y_{1j}^{(s)}]_\tau \mid X_1^{(s)}\right] - X_1^{(s)}(T_{1j})\right|
\end{align*}
Next conditional of $X_1^{(s)}$ we have that $Y_{1j}^{(s)}$ is gaussian with mean $X_1^{(s)}(T_{1j})$ and variance $1$. Hence using Lemma \ref{lemma:bias-truncated-gaussians} we have that
\[
\left|  \E\left([Y_{1j}^{(s)}]_\tau \mid X_1^{(s)}\right)-X_1^{(s)}(\zeta_{j})\right| \leq 4 |X_1^{(s)}(T_{1j})| e^{\frac{1}{2} (\tau - |X_1^{(s)}(T_{1j})|)^2}
\]
Using the fact $X \in \cH^\alpha(R)$ we have that $|X_1^{(s)}(T_{1j})| \leq \|X_1^{(s)}\|_\infty \leq C_{\alpha,R}$ and our choice of $\tau = \sqrt{2\log N} + C_{\alpha,R}$ we have that
\[
\left|  \E\left([Y_{1j}^{(s)}]_\tau \mid X_1^{(s)}\right)-X_1^{(s)}(\zeta_{j})\right| \leq 4 C_{\alpha,R}\frac{1}{N}
\]
which implies $\E[Y_{1j}^{(s)}]_\tau -f(\zeta_{j}) = O(1/N)$. Using~\eqref{eq:global-bias-to-server-bias} we have that $|\mathbb{E}\bar{Y}_{j}^P - f(\zeta_{j}) | = O(1/N)$.
Thus the second term in~\eqref{eq:bias-common-design} satisfies,
\[
\sum_{j \in G_b} \left| W_{b,j}^*(x) \right| \cdot O\left( \frac{1}{N} \right) \leq O\left( \frac{1}{N} \right),
\]
which can be safely ignored since \( O\left( \frac{1}{n} \right)\)is smaller than the order of the variance as we shall see later. Hence putting everything together
$$
bias(x) = O\left( h^\alpha + \frac{1}{N}\right).
$$ 

\paragraph{Variance}

The variance of \( \widehat{f}^P_B(x) \), $\sigma^2(x)$ is given by:
\begin{align*}
	\text{Var} \left( \widehat{f}^P_B(x) \right) 
	&= \text{Var} \left( \frac{1}{B} \sum_{b=1}^B \widehat{f}^P_b (x) \right) \\
	&= \text{Var} \left( \frac{1}{B} \sum_{b=1}^B  \sum_{j\in G_b} \bar{Y}_{j}^P W_{b,j}^*(x)   \right) 
\end{align*}
We can rewrite $ \bar{Y}_{j}^P$ as $\bar{Y}_{j}^P = \bar{Y}_{j}^\tau + W_j$ where
$$
\text{Var} \left( \widehat{f}^P_B(x) \right)  =  \bar{Y}_{j}^\tau = \sum_{s=1}^S w_s \bar{Y}_j^{\tau,s} \text{ with }  \bar{Y}_j^{\tau,s} = \frac{1}{n_s} \sum_{i=1}^{n_s} \left[ Y^{(s)}_{i,j} \right]_{\tau}, \quad\text{and}\quad W_j = \sum_{s=1}^S w_s W_j^{(s)}.
$$
Hence
\begin{equation}\label{eq:variance-decomposition-common-design}
	\text{Var} \left( \frac{1}{B} \sum_{b=1}^B  \sum_{j\in G_b} \bar{Y}^\tau_{j} W_{l,j}^*(x)   \right) + \text{Var} \left( \frac{1}{B} \sum_{b=1}^B  \sum_{j\in G_b} W_{j} W_{l,j}^*(x)   \right)
\end{equation}
We begin by bounding the first term~\eqref{eq:variance-decomposition-common-design}.
Suppose we have an upper bound on $\mathrm{Var} ( \bar{Y}^\tau_{j}) \leq \tilde{\sigma}^2$.
Set $A_b = \sum_{j\in G_b} \bar{Y}^\tau_{j} W_{b,j}^*(x) $ and use Lemma \ref{lemma:var-identity} to conclude
\begin{align*}
	\mathrm{Var}(A_b) &= \mathrm{Var}\left(\sum_{j\in G_b} \bar{Y}^\tau_{j} W_{b,j}^*(x) \right)\\
	&\leq \left( \sum_{j\in G_b} \sqrt{\mathrm{Var} ( \bar{Y}^\tau_{j} W_{b,j}^*(x))} \right)^2\\
	&\leq  \tilde{\sigma}^2 \left( \sum_{j\in G_b} | W_{b,j}^*(x))| \right)^2 \leq  \tilde{\sigma}^2C_f^2
\end{align*}
the last line follows from Lemma \ref{lemma:properties-of-weights}.
Hence again using Lemma \ref{lemma:var-identity}
\begin{align*}
	\text{Var} \left( \frac{1}{B} \sum_{b=1}^B \sum_{j\in G_b} \bar{Y}^{\tau}_{j} W_{b,j}^*(x)   \right) &= \mathrm{Var}\left(\frac{1}{B} \sum_{b=1}^B A_b\right)\\
	&\leq \frac{1}{B^2}\left(\sum_{b=1}^B\sqrt{  \tilde{\sigma}^2C_f^2}\right)^2 =  \tilde{\sigma}^2 C_f^2
\end{align*}
Next we need to find $ \tilde{\sigma}^2$ such that $\mathrm{Var} ( \bar{Y}^\tau_{j}) \leq \tilde{\sigma}^2$. In that direction we proceed as follows
\begin{align*}
	\mathrm{Var}(\bar Y_j^{\tau}) &= \mathrm{Var}\left(\sum_{s=1}^S w_s\bar Y_j^{\tau,s}\right) \\
	&= \sum_{s=1}^S w_s^2 \mathrm{Var}\left(\bar Y_j^{\tau,s}\right) \\
	&=  \sum_{s=1}^S w_s^2 \frac{1}{n_s}\mathrm{Var}\left([Y_{1j}^{(s)}]_\tau\right) 
\end{align*}
Next we use the fact that for any random variable $V$, $\mathrm{Var}([V]_\tau) \leq \mathrm{Var}(V)$ to obtain
\begin{align*}
	\mathrm{Var}(\bar Y_j^{\tau}) &\leq \sum_{s=1}^S w_s^2 \frac{1}{n_s}\mathrm{Var}\left(Y_{1j}^{(s)}\right) \\
	&=  \sum_{s=1}^S w_s^2 \frac{1}{n_s}\left( \mathrm{Var}\left(X_{1}^{(s)}(T_{1j})\right) +  \mathrm{Var}\left(\varepsilon_{1j}^{(s)}\right)\right)\\
	&\leq \sum_{s=1}^S w_s^2 \frac{1}{n_s} (C^2_{\alpha,R} + 1)
\end{align*}
Putting everything together we have that 
$$
\text{Var} \left( \frac{1}{B} \sum_{b=1}^B  \sum_{j\in G_b} \bar{Y}^\tau_{j} W_{l,j}^*(x)   \right)  \leq  (C^2_{\alpha,R} + 1)C_f^2 \sum_{s=1}^S  \frac{w_s^2}{n_s}
$$
Next we bound the second term in~\eqref{eq:variance-decomposition-common-design}
\begin{align*}
	\text{Var} \left( \frac{1}{B} \sum_{b=1}^B  \sum_{j\in G_b} W_{j} W_{b,j}^*(x)   \right) &= \frac{1}{B^2} \sum_{b=1}^B  \sum_{j\in G_b} \mathrm{Var}(W_{j}) (W_{b,j}^*(x))^2\\
	&=  \mathrm{Var}(W_{1}) \frac{1}{B^2} \sum_{b=1}^B  \sum_{j\in G_b}  (W_{b,j}^*(x))^2\\
	&=  \mathrm{Var}(W_{1}) \frac{1}{B} \frac{C_f^2}{m_0 h}
\end{align*}
where we used $ \sum_{j\in G_b}  (W_{b,j}^*(x))^2 \leq \max_{j} |W_{b,j}^*(x)| \sum_{j\in G_b}  |W_{b,j}^*(x)|$ and Lemma \ref{lemma:properties-of-weights} in the last line.
Next by independence of the noise terms we have that $ \mathrm{Var}(W_1) = \sum_{s=1}^S w_j^2 \mathrm{Var}(W^{(s)}_1)$.
Set $\delta' =\min_s \delta_s$ and $h \asymp m_0^{-1}$.
Using the properties of the noise term \(W^{(s)}_1\) added for privacy, we have:
\begin{align*}
	\sigma^2(x) &\leq (C^2_{\alpha,R} +1)C_f^2 \sum_{s=1}^S  \frac{w_s^2}{n_s}  + \frac{C^2_f}{Bm_0 h}\sum_{s=1}^Sw_s^2 \frac{4 \tau^2 m\log (2/\delta)}{n_s^2 \varepsilon_s^2}
\end{align*}
where the first term represents the intrinsic variance and the second term represents the additional variance due to privacy noise.

\subsubsection{Proof of the Minimax Rate}

Hence we can simplify the upper bound on variance as
\begin{align*}
	\sigma^2(x)
	&\lesssim \log(N)\log(2/\delta') \sum_{s=1}^S w_s^2 \frac{1}{n_s \wedge m_0^{-1}n_s^2\varepsilon_s^2}\\
	&\lesssim  \log(N)\log(2/\delta') \frac{1}{ \sum_{s=1}^S  n_s \wedge m_0^{-1}n_s^2\varepsilon_s^2}
\end{align*}
where we used the fact that $m/B \leq m_0$ and our choice of weights. Set $m_0 = D \wedge m$ where $D$ solves~\eqref{eq:rate-determining-relation:common-design} to obtain that $\sigma^2(x) \lesssim \log(N)\log(2/\delta') (D\wedge m)^{-2\alpha}$ , and we also know that $bias^2(x) = h^{2\alpha} = (m_0)^{2\alpha} = (D \wedge m)^{-2\alpha}$. Combining these two the minimax rate  is given by
$$
\log(N)\log(2/\delta') \left(D^{-2\alpha}\vee m^{-2\alpha}\right).
$$
\subsection{Auxiliary Lemmas and Propositions}
\begin{lemma}[Variance Identity]\label{lemma:var-identity}
	\[
	\text{Var} \left( \sum_{l=1}^L A_l \right) =  \sum_{l , l'} \text{Cov}(A_l, A_{l'}) \leq   \sum_{l ,l'} \sqrt{\text{Var}(A_l)}  \sqrt{\text{Var}(A_{l'})} \leq \left(\sum_{l=1}^L  \sqrt{\text{Var}(A_l)}\right)^2.
	\]
\end{lemma}
\begin{lemma}[Properties of Weights \( W_{m_0,j}^*(x) \)]\label{lemma:properties-of-weights}
	Under Assumptions (LP1)–(LP2), \( m_0 \geq \tilde m \), \( h \geq 1/2m_0 \), and \( x \in [0,1] \), the weights \( W_{m_0,j}^*(x) \) satisfy:
	\begin{itemize}
		\item [(i)] \( \sup_{j: |\zeta_{j} - x| \leq h} \left| W_{m_0,j}^*(x) \right| \leq \frac{C_f}{m_0 h} \),
		\item [(ii)] \( \sum_{j=1}^{m_0} \left| W_{m_0,j}^*(x) \right| \leq C_f \),
		\item [(iii)] \( W_{m_0,j}^*(x) = 0 \) if \( |\zeta_{j} - x| > h \).
	\end{itemize}
\end{lemma}
\begin{proof}
	For a proof see Lemma 1.3 of \cite{Tsybakov2008IntroductionTN}.
\end{proof}
\begin{lemma}[Polynomial Interpolation Property]\label{prop:poly-interpolation}
	For any polynomial \( Q \) of order \( l \), the weights \( W_{m_0,j}^*(x) \) satisfy:
	\[
	\sum_{j=1}^{m_0} Q(\zeta_{j}) W_{m_0,j}^*(x) = Q(x).
	\]
	in particular it implies that $\sum_{j=1}^{m_0} W^*_{m_0,j}(x) =1$ and $\sum_{j=1}^{m_0} (\zeta_{j}- x)^k W_{m_0,j}^*(x)= 0$ for $k\in[l]$.
\end{lemma}
\begin{proof}
	For a proof see Proposition 1.12 of \cite{Tsybakov2008IntroductionTN}.
\end{proof}
\begin{lemma}[Bias of truncated gaussians]\label{lemma:bias-truncated-gaussians}
	Let $W \sim N(\mu, \sigma^2)$, $\mu \in \mathbb{R}$  then the bias of $[W]^{\tau}_{-\tau}$ for $\tau >|\mu|$ is bounded by
	\[
	\left|\E[W]^{\tau}_{-\tau} -\mu\right| \leq 4|\mu| e^{-\frac{1}{2\sigma^2}(\tau - |\mu|)^2}
	\]
	
\end{lemma}
\begin{proof}
	For simplicity assume $\sigma = 1$.
	First note that $\E[W]^{\tau}_{-\tau} -\mu =\mathbb{E}[Z]^{\tau+\mu}_{-\tau+\mu} $ where \( Z \sim N(0, 1) \). Then, for \( \mu > 0 \) and \( \tau > \mu \),
	\[
	0 < \mathbb{E}[Z]^{\tau+\mu}_{-\tau+\mu} 
	\leq \mathbb{E}[Z]^{\tau-\mu}_{-\tau+\mu} + (2\mu)\mathbb{P}(\tau-\mu \leq Z \leq \tau+\mu) 
	\leq 4\mu e^{-\frac{1}{2}(\tau-\mu)^2}.
	\]
	Similarly, for \( \mu < 0 \) and \( \tau > \mu \), we obtain that
	\[
	0 > \mathbb{E}[Z]^{\tau+\mu}_{-\tau+\mu} 
	\geq 4\mu e^{-\frac{1}{2}(\tau+\mu)^2}.
	\]
	Combining these two cases together, we have for \( \mu \in \mathbb{R} \) and \( \tau \) large that
	\[
	\mathbb{E}[Z]^{\tau+\mu}_{-\tau+\mu} \leq 4|\mu|e^{-\frac{1}{2}(\tau - |\mu|)^2}.
	\]
\end{proof}

\section{Proofs related to the lower bound theorems}\label{sec:proofs-lower-bound}

\subsection{Proofs for the independent design lower bound (Theorem \ref{thm:lower-bound-independent-design})}\label{sec:proofs-lower-bound-independent-design}

% Proof of Lemma \ref{lem:functional-von-trees}
\subsubsection{Proof of Lemma \ref{lem:functional-von-trees}}\label{ssec:proof-functional-von-trees}
%\begin{proof}
	As a first step, we enlarge the collection of differentially private protocols to selectively include protocols which rely on both $Y^{(s)}$ and $X^{(s)}$ for estimating $f$, for a selection of servers. This step allows us to circumvent technical data processing arguments later on in the proof. To that extent, consider $\cS \subset [S]$, to be determined later.
	% \begin{equation*}
	% 	\cS = \{ s \in [S] : (D^{} n_s m_s) \wedge (m_s n_s^2 \varepsilon_s^2) > (D^{2\alpha} n_s) \wedge ( D^{2\alpha - 1} n_s^2 \varepsilon_s^2) \}.
	% \end{equation*}
	Write $\Delta X^{(s)} = (X_i^{(s)} - \E_f X_i^{(s)})_{i \in [n]}$. Let $\overline{\cM_{\mathbf{\varepsilon},\mathbf{\delta}}}$ denote the collection of all $(\mathbf{\varepsilon},\bf{\delta})$-differentially private protocols where the $s$-th transcript $T^{(s)}$ is allowed to depend on $(X^{(s)},Y^{(s)})$ if $s \in \cS$, whereas $T^{(s)}$ is allowed to depend on $(Y^{(s)}, \Delta X^{(s)})$ if $s \notin \cS$. Note that this does not impact our definition of the channel generating $T^{(s)}$ being differentially private otherwise, as the unit of privacy is the same for $(Y_{i\cdot}^{(s)},X_i^{(s)},T_{i\cdot}^{(s)})$ as it is for $(Y_{i\cdot}^{(s)},T_{i\cdot}^{(s)})$. The class $\overline{\cM_{\mathbf{\varepsilon},\mathbf{\delta}}}$ is a superset of $\cM_{\mathbf{\varepsilon},\mathbf{\delta}}$, so 
	\begin{align*}
		\inf_{\hat f \in \cM_{\mathbf{\varepsilon},\mathbf{\delta}}} \sup_{f \in \cH^\alpha(R)} \E_f \|\hat f - f\|_2^2 &\geq \inf_{\hat f \in \overline{\cM_{\mathbf{\varepsilon},\mathbf{\delta}}}} \sup_{f \in \cH^\alpha(R)} \E_f \|\hat f - f\|_2^2.
	\end{align*}

	Consider $D_*$ to be the solution to the equation \eqref{eq:rate-determining-relation:independent-design} and set $L =  \lceil \log_2 (D_*) \rceil$. Consider now the $L$-dimensional sub-model given by
	\begin{equation}\label{eq:holder-submodel}
	\cH^\alpha_{R,L} : \left\{f \in \cH_\alpha^{R} \,:\, f =\sum_{k=1}^{2^L}f_k \phi_{Lk},\, f_k \in [-2^{-L (\alpha +1/2)}R, 2^{-L (\alpha +1/2)}R]  \right\},
	\end{equation}
	where $\phi_k$ are the wavelet basis functions at resolution level $L$.  Let $\mu$ and $\nu$ denote dominating measures for $\P_f^{Y^{(s)},}$ and $\P_f^{(Y^{(s)}X^{(s)})}$, respectively, for $f \in \cH^\alpha_{R,L}$. Denote $f_L$ by the vector $\{f_{l}: l=1,\dots,L\}$. Let $\nabla_{f_L}$ denote the vector of partial derivatives with respect to the vector $f_L$. 

	% Bayes risk
	Since $\cH^\alpha_{R,L}$ is a subset of $\cH_\alpha^{R}$, we have that
	\begin{equation*}
		 \sup_{f \in \cH^\alpha(R)} \E_f \|\hat f - f\|_2^2 \geq  \sup_{f \in \cH^\alpha_{R,L}} \E_f \|\hat f - f\|_2^2,
	\end{equation*}
	where $f = \sum_{k=1}^{2^L} f_k \phi_{Lk}$. Furthermore, for any probability distribution $\pi$ on $\cH^\alpha_{R,L}$, the latter is futher lower bounded by
	\begin{equation*}
		 \int_{f_L} \E_f \|\hat f - f\|_2^2 \pi(f_L) df_L.
	\end{equation*}
	Note that $\E_f\|\hat f -f\|_2^2 \geq \E_f \left(\sum_{k=1}^{2^L} (\hat f_{k}-f_{k})^2\right)$, where $\hat{f}_k = \int \hat{f} \phi_{Lk}$. 
	For getting a lower bound on the minimax risk we will use a multivariate version of the Van-Trees inequality due to \cite{gill1995applications} (Theorem 1), which bounds the $\ell^2$-risk by
	\begin{equation*}%\label{eq:van-trees-inequality}
		\int_{f_L}\E\left(\sum_{k=1}^{2^L} (\hat f_{k}-f_{k})^2\right) \lambda(f_L) df_L \geq \frac{2^{2L}}{\int\mathrm{Tr}(I^{T}_{f_L})\pi(f_L)df_L + J(\pi)},
	\end{equation*}
	where $I^{T}_{f_L}$ is the sub-model Fisher information associated with the transcript $T = (T^{(1)},\dots,T^{(S)})$:
	\begin{equation*}
		I^{T}_{f_L}= \E_f \left[\nabla_{f_L} \log \left( \frac{d \P^{T}_f}{d \mu'} \right)\right] \left[\nabla_{f_L} \log \left( \frac{d \P^{T}_f}{d \mu'} \right)\right]^\top,
	\end{equation*}
	and  $\pi(f_L) = \prod_{k = 1}^{2^L} \pi_k(f_{k})$ is a prior for the parameter $f_L$ and $J(\pi)$ is the Fisher information associated with the prior $\pi$:
	\[
	J(\pi) = \sum_{k=1}^{2^L} \int \frac{\pi'_k(f_{k})^2}{\pi_k(f_{k})} df_{Lk}.
	\]
	The independence of the data in the servers implies independence of the transcripts, which yields
	\begin{equation*}
		\nabla_{f_L} \log \left( \frac{d \P^{T^{}}_f}{d \mu'} \right) = \sum_{s=1}^S \nabla_{f_L} \log \left( \frac{d \P^{T^{(s)}}_f}{d \mu'} \right).
	\end{equation*}
	Let $Z^{(s)}$ denote either $(Y^{(s)},\Delta X^{(s)})$ if $s \in \cS^c$ or $(Y^{(s)},X^{(s)})$ if $s \in \cS$ and let $\eta$ denote $\mu$ or $\nu$ accordingly. By Bayes rule, it follows that
	\begin{align*}
		\nabla_{f_L} \log \left( \frac{d \P^{T^{(s)}}_f}{d \mu'} \right) &= \frac{ \int \frac{d\P^{T^{(s)}|Z^{(s)}}}{d \mu'} \nabla_{f_L} \frac{d \P^{Z^{(s)}}_f}{d \eta} d\eta  }{\frac{d \P^{T^{(s)}}_f}{d \mu'}} \\ 
		&= \E_f \left[ \nabla_{f_L} \log \left( \frac{d \P^{Z^{(s)}}_f}{d \mu} \right) \bigg| T^{(s)}  \right].
	\end{align*}
	The identity \eqref{eq:fisher-information-Ys} follows. The data generating mechanism forms a Markov chain; $f \to X^{(s)} \to Y^{(s)}$, where the joint distribution degenerates as
	\begin{equation*}
		d\P_f^{(X^{(s)},Y^{(s)})} = d\P_f^{(X^{(s)})} d\P^{(Y^{(s)}|X^{(s)})}.
	\end{equation*}
	Hence, 
	\begin{equation*}
		\E_f \left[ \nabla_{f_L} \log \left( \frac{d\P_f^{(X^{(s)},Y^{(s)})}}{d\mu \times \nu} \right)  \bigg| T^{(s)} \right] = \E_f \left[ \nabla_{f_L} \log \left( \frac{d\P_f^{(X^{(s)})}}{d\nu} \right)  \bigg| T^{(s)} \right]
	\end{equation*}
	which implies that the Fisher information corresponding to the joint distribution $Z^{(s)} = (Y^{(s)},X^{(s)})$ equals 
	\begin{equation*}
		I^{X^{s}|T^{(s)}}_f = \E_f \, \E_f \left[ \nabla_{f_L} \log \left( \frac{d\P_f^{(X^{(s)})}}{d\nu} \right)  \bigg| T^{(s)} \right] \E_f \left[ \nabla_{f_L} \log \left( \frac{d\P_f^{(X^{(s)})}}{d\nu} \right)  \bigg| T^{(s)} \right]^\top.
	\end{equation*}
	Let $X^{(s)}$ follow the dynamics of \eqref{eq:holder-smooth-X-generation}. By the orthogonality of the $\phi_{lk}$'s, the wavelet transformation of $X_i^{(s)}$;
	\begin{equation}\label{eq:observational-model-Xs-sufficient}
		X_{ik}^{(s)} = \int X_i^{(s)}(t) \phi_{Lk}(t) dt \overset{d}{=} f_l + 2^{-L(\alpha + 1/2) } (1/2 - B_{ik}^{(s)}) \quad \text{where} \quad B_{ik}^{(s)} \overset{iid}{\sim} \text{Beta}(2,2),
	\end{equation}
	with $k = 1,\ldots,2^L$, is sufficient for the observations $X^{(s)}$ and submodel $\cH^\alpha_{R,L}$. By the Neyman-Pearson factorization lemma,
	\begin{equation*}
		\E \left[ \nabla_{f_L} \log \left( \frac{d\P_f^{(X^{(s)})}}{d\nu} \right)  \bigg| T^{(s)} \right] = \E \left[ \nabla_{f_L} \log \left( \frac{d\P_f^{(X^{(s)}_L)}}{d\lambda} \right)  \bigg| T^{(s)} \right],
	\end{equation*}
	where $X^{(s)}_L = \{X_{ik}^{(s)}: k = 1,\ldots,2^L\}$ and $\lambda$ is the Lebesgue measure on $\R^{2^L}$. 
	
	Similarly, since for the observations $(Y^{(s)}_{ij},X_i^{(s)} - \E_f X_i^{(s)})$ the distribution of $X_i^{(s)} - \E_f X_i^{(s)}$ is stationary in $f$, the statistics $\tilde{Y}^{(s)}_{ij} = Y^{(s)}_{ij}- [X_i^{(s)} - \E_f X_i^{(s)}](\zeta^{(s)}_{ij})$ are sufficient for $f$;
	\begin{equation}\label{eq:score-determining-Ys}
		\tilde{Y}^{(s)}_{ij} = f(\zeta^{(s)}_{ij}) + \xi_{ij}^{(s)},
	\end{equation}
	and for $\tilde{Y}^{(s)} = (\tilde{Y}^{(s)}_{ij})_{i\in[n],j\in[m_s]}$, we obtain
	\begin{equation*}
		I_f^{Y^{(s)}|T^{(s)}} = \E_f \, \E_f \left[ \nabla_{f_L} \log \left( \frac{d\P_f^{(\tilde{Y}^{(s)}|T^{(s)})}}{d\mu} \right)  \bigg| T^{(s)} \right] \E_f \left[ \nabla_{f_L} \log \left( \frac{d\P_f^{(\tilde{Y}^{(s)}|T^{(s)})}}{d\mu} \right)  \bigg| T^{(s)} \right]^\top.
	\end{equation*}
	Let $\phi(\zeta^{(s)}_{ij}) = \left(\phi_l(\zeta^{(s)}_{ij})\right)_{l\in[L]}$. The score function for $\tilde{Y}^{(s)}$ is given in \eqref{eq:effective-score-Ys}. By linearity of the trace operation,
	\begin{equation*}
		\int \text{Tr} (I_{T^{(s)}}(f)) d\pi(f) = \sum_{s \in \cS} \int \text{Tr} (I_{T^{(s)}}(f)) d\pi(f) + \sum_{s \in \cS^c} \int \text{Tr} (I_{T^{(s)}}(f_L)) d\pi(f).
	\end{equation*}
	Now, setting 
	\begin{equation*}
		\cS = \left\{ s \in [S] \, : \, \text{Tr}(I^{X^{(s)}|T^{(s)}})  <  \text{Tr}(I^{Y^{(s)}|T^{(s)}}) \right\},
	\end{equation*}
	and noting that by standard data processing arguments it holds that 
	\begin{equation*}
		\text{Tr}(I^{X^{(s)}|T^{(s)}}) \leq \text{Tr}(I^{X^{(s)}}) \quad \text{ and } \quad \text{Tr}(I^{Y^{(s)}|T^{(s)}}) \leq \text{Tr}(I^{Y^{(s)}}),
	\end{equation*}
	(see e.g. \cite{zamir1998fisherinfo}).

	By taking as prior $\pi_k$ a rescaled version of the density $t \mapsto \cos^2(\pi t/2) \mathbbm{1}\{|t| \leq 1\}$ such that it is supported on $[-2^{-L(\alpha+1/2)}R,2^{-L(\alpha+1/2)}R]$, we obtain that prior is supported on $\cH^\alpha_{R,L}$. By a straightforward calculation for the Fisher information associated with the prior, we have that the minimax risk is lower bounded by
	\[
		\inf_{\hat f \in \cM_{\mathbf{\varepsilon},\mathbf{\delta}}}  \sup_{f \in \cH^\alpha_{R,L}} \E_f\|\hat f - f\|_2^2 \geq \frac{2^{2L}}{\sup_{f_L}\mathrm{Tr}(I_{T}(f_L)) + \frac{2^L\pi^2}{(2^{-L(\alpha+1/2)}R)^2}}.
	\]
	Combining this with the earlier identities of Lemmas \ref{lem:functional-von-trees} and \ref{lem:fisher-info-bounds} yields the (further) lower bound
	\begin{equation*}
		\frac{1}{ C \left( \sum_{s=1}^S  \min \left\{ 2^{-2L} m_s n_s^2 \varepsilon_s^2 , 2^{-L} m_s n_s, 2^{L(2\alpha - 1)} n_s^2 \varepsilon_s^2, 2^{2L\alpha} n_s \right\} + 2^{2\alpha L} \right)}
	\end{equation*}
	For any $D_* \geq 1$ and $1 \leq D \leq D_*$ such that 
	\begin{equation*}%\label{eq:rate-determining-relation:independent-design}
		D^{2\alpha}_* \geq \sum_{s=1}^S \min \left\{ D^{-1} n_s m_s, D^{-2}m_s n_s^2 \varepsilon_s^2, D^{2\alpha} n_s , D^{2\alpha - 1} n_s^2 \varepsilon_s^2  \right\},
	\end{equation*}
	the latter quantity is further lower bounded by
	\begin{align*}
		\frac{1}{ C \left( \sum_{s=1}^S \min \left\{ D^{-1} n_s m_s, D^{-2}m_s n_s^2 \varepsilon_s^2, D^{2\alpha} n_s , D^{2\alpha - 1} n_s^2 \varepsilon_s^2  \right\} + D_*^{2\alpha} \right)} \gtrsim D_*^{-2\alpha}.
	\end{align*}
	%\end{proof}

\subsubsection{Proof of Lemma \ref{lem:fisher-info-bounds}}\label{sssec:proof-lemma-fisher-info-bounds}
%\begin{proof}
	We provide a bound for each of the four terms in the minimum separately. The data processing inequalities capturing the loss of information due to privacy constraints follow from Lemma \ref{lemma:eigen-value-bound-for-fisher-info} and Lemma \ref{lem:trace-fisher-info-Xs-bounds}. Recall that 
\begin{equation*}
	I_f^{Y^{(s)}} = \E_f \, \E_f \left[ S_f^{Y^{(s)}} \big| T^{(s)} \right] \E_f \left[ S_f^{Y^{(s)}} \big| T^{(s)} \right]^\top,
\end{equation*}
	where $S_f^{Y^{(s)}}$ is given by $S_f^{Y^{(s)}} = \sum_{i=1}^n S_f^{Y^{(s)}_i}$ and
	\begin{equation}\label{eq:score-function-Ys}
		S_f^{Y^{(s)}_i} = \sum_{j=1}^m \left(Y_{ij}^{(s)} - \sum_{k=1}^{2^L}f_k \phi_l(\zeta^{(s)}_{ij})\right) \phi(\zeta^{(s)}_{ij}),
	\end{equation}	
	%Similarly, we write $S_f^{X^{(s)}_i} = (X_{il}^{(s)} - f_l)_{l\in [2^L]}$ for the influence function of $X^{(s)}_i$ and $S_f^{X^{(s)}} = \sum_{i=1}^n S_f^{X^{(s)}_i}$. 
	By a straightforward calculation, 
	\begin{equation*}
		\E_f [S_f^{Y^{(s)}}] [S_f^{Y^{(s)}}]^\top = m_s n_s I_{2^L},
	\end{equation*}
	so the bound $\lambda_{\max}(\E_f [S_f^{Y^{(s)}}] [S_f^{Y^{(s)}}]^\top) \leq m_s n_s$ in combination with Lemma \ref{lemma:eigen-value-bound-for-fisher-info} yields that $\text{Tr}(I_f^{Y^{(s)}}) \lesssim m_s n_s^2 \varepsilon_s^2$ and 
	\begin{equation*}
		\text{Tr}(I_f^{Y^{(s)}}) = 2^L m_s n_s.
	\end{equation*}

	The relationship \eqref{eq:observational-model-Xs-sufficient} defines a regular location-shift model, which yields that the Fisher information for $f$ in the observational model for $X^{(s)}$ is given by
	\begin{equation*}
		I^{X^{(s)}}_f = \E_f \, \E_f \left[ \nabla_f \log \left( \frac{d\P_f^{(X^{(s)}_L)}}{d\nu} \right)   \right]  \E \left[ \nabla_f \log \left( \frac{d\P_f^{(X^{(s)}_L)}}{d\lambda} \right)   \right]^\top = c n_s 2^{2L(\alpha + 1/2)} I_{2^L},
	\end{equation*}
	for some constant $c>0$. The lemma follows by combining the bounds for the four terms in the minimum.	
%\end{proof}

\subsection{Proof of the common design lower bound (Theorem \ref{thm:common-design-lower-bound})}\label{sec:proofs-lower-bound-common-design}

 The lower bound of the order $m^{-2\alpha}$ is the same as in the non-private case, see e.g. \cite{cai2011meanfunction}, Theorem 2.1. Hence, we are required to show to ``private part'' of the lower bound. Consider the $L+K$-th resolution level of an $A > \alpha$ smooth orthonormal wavelet basis $\phi_{lk}(x) = 2^{l/2}\phi(2^l x - k)$ for $l \in \{ l_0,l_0+1,\dots,$ and $k = 1,\ldots,2^{L+K}\}$, with $K \in \N$ such that a subset of the wavelets $\{ \phi_{k} \, : \, k =1,\dots,2^L \}$ satisfies $\phi_k = \phi_{(L+K) i_k}$ disjoint support.

	Define $f_{ \mu}$ with $ \mu = (\mu_1,\ldots,\mu_{2^L}) \in \R^{2^L}$ as
	\[
	f_{  \mu} = \sum_{k=1}^{2^L} \mu_k \phi_{k}.
	\]
	Let $ d \in \mathbb{N}$ such that
	$m = \lfloor d \cdot 2^L \rfloor$. Consider now the model
	\[
	Y^{(s)}_{ij} = f_{ \mu}(\zeta_j) + V_{ij}  + \xi_{ij}^{(s)}, \quad V_i(t) = \sum_{k}^{2^L} Z_{ik}^{(s)} u_k(t) \quad \xi_{ij}^{(s)} \sim N(0, 1),
	\]
	with $V_{ij} = V_{i}(\zeta_j)$, $u_k(t) = \left(2^{L/2} \|\phi\|_\infty \right)^{-1} \phi_{k}(t)$ and $Z_{ik}^{(s)}$ i.i.d. truncated standard normal $N_{[-1,1]}(0,1)$. 
	
	Note that the function $X_i = f_{ \mu} + V_i$ has mean $f_\mu$ and is in $\mathcal{H}^\alpha(R)$ for $R>1$, for small enough $\|\mu\|_2$ and large enough $m$. Write $ \phi_d := \left(\phi_1(\frac{j}{m})\right)_{j \in [d]}$ and $u = \frac{\phi_d}{2^{L/2}\|\phi\|_\infty}$. The vector $V_i = (V_{ij})_{i \in [n]}$ has covariance matrix proportional to $\Sigma := \operatorname{diag}(u u^T, \dots, u u^T)$. We denote by $ Y^{(s)}_i$ the vector $(Y^{(s)}_{ij})_{j \in[m]}$ and by $f_{ \mu}( T)$ the vector $(f_{ \mu}(\zeta_j))_{j \in[m]}$

	Wherever it is non-zero, the likelihood for the $i$-th individual satisfies
	\[
	p_{ \mu}(Y^{(s)}_i) \propto \exp \left( -\frac{1}{2} \left(  Y^{(s)}_i -f_{ \mu}( T) \right)^T \left( \Sigma + I_m \right)^{-1} \left(  Y^{(s)}_i - f_{ \mu}( T) \right)  \right),
	\]
	where the normalizing constant can be seen to not depend on $ \mu$.

	Writing $\zeta$ for the vector $(\zeta_1,\dots,\zeta_m)$, $f_{ \mu}( \zeta)$ can be written as
	\[
	f_{ \mu}( \zeta) = 
	\begin{pmatrix}
		\sum_k \mu_k \phi_k(\zeta_1) \\
		\vdots \\
		\sum_k \mu_k \phi_k(\zeta_m)
	\end{pmatrix}
	= 
	\begin{pmatrix}
		\phi_1(\zeta_1) & \cdots & \phi_N(\zeta_1) \\
		\vdots & \ddots & \vdots \\
		\phi_1(\zeta_m) & \cdots & \phi_N(\zeta_m)
	\end{pmatrix}
	\begin{pmatrix}
		\mu_1 \\
		\vdots \\
		\mu_{2^L}
	\end{pmatrix}
	= \Phi^T{\mu},
	\]
	where $\Phi$ is an $2^L \times m$ matrix. Hence, its score in $ \mu$ is given by
	\begin{equation}\label{eq:cd-score}
	S_f^{Y^{(s)}_i} := \Phi(\Sigma + I_m)^{-1} ( Y^{(s)}_i - f_{ \mu}( \zeta))
	\end{equation}
	Hence, the Fisher information for the $i$-th individual $- \E_f \frac{\partial^2}{\partial  \mu^2} \log p_{ \mu}(Y^{(s)}_i)$ is bounded above by
	\[
	% \E 
	 \Phi \left( \Sigma + I_m \right)^{-1} \Phi^\top.
	\]
	By Lemma \ref{lemma:bounding-tail-term-trace-attack-cd}, we have that
	\begin{equation*}
		\Phi \left( \Sigma + I_m \right)^{-1} \Phi^\top \preceq 2^L \|\phi\|^2_\infty I_{2^L}.
	\end{equation*}
	Write $S_f^{Y} = \sum_{i=1}^{n_s} S_f^{Y^{(s)}_i}$ and  
	\begin{equation*}
		\bm C_f(T) = \mathbb{E}\left[ S_f^{Y}\mid T\right]\mathbb{E}\left[ S_f^{Y}\mid T\right]^T.
	\end{equation*}
	By combining the previous inequality with Lemma \ref{lem:trace-processing} and \ref{lemma:eigen-value-bound-for-fisher-info}, we have that 
\begin{align}\label{eq:trace-upper-bound-cd}
\E	\text{Tr}(\bm C_f(T)) &\lesssim 2 n_s \sqrt{2^L} \varepsilon_s \sqrt{\E\text{Tr}(\bm C_f(T))} + n_s \delta_s \cdot 2^L (\sqrt{m_s n_s} \vee m_s) \log(1/\delta_s) + n_s \delta_s.
\end{align}
If $\sqrt{\text{Tr}(\E(\bm C_f(T)))} \leq \sqrt{2^L} n_s \varepsilon_s$, then there is nothing to prove. So assume instead that $\sqrt{\text{Tr}(\E(\bm C_f(T)))} \geq \sqrt{2^L} n_s \varepsilon_s$. Combining the above display with \eqref{eq:trace-upper-bound-cd}, we get
\begin{align*}
\sqrt{\E \text{Tr}(\bm C_f(T))}  &\lesssim 2 \sqrt{2^L} n_s \varepsilon_s + \delta_s \varepsilon_s^{-1} n_s^{-1} \sqrt{2^L} (\sqrt{m_s n_s} \vee m_s) \log(1/\delta_s) + \frac{\delta_s}{\sqrt{2^L} \varepsilon_s}.
\end{align*}
The assumption of small enough $\delta'$ ($\delta_s \log(1/\delta_s) \lesssim \frac{n_s^2 \varepsilon_s^2}{\sqrt{m_s n_s} \vee m_s}$), implies that the last two terms are $o(\sqrt{m_s} n_s \varepsilon_s)$ which yields that
$$
\text{Tr}(\E(\bm C_f(T))) \lesssim 2^L n_s^2 \varepsilon_s^2.
$$
We also have $I_f(\bm Z) = \mathrm{Var}(S_f(Z_1)) \lesssim 2^L I_{2^L}$ as already shown. Hence we have that $\mathrm{Tr}(\bm C_f(T)) \leq (2^L)^2 n_s$. Using the von-Trees construction of Lemma \ref{lem:functional-von-trees}, we have that 
$$
\sup_{f \in \cH_\alpha^{R}} \mathbb{E}_{f} \|\hat f - f\|_2^2 \gtrsim \frac{2^{2L}}{ \sum_{s=1}^S (2^L n_s^2 \varepsilon_s^2) \wedge (n_s 2^{2L}) + \pi^2 (2^L)^{2\alpha+2}}.
$$
Optimizing over $L$, the result follows.

\begin{lemma}\label{lemma:eigen-value-bound-for-fisher-info}
	Consider $	\Sigma = \operatorname{diag} \left( \underbrace{u u^T, \dots, u u^T}_{N \,\text{times}} \right)$ where $u = c_L \phi_d$, $\bm \phi_d := \left(\phi_1(\frac{j}{m})\right)_{j \in [d]}$, $c_{L} = \frac{1}{\sqrt{2^L}} \|\bm \phi\|^{-1}_\infty$.  and define 
	\begin{equation*}
		\Phi = \begin{pmatrix}
			\phi_1(\zeta_1) & \cdots & \phi_{2^L}(\zeta_1) \\
			\vdots & \ddots & \vdots \\
			\phi_1(\zeta_m) & \cdots & \phi_{2^L}(\zeta_m)
		\end{pmatrix}.
	\end{equation*}
	It holds that
	$$
	\Phi \left( \Sigma + I_m \right)^{-1} \Phi^T =   a_{L} I_{2^L} \preceq 2^L \|\phi\|^2_\infty I_{2^L}.
	$$
	where $a_{L} = \frac{\|\bm\phi_d\|_2^2}{1 + c_{L}^2 \|\bm\phi_d\|_2^2}$.
\end{lemma}
\begin{proof}
	We have
	\begin{align*}
			\Phi \left( \Sigma + I_m \right)^{-1} \Phi^T  &= \Phi \left( \operatorname{diag} \left( u u^T, \dots, u u^T \right)+ I_m \right)^{-1} \Phi^T\\
			&=  \Phi \left( \operatorname{diag} \left[ \left(I_d - \frac{u u^T}{1+ u^Tu}\right), \dots,  \left(I_d - \frac{u u^T}{1+ u^Tu}\right)\right] \right) \Phi^T\\
			&= \Phi\Phi^T - \frac{1}{1 + u^T u}\Phi\operatorname{diag} \left( uu^T, \ldots, uu^T\right)\Phi^T.\\
	\end{align*}
Leveraging the fact that the $\phi$'s have disjoint support, we have that
\begin{align*}
\Phi \Phi^T &= 
\begin{bmatrix}
	\sum_{j=1}^m \phi_1^2(\zeta_j) & \sum_{j=1}^m \phi_1(\zeta_j) \phi_2(\zeta_j) & \cdots \\
	\sum_{j=1}^m \phi_1(\zeta_j) \phi_2(\zeta_j) & \sum_{j=1}^m \phi_2^2(\zeta_j) & \cdots \\
	\vdots & \vdots & \ddots 
\end{bmatrix}\\
&= 
\begin{bmatrix}
	\sum_{j=1}^m \phi_1^2(\zeta_j) & 0 & \cdots \\
	0 & \sum_{j=1}^m \phi_2^2(\zeta_j) & \cdots \\
	\vdots & \vdots & \ddots 
\end{bmatrix}.
\end{align*}
	Recalling that $\phi_{k}(x) = 2^{L/2 + K/2}\phi(2^{L+K} x - k)$ and $ d \in \mathbb{N}$ such that $m = \lfloor d 2^L \rfloor$, we have that $ \sum_{j=1}^m \phi_k^2(\zeta_j) \leq 2\sum_{j = 1}^d \phi_1^2(\zeta_j)$. This implies that
\[
\Phi \Phi^T= 2\left( \sum_{j=1}^d \phi_1^2(\zeta_j) \right) I_{2^L}.
\]
Similarly, we have that 
\[
\Phi\operatorname{diag} \left( uu^T, \ldots, uu^T\right)\Phi^T \succeq  \begin{pmatrix}
	\bm \phi^T_d u u^T \bm \phi_d & 0 & 0 & \cdots \\
	0 & 	\bm \phi^T_d u u^T \bm \phi_d& 0 & \cdots \\
	\vdots & \vdots & \vdots & \ddots 
\end{pmatrix} = (u^\top \bm \phi_d)^2 I_{2^L}
\]
We have that \( u =c_L \phi_d \). Denote $c_L = (2^{L/2}\| \phi\|_\infty)^{-1}$ and note that $u^\top \phi_d = c_L \|\phi_d\|_2^2$ and $u^\top u = c_L^2 \|\phi_d\|_2^2$. The result now follows:
\begin{align*}
	\Phi \left( \Sigma + I_m \right)^{-1} \Phi^T &\preceq \left( \| \phi_d\|_2^2 - \frac{c_L^2 \| \phi_d\|_2^4}{1 + c_L^2 \| \phi_d\|_2^2} \right) I_{2^L}\\
&= \frac{\|\phi_d\|_2^2}{1 + c_L^2 \|\phi_d\|_2^2} I_{2^L}
\preceq \frac{1}{c_L^2} I_{2^L} \preceq 2^L \|\phi\|^2_\infty I_{2^L}.
\end{align*}
\end{proof}

\begin{lemma}\label{lemma:bounding-tail-term-trace-attack-cd}
    For $M \asymp  2^L(\sqrt{m_s n_s} \vee m_s)  \log(1/\delta_s)$ we have that
    \[
    2 M \delta_s + \int_M^\infty \P \left( |G_1| \geq t \right) dt + \int_M^\infty \P \left( |\breve{G}_1| \geq t \right)dt \lesssim  \delta_s 2^L(\sqrt{m_s n_s} \vee m_s)  \log(1/\delta_s)  +   \delta_s.
    \]
    where $S_f^{Y_i} := \Phi(\Sigma + I_m)^{-1} ( Y_i - f_{ \mu}( T))$, $S_f^{Y^{(s)}} := \sum_{i=1}^n S_f^{Y_i}$ and $G_i = \langle \mathbb{E}[S_f^{Y_i}|T],S_f^{Y_i} \rangle$ and $\breve{G}_i = \langle \mathbb{E}[S_f^{Y_i}|T],\breve{S}_f^{{Y}^{(s)}_i} \rangle$ with $\breve{S}_f^{{Y}^{(s)}_i}$ an independent copy ${S}_f^{{Y}^{(s)}_i}$. 
	%with the score functions as defined in~\eqref{eq:cd-score}.
\end{lemma}	
\begin{proof}[Proof of Lemma \ref{lemma:bounding-tail-term-trace-attack-cd}]
    We proceed along the same line as in Lemma\ref{lemma:bounding-tail-term-trace-attack}. We first analyze the term concerning $G_i =\langle \mathbb{E}[S_f^{Y^{(s)}}|T],S_f^{Y^{(s)}_i} \rangle$. Using Jensen's and the law of total probability, we have that for $i=1,\dots,n_s$,
    \begin{align*}
        \E e^{t |G_i|} \leq\E e^{t|\langle S_f^{Y^{(s)}},S_f^{Y^{(s)}_i} \rangle|}.
    \end{align*}
    Without loss of generality, consider the index $i$ to be equal to $1$. Note that we have that $\langle S_f^{Y^{(s)}},S_f^{Y^{(s)}_1} \rangle = \sum_{i=1}^{n_s} \langle S_f^{Y^{(s)}_i},S_f^{Y^{(s)}_1} \rangle $.
        
    Using our model we have that 
    $$
    S_f^{Y^{(s)}_i} = \Phi(\Sigma + I_{m_s})^{-1}(\bm \xi_{i} + V_i)
    $$
    This implies that $S_f^{Y^{(s)}_i} \sim N(0, \Phi(\Sigma + I_{m_s})^{-1} \Phi^T) =  N(0, a_{2^L}I_d)$, where $a_{2^L}$ is given by Lemma \ref{lemma:eigen-value-bound-for-fisher-info}.
    Using the linearity of expectation, we have that $\langle S_f^{Y^{(s)}},S_f^{Y^{(s)}_1}\rangle =\sum_{i=1}^{n_s} \langle \bm S_f( Z_i),S_f^{Y^{(s)}_1} \rangle $.
    We can split the sum into two terms 
    $$
    \langle S_f^{Y^{(s)}},S_f^{Y^{(s)}_1} \rangle = \|S_f^{Y^{(s)}_1}\|_2^2 + \sum_{i=2}^{n_s} \sum_{j = 1}^{m_s} \xi_{ij}\langle \bm S_f( Z_i),S_f^{Y^{(s)}_1} \rangle 
    $$
    Hence we have that
    \begin{align*}
        \E e^{t|\langle S_f^{Y^{(s)}},S_f^{Y^{(s)}_i} \rangle|} &\leq \E\left( e^{t \|S_f^{Y^{(s)}_1}\|_2^2} e^{t\left| \sum_{i=2}^{n_s}\langle \bm S_f( Z_i),S_f^{Y^{(s)}_1} \rangle  \right|}\right)\\
        &= \E\left( e^{t \|S_f^{Y^{(s)}_1}\|_2^2} \E \left(e^{t\left| \sum_{i=2}^{n_s}\langle \bm S_f( Z_i),S_f^{Y^{(s)}_1} \rangle  \right|}\mid  Z_1  \right)\right)
    \end{align*}
    Conditional on  $Z_1$, we have that $\sum_{i=2}^{n_s}\langle \bm S_f( Z_i),S_f^{Y^{(s)}_1} \rangle$ behaves as follows
    $$
    \sum_{i=2}^{n_s}\langle \bm S_f( Z_i),S_f^{Y^{(s)}_1} \rangle \mid Z_1  \sim N\left(0, (n_s-1)a_{2^L}\|S_f(Z_1)\|_2^2 \right)
    $$
    Hence, using the m.g.f of Gaussian, we have that
    \begin{align*}\label{eq:bound-mgf-cd}
        \E e^{t|\langle S_f^{Y^{(s)}},S_f^{Y^{(s)}_i} \rangle|} 
        &\leq 2 \E\left( e^{t \|S_f^{Y^{(s)}_1}\|_2^2}  e^{\frac{1}{2}t^2(n_s-1)a_{2^L}\|S_f^{Y^{(s)}_1}\|^2}\right)\\
        &= 2 \E\left(  e^{\left(t + \frac{1}{2}t^2(n_s-1)a_{2^L}\right)\|\bm S_f(Z_1)\|^2}\right)\\
        &= 2 \E\left(  e^{\left(t + \frac{1}{2}t^2(n_s-1)a_{2^L}\right)a_{2^L}\sum^{m_s}_{j=1}U_j^2}\right)\\
        &= 	 2 \E\left(  e^{\left(t + \frac{1}{2}t^2(n_s-1)a_{2^L}\right)a_{2^L} m_s U_1^2}\right)
    \end{align*}
    where we have used the fact that $S_f^{Y^{(s)}_i} \sim  N(0, a_{2^L}I_d)$ hence $\|S_f^{Y^{(s)}_i}\|_2 = a_{2^L} \sum_{i=1}^{m_s} U_i^2$ with $U_i$ as i.i.d standard Gaussians. The last expectation is bounded by a constant if $\left(t + \frac{1}{2}t^2(n_s-1)a_{2^L}\right)a_{2^L} m_s \leq \frac{1}{2}$ (using the fact that $U_1^2$ is $\chi^2_1$). Hence if $t\lesssim \frac{1}{m_s 2^L} \wedge \frac{1}{\sqrt{m_s n_s} 2^L}$ then $\E e^{t|G_1|} \leq c$.
    It follows that
    \begin{align*}
        \P(|G_1| > u) &\leq \P(e^{t|G_1| }\geq e^{tu})\\
        &\leq e^{-tu}\E e^{t|G_1|}\\
        &\leq 2e^c e^{-u\left( \frac{1}{m_s 2^L} \wedge \frac{1}{\sqrt{m_s n_s} 2^L}\right) }.
    \end{align*}
    This means that for $M \gtrsim 2^L(\sqrt{m_s n_s} \vee m_s)  \log(1/\delta_s)$, we obtain
    \begin{align*}
        \int_M^\infty \P \left( |G_i| \geq t \right) dt \lesssim 2e^{- \log(1/\delta_s)} = 2\delta_s.
    \end{align*}
    In a similar fashion we can show that $\int_M^\infty \P \left( |\breve G_i| \geq t \right) dt \lesssim 2\delta_s$.
\end{proof}

\subsection{Differential privacy data processing inequalities for the Fisher information}\label{sec:proofs-data-processing-inequalities}

\begin{lemma}\label{lem:trace-processing}
	Let $\bm C_f(T)$ denote the matrix
	\begin{equation}\label{eq:conditional_covariance_of_influence_function}
		\; \mathbb{E}_f\left[S_f^{Y^{(s)}}\mid T\right]\mathbb{E}_f\left[ S_f^{Y^{(s)}} \mid T\right]^T,
	\end{equation}	
	where $S_f^{Y^{(s)}} = \sum_{i=1}^{n_s} S_f^{Y^{(s)}_i} $ denotes the score function of $f$ for the observational model $Y^{(s)}$ as defined in \eqref{eq:score-function-Ys}.
	
	It holds that $\E_f \mathrm{Tr}(\bm C_f(T)) \leq 2^L m_s n_s$.
	
	Furthermore, assuming $\delta_s$ is such that $\delta_s \log(1/\delta_s) \lesssim  \left(\frac{n_s}{m_s} \wedge \sqrt{\frac{n_s}{m_s}}\right)\varepsilon_s^2L^{-1}$, it also holds that
	\[
	\E_f \mathrm{Tr}(\bm C_f(T)) \lesssim m_s n_s^2\varepsilon_s^2.
	\]
\end{lemma}

\begin{proof}[Proof of Lemma \ref{lem:trace-processing}]
	The first assertion follows by standard data processing arguments for the Fisher information, see e.g. \cite{zamir1998fisherinfo}. By Lemma 4.2 in \cite{cai2024optimal}, there exists a universal constant \(C>0\) such that
	\begin{align*}
		\E_f \mathrm{Tr}(\bm C_f(T)) \leq  & C n_s\varepsilon_s \sqrt{\E_f \mathrm{Tr}(\bm C_f(T))} \sqrt{\lambda_{\max}(\mathrm{Var}(S_f^{Y^{(s)}_i})} \\
		&+  2 n_s M \delta_s + n_s \int_M^\infty \P \left( |G_i| \geq t \right) dt + n_s \int_M^\infty \P \left( |\breve{G}_i| \geq t \right) dt,
	\end{align*}
with
	\begin{equation}\label{eq:Gi-joe}
		G_i = \left\langle \mathbb{E}_f[ S_f^{Y^{(s)}} \mid T ],\ S_f^{Y^{(s)}_i} \right\rangle \text{ and }\breve{G}_i = \left\langle \mathbb{E}_f[ S_f^{Y^{(s)}} \mid T ],\ \breve{S}_f^{Y^{(s)}_i} \right\rangle,
	\end{equation}
and \(\breve{S}_f^{Y^{(s)}_i}\) denotes the data set \(S_f^{Y^{(s)}_i}\) with its \(i\)th data point \(Y_i^{(s)}\) replaced by an independent copy \(\breve{Y}_i^{(s)}\).

Note that \(\mathrm{Var}(S_f^{Y^{(s)}_i}) = m_sI_L\),
\begin{equation*}
	\mathrm{Var}_f(S_f^{Y^{(s)}}) = \sum_{j=1}^{m_s} \left(Y_{ij}^{(s)} - \sum_{k=1}^{2^L} f_k \phi_l(\zeta^{(s)}_{ij})\right) \phi(\zeta^{(s)}_{ij})
\end{equation*}
hence \(\lambda_{\max}(\mathrm{Var}(S_f^{Y^{(s)}})) = m_s\). Lemma \ref{lemma:bounding-tail-term-trace-attack} yields that for \(M \asymp 2^L(\sqrt{n_s}m_s^{3/2} \vee m_s^2) \log(1/\delta_s)\), we have that
\[
2 M \delta_s + \int_M^\infty \P \left( |G_1| \geq t \right) dt + \int_M^\infty \P \left( |\breve{G}_1| \geq t \right)dt \lesssim  \delta_s 2^L (\sqrt{n_s}m_s^{3/2} \vee m_s^2)  \log(1/\delta_s) +   \delta_s.
\]
Using Lemma \ref{lemma:bounding-tail-term-trace-attack}, we have that
	\begin{align}\label{eq:trace-upper-bound}
	\E_f\text{Tr}(\bm C_f(T))&\lesssim 2 n_s\sqrt{m_s} \varepsilon_s\sqrt{\E_f\text{Tr}(\bm C_f(T))} + n_s \delta_s 2^L(\sqrt{n_s}m_s^{3/2} \vee m_s^2) \log(1/\delta_s) + n_s \delta_s.
	\end{align}
	If \(\sqrt{\text{Tr}(\E_f(\bm C_f(T)))} \leq \sqrt{m_s} n_s \varepsilon_s\), then there is nothing to prove. So assume instead that \(\sqrt{\text{Tr}(\E_f(\bm C_f(T)))} \geq \sqrt{m_s} n_s \varepsilon_s\). Combining the above display with \eqref{eq:trace-upper-bound}, we get
	\begin{align*}
		\sqrt{\E_f \text{Tr}(\bm C_f(T))}  &\lesssim {2 \sqrt{m_s} n_s \varepsilon_s} +  \delta_s \varepsilon_s^{-1}2^L(\sqrt{n_s}m_s \vee m_s^{3/2}) \log(1/\delta_s) + \frac{\delta_s}{\sqrt{m_s} \varepsilon_s}.
	\end{align*}
	The assumption of small enough \(\delta_s\) (\(\delta_s \log(1/\delta_s)\lesssim \left(\frac{n_s}{m_s} \wedge \sqrt{\frac{n_s}{m_s}}\right)\varepsilon_s^2 2^{-L}\)) implies that the last two terms are \(o(\sqrt{m_s} n_s\varepsilon_s)\) which yields the result.
\end{proof}

\begin{lemma}\label{lemma:bounding-tail-term-trace-attack}
	Let
	\begin{equation*}%\label{eq:Gi-joe}
		G_i = \left\langle \mathbb{E}_f[ S_f^{Y^{(s)}} \mid T ],\ S_f^{Y^{(s)}_i} \right\rangle, \quad \breve{G}_i = \left\langle \mathbb{E}_f[ S_f^{Y^{(s)}} \mid T ],\ \breve{S}_f^{Y^{(s)}_i} \right\rangle,
	\end{equation*}
and \(\breve{S}_f^{Y^{(s)}_i}\) denotes the data set \(S_f^{Y^{(s)}_i}\) with its \(i\)th data point \(Y_i^{(s)}\) replaced by an independent copy \(\breve{Y}_i^{(s)}\).
	
	For $M \asymp 2^L(\sqrt{n_s}m_s^{3/2} \vee m_s^2) \log(1/\delta_s)$ we have that
	\[
	2 M \delta_s + \int_M^\infty \P_f \left( |G_1| \geq t \right) dt + \int_M^\infty \P_f \left( |\breve{G}_1| \geq t \right)dt \lesssim  \delta_s 2^L(\sqrt{n_s}m_s^{3/2} \vee m_s^2)  \log(1/\delta_s) +   \delta_s.
	\]
\end{lemma}	

\begin{proof}[Proof of Lemma \ref{lemma:bounding-tail-term-trace-attack}]
	We first analyze the term with $G_i$ in the integrand. Using Jensen's inequality and the law of total probability, we have for \(i=1,\dots,n_s\),
	\begin{align*}
		\mathbb{E}_f \exp\!\left(t \left|G_i\right|\right) 
		&= \mathbb{E}_f \exp\!\left(t \left|\left\langle \mathbb{E}_f[ S_f^{Y^{(s)}} \mid T ],\ S_f^{Y^{(s)}_i} \right\rangle\right|\right) \\ 
		&= \int \int \exp\!\left(t \left|\left\langle \mathbb{E}_f[ S_f^{Y^{(s)}} \mid T = u ],\ S_f^{Y^{(s)}_i}(y_i) \right\rangle\right|\right) dP^{T \mid Y^{(s)}_i = y_i}(u) dP^{Y^{(s)}_i}(y_i) \\[-2pt]
		&\le \int \int \mathbb{E}_f \left[\exp\left(t \left|\left\langle S_f^{Y^{(s)}},\ S_f^{Y^{(s)}_i}(y_i) \right\rangle\right|\right) \mid T = u \right] dP^{T \mid Y^{(s)}_i = y_i}(u) dP^{Y^{(s)}_i}(y_i) \\[-2pt]
		&= \mathbb{E}_f \exp\!\left(t \left|\left\langle S_f^{Y^{(s)}},\ S_f^{Y^{(s)}_i} \right\rangle\right|\right).
	\end{align*}
	Without loss of generality, consider \(i=1\). we have
	\[
	\left\langle S_f^{Y^{(s)}},\ S_f^{Y^{(s)}_1} \right\rangle 
	= 
	\sum_{i=1}^{n_s}\,\left\langle S_f^{Y^{(s)}_i},\ S_f^{Y^{(s)}_1} \right\rangle 
	= 
	\sum_{i=1}^{n_s} \sum_{j=1}^{m_s}\,\xi_{ij}^{(s)}\,\left\langle \bm{\phi}(\zeta^{(s)}_{ij}),\ S_f^{Y^{(s)}_{1}} \right\rangle,
	\]
	where \(\xi_{ij}^{(s)} \overset{\text{i.i.d.}}{\sim} \mathcal{N}(0,1)\). We have
	\[
	\left\langle S_f^{Y^{(s)}},\ S_f^{Y^{(s)}_1} \right\rangle 
	= 
	\left\| S_f^{Y^{(s)}_1} \right\|_2^2 
	+ 
	\sum_{i=2}^{n_s} \sum_{j = 1}^{m_s} \xi_{ij}^{(s)}\,\left\langle \bm{\phi}(\zeta^{(s)}_{ij}),\ S_f^{Y^{(s)}_1} \right\rangle.
	\]
	Hence,
	\begin{align*}
		\mathbb{E}_f \exp\!\left(t \left|\left\langle S_f^{Y^{(s)}},\ S_f^{Y^{(s)}_1} \right\rangle\right|\right) 
		&\le 
		\mathbb{E}_f\!\left(\exp\!\left(t \left\| S_f^{Y^{(s)}_1} \right\|_2^2\right)\;\exp\!\left(t \left|\sum_{j=2}^{n_s}\left\langle S_f^{Y^{(s)}_j},\ S_f^{Y^{(s)}_1} \right\rangle\right|\right)\right)\\[-2pt],
	\end{align*}
	which in turn equals
	\begin{equation*}
		\mathbb{E}_f\!\left[\exp\!\left(t \left\| S_f^{Y^{(s)}_1} \right\|_2^2\right)\,  \mathbb{E}_f\!\left(\exp\!\left(t \left|\sum_{j=2}^{n_s}\left\langle S_f^{Y^{(s)}_j},\ S_f^{Y^{(s)}_1} \right\rangle\right|\right) \,\Bigg|\, \{\zeta^{(s)}_{ij}\},\, Y^{(s)}_1 \right)\right].
	\end{equation*}
	Conditionally on \(\{\zeta^{(s)}_{ij}\}\) and \(Y^{(s)}_1\), the sum \(\sum_{j=2}^{n_s}\left\langle S_f^{Y^{(s)}_j},\ S_f^{Y^{(s)}_1} \right\rangle\) is normal with mean zero, since
	\[
	\sum_{j=2}^{n_s}\sum_{j=1}^{m_s} \xi_{ij}^{(s)}\,\left\langle \bm{\phi}(\zeta^{(s)}_{ij}),\ S_f^{Y^{(s)}_1} \right\rangle \bigg| [\{\zeta^{(s)}_{ij}\},\, Y^{(s)}_1]
	\sim
	\mathcal{N}\!\left(0,\ \sum_{j=2}^{n_s}\sum_{j=1}^{m_s} \left\langle \bm{\phi}(\zeta^{(s)}_{ij}),\ S_f^{Y^{(s)}_1} \right\rangle^2 \right).
	\]
	Using the moment generating function of the normal distribution, we have
	\begin{equation}\label{eq:bound-mgf}
		\mathbb{E}_f \exp\!\left(t \left|\left\langle S_f^{Y^{(s)}},\ S_f^{Y^{(s)}_1} \right\rangle\right|\right) 
		\le 
		2\,\mathbb{E}_f\!\left[\exp\!\left(t \left\| S_f^{Y^{(s)}_1} \right\|_2^2\right)\,\exp\!\left(\tfrac{1}{2}\,t^2\,\sum_{j=2}^{n_s}\sum_{j=1}^{m_s}\left\langle \bm{\phi}(\zeta^{(s)}_{ij}),\ S_f^{Y^{(s)}_1} \right\rangle^2\right)\right].
	\end{equation}
	Next, 
	\[
	\sum_{j=2}^{n_s}\sum_{j=1}^{m_s}\left\langle \bm{\phi}(\zeta^{(s)}_{ij}),\ S_f^{Y^{(s)}_1} \right\rangle^2
	\le
	\sum_{j=2}^{n_s}\sum_{j=1}^{m_s}
	\left\|\bm{\phi}(\zeta^{(s)}_{ij})\right\|_2^2\,\left\|S_f^{Y^{(s)}_1}\right\|_2^2
	\le
	c_\phi^2\,m_s\,(n_s-1)\,2^L\,\left\|S_f^{Y^{(s)}_1}\right\|_2^2.
	\]
	Hence, as a consequence of \eqref{eq:bound-mgf},
	\begin{align}\label{eq:bound-mgf-2}
		\mathbb{E}_f \exp\!\left(t \left|\left\langle S_f^{Y^{(s)}},\ S_f^{Y^{(s)}_1} \right\rangle\right|\right) 
		&\le 
		2\,\mathbb{E}_f \!\left[\exp\!\left(t \left\| S_f^{Y^{(s)}_1} \right\|_2^2\right)\,\exp\!\left(\tfrac{1}{2}\,t^2\,c_\phi^2\,m_s\,(n_s-1)\,2^L\,\left\| S_f^{Y^{(s)}_1} \right\|_2^2\right)\right] \nonumber\\[-2pt]
		&= 
		2\,\mathbb{E}_f \!\left[\exp\!\left(\left(t + \tfrac{1}{2}\,t^2\,c_\phi^2\,m_s\,(n_s-1)\,2^L\right)\,\left\| S_f^{Y^{(s)}_1} \right\|_2^2\right)\right].
	\end{align}
	Setting $s = t + \tfrac{1}{2}\,t^2\,c_\phi^2\,m_s\,(n_s-1)\,2^L$, we proceed to bound \(\mathbb{E} \exp\!\left(s\,\left\| S_f^{Y^{(s)}_1} \right\|_2^2\right)\).
	Observe that
	\[
	\left\| S_f^{Y^{(s)}_1} \right\|_2^2
	= 
	\sum_{k=1}^{2^L}
	\left(\sum_{j=1}^{m_s}\xi_{1j}\,\phi_{Lk}(\zeta^{(s)}_{1j})\right)^{2}
	\le
	\sum_{k=1}^{2^L}
	\left(\sum_{j=1}^{m_s}\xi_{1j}^2\right)\,\left(\sum_{j=1}^{m_s}\phi_{Lk}^2(\zeta^{(s)}_{1j})\right)
	\le
	c_\phi^2\,m_s\,2^L\,\sum_{j=1}^{m_s}\xi_{1j}^2.
	\]
	Hence by independence and the m.g.f.\ bound for \(\xi_{1j}^2\),
	\[
	\mathbb{E}_f \exp\!\left(s\,\left\| S_f^{Y^{(s)}_1} \right\|_2^2\right)
	\le
	\mathbb{E}_f \exp\!\left(s\,c_\phi^2\,m_s\,2^L\,\sum_{j=1}^{m_s}\xi_{1j}^2\right)
	= 
	\left(\mathbb{E}_f \exp\!\left(s\,c_\phi^2\,m_s\,2^L\,\xi_{11}^2\right)\right)^{m_s}.
	\]
	One checks this is bounded by a constant as soon as $s \le \tfrac{1}{m_s^2\,2^L}$, which, recalling that $s = t + \tfrac{1}{2}\,t^2\,c_\phi^2\,m_s\,(n_s-1)\,2^L$, is the case when
	$t 
	\lesssim
	\tfrac{1}{m_s^2\,2^L} 
	\wedge 
	\tfrac{1}{\sqrt{n_s}\,m_s^{3/2}\,2^L}$.

	Within this regime,
	\(\mathbb{E}\exp\!\left(t\,|G_1|\right)\le e^c\) for some constant \(c>0\).  Consequently,
	\begin{align*}
		\P\left(|G_1| > u\right) 
		&\le 
		\P\left(\exp\!\left(t\,|G_1|\right) \ge \exp\!\left(t\,u\right)\right) \\
		&\le 
		\exp\!\left(-t\,u\right)\,\mathbb{E} \exp\!\left(t\,|G_1|\right)\\[-4pt]
		&\le 
		2\,e^c\,\exp\!\left(-\,u\,\left(\,\tfrac{1}{\sqrt{n_s}\,m_s^{3/2}\,2^L}\,\wedge\,\tfrac{1}{m_s^2\,2^L}\right)\right).
	\end{align*}
	Thus, setting 
	\[
	M \gtrsim 2^L\left(\sqrt{n_s}\,m_s^{3/2}\,\vee\,m_s^2\right)\log\!\left(\tfrac{1}{\delta_s}\right),
	\]
	one gets
	\[
	\int_M^\infty \P\left(|G_1| \ge t\right)\,dt 
	\lesssim 
	2\,\exp\!\left(-\log\left(\tfrac{1}{\delta_s}\right)\right)
	= 
	2\,\delta_s.
	\]
	
	We now turn to the term \(\int_M^\infty \P\left(|\breve{G}_1| \ge t\right)\,dt\). By a similar argument, observe
	\[
	\mathbb{E}_f \exp\!\left(t\,|\breve{G}_i|\right)
	= 
	\mathbb{E}_f \exp\!\left(t\,\left|\left\langle \mathbb{E}_f[ S_f^{Y^{(s)}} \mid T ],\ \breve{S}_f^{Y^{(s)}_i} \right\rangle\right|\right)
	\le 
	\mathbb{E}_f \exp\!\left(t\,\left|\left\langle S_f^{Y^{(s)}},\ \breve{S}_f^{Y^{(s)}_i} \right\rangle\right|\right).
	\]
	By symmetry,
	\[
	\left\langle S_f^{Y^{(s)}},\ \breve{S}_f^{Y^{(s)}_1} \right\rangle
	= 
	\sum_{j=1}^{n_s}\sum_{s=1}^{m_s} \xi_{ij}\,\left\langle \bm{\phi}(Y^{(s)}_{js}),\ \breve{S}_f^{Y^{(s)}_1} \right\rangle.
	\]
	The same argument as before yields
	\[
	\mathbb{E}_f \exp\!\left(t\,\left|\left\langle S_f^{Y^{(s)}},\ \breve{S}_f^{Y^{(s)}_1} \right\rangle\right|\right)
	\le 
	2\,\mathbb{E}_f\!\left[\exp\!\left(\tfrac{1}{2}\,t^2\,c_\phi^2\,m_s\,n_s\,2^L\,\left\| \breve{S}_f^{Y^{(s)}_1} \right\|_2^2\right)\right].
	\]
	As before, this remains bounded when 
	\(
	t 
	\lesssim 
	\tfrac{1}{m_s^2\,2^L}
	\wedge
	\tfrac{1}{\sqrt{n_s}\,m_s^{3/2}\,2^L}.
	\)
	Hence the same exponential tail bound applies;
	\[
	\int_M^\infty \P\left(|\breve{G}_1| \ge t\right)\,dt 
	\lesssim 
	2\,\exp\!\left(-\log \left(\tfrac{1}{\delta_s}\right)\right)
	= 
	2\,\delta_s
	\quad
	\text{for}
	\quad
	M \gtrsim 2^L\left(\sqrt{n_s}\,m_s^{3/2}\,\vee\,m_s^2\right)\log\!\left(\tfrac{1}{\delta_s}\right).
	\]
	Putting all pieces together yields the statement of the lemma.
\end{proof}

\begin{lemma}\label{lem:trace-fisher-info-Xs-bounds}
	Consider $X_{L;i}^{(s)} = (X_{1i}^{(s)},\dots,X_{2^Li}^{(s)})$ where $X_{ki}^{(s)}$ is defined as in \eqref{eq:observational-model-Xs-sufficient}. It holds that $\text{Tr}(I^{X^{(s)}}_f) \leq n_s 2^L$ and if $0 \leq \delta_s \leq  \left( \left({n_s}{2^{-L}} \wedge n_s^{1/2}{2^{-L/2}}  \right) \varepsilon_s^2 \right)^{1+p}$ for some $p>0$, we also have that 
	\begin{equation*}
		\text{Tr}(I^{X^{(s)}}_f) \leq C n_s^2 \varepsilon_s^2 
	\end{equation*}
	for a constant $C > 0$.
\end{lemma}

% Proof of Lemma \ref{}
\begin{proof}
	The bound $\text{Tr}(I^{X^{(s)}}_f) \leq n_s 2^L$ follows by the fact that conditional expectation contracts the $L_2$-norm. For the second statement, we start introducing the notations $\overline{X^{(s)}} = n_s^{-1} \sum_{i=1}^{n_s} X_{L;i}^{(s)}$ and
\[
G_i = \left \langle \E_0 \left[ n_s \overline{X^{(s)}} \mid T^{(s)} \right], X_{L;i}^{(s)} \right \rangle.
\]
For the remainder of the proof, consider versions of $X^{(s)}$ and $T^{(s)}$ defined on the same probability space, and we shall write as a shorthand
\[
\P^{s} \equiv \P^{(X^{(s)},T^{(s)})}_0 \text{ and } \E^{s} \equiv \E^{(X^{(s)},T^{(s)})}_0.
\]
For random variables $V,W$ defined on the same probability space, it holds that
\[
\E W \E[ W \mid V] = \E \E[ W \mid V] \E[ W \mid V],
\]
since \(W - \E[ W \mid V]\) is orthogonal to \(\E[ W \mid V]\). Combining this fact with the linearity of the inner product and conditional expectation, we see that
\[
\text{Tr}(I^{X^{(s)}}_f) = \E^{T^{(s)}}_0 \left\| \E_0[ n_s \overline{X^{(s)}} \mid T^{(s)} ]\right\|_2^2 = \underset{i=1}{\overset{n_s}{\sum}}  \E^{s} G_i.
\]
Define also
\[
\breve{G}_i = \left \langle \E_0[ n_s \overline{X^{(s)}} \mid T^{(s)} ], \breve{X}_{Li}^{(s)} \right \rangle,
\]
where \(\breve{X}_{Li}^{(s)}\) is an independent copy of \(X_{L;i}^{(s)}\) (defined on the same, possibly enlarged probability space) and note that \(\E^s \breve{G}_i = 0\). Write \(G_i^+ := 0 \vee G_i\) and \(G_i^- = - (0 \wedge G_i)\). We have
\begin{align*}
    \E^s G_i^+ &=  \int_0^\infty \P^s \left( G_i^+ \geq t \right) dt = \int_0^T \P^s \left( G_i^+ \geq t \right) + \int_T^\infty \P^s \left( G_i^+ \geq t \right)  \\
    &\leq e^{\varepsilon_s} \int_0^T \P^s \left( \breve{G}_i^+ \geq t \right) dt +T \delta_s + \int_T^\infty \P^s \left( G_i^+ \geq t \right) \\ 
    &\leq  \int_0^T \P^s \left( \breve{G}_i^+ \geq t \right)dt + 2 \varepsilon_s \int_0^T \P^s \left( \breve{G}_i^+ \geq t \right)dt +T \delta_s + \int_T^\infty \P^s \left( G_i^+ \geq t \right) \\ 
    &\leq  \int_0^\infty \P^s_0 \left( \breve{G}_i^+ \geq t \right)dt + 2 \varepsilon_s \int_0^\infty \P^s \left( \breve{G}_i^+ \geq t \right)dt +T \delta_s + \int_T^\infty \P^s \left( G_i^+ \geq t \right),
\end{align*}
Similarly, we obtain
\begin{align*}
    \E^s G_i^-  &\geq  \int_0^T \P^s \left( G_i^- \geq t \right)   \\
    &\geq e^{-\varepsilon_s} \int_0^T \P^s \left( \breve{G}_i^- \geq t \right) dt - T \delta_s  \\ 
    &\geq  \int_0^T \P^s \left( \breve{G}_i^- \geq t \right)dt - 2 \varepsilon_s \int_0^\infty \P^s \left( \breve{G}_i^- \geq t \right)dt - T \delta_s  \\
    &\geq \int_0^\infty \P^s \left( \breve{G}_i^- \geq t \right)dt - 2 \varepsilon_s \int_0^\infty \P^s \left( \breve{G}_i^- \geq t \right)dt - T \delta_s - \int_T^\infty \P^s \left( \breve{G}_i^- \geq t \right) dt.
\end{align*}
Putting these together with \(G_i = G_i^+ - G_i^-\), we get
\begin{align*}
    \E^s G_i &\leq \int_0^\infty \P^s \left( \breve{G}_i^+ \geq t \right)dt - \int_0^\infty \P^s \left( \breve{G}_i^- \geq t \right)dt + 2 \varepsilon_s \int_0^\infty \P^s \left( |\breve{G}_i| \geq t \right)dt \\ 
    &\;\;\;\;\;\;\;\;+ 2 T \delta_s + \int_T^\infty \P^s \left( G_i^+ \geq t \right) dt + \int_T^\infty \P^s \left( \breve{G}_i^- \geq t \right) dt \\
    &= \E^s \breve{G}_i + 2 \varepsilon_s \E^s | \breve{G}_i | + 2 T \delta_s + \int_T^\infty \P^s \left( G_i^+ \geq t \right) dt + \int_T^\infty \P^s \left( \breve{G}_i^- \geq t \right) dt.
\end{align*}
The first term in the last display equals \(0\). For the second term, observe that 
\[
\breve{G}_i \bigg| \left[T^{(s)},X^{(s)}\right] \sim N(0, \| \E_0[ n_s \overline{X^{(s)}} \mid T^{(s)} ] \|_2^2),
\]
so
\[
\E^s | \breve{G}_i | = \E^{X^{(s)},T^{(s)}} \E^{\breve{X}^{(s)}} | \breve{G}_i | = \E \| \E[ n_s \overline{X^{(s)}} \mid T^{(s)} ] \|_2 \leq \sqrt{\text{Tr}(I^{X^{(s)}}_f)},
\]
where the last inequality is Cauchy-Schwarz. To bound the terms
\[
\int_T^\infty \P^s \left( G_i^+ \geq t \right) dt + \int_T^\infty \P^s \left( \breve{G}_i^- \geq t \right) dt,
\]
we shall employ tail bounds, which follow after showing that \(G_i\) is \(\sqrt{2^L n_s}\)-sub-exponential. To see this, note that by applying Cauchy-Schwarz and Jensen's inequality followed by the law of total probability, we have that
\begin{align*}
    \E^s e^{t |G_i|} &= \E^s e^{t |\langle \E_0[ n_s \overline{X^{(s)}} \mid T^{(s)} ], X_{L;i}^{(s)} \rangle|} \\ 
    &\leq \E^s e^{\frac{t}{2} \left( \left\| \E_0[ n_s \overline{X^{(s)}} \mid T^{(s)} ] \right\|_2^2 + \left\|X_{L;i}^{(s)}\right\|_2^2 \right)} \\
    &\leq \sqrt{ \E_0 e^{{t}  \left\| \E_0[ n_s \overline{X^{(s)}} \mid T^{(s)} ] \right\|_2^2}} \sqrt{ \E_0 e^{{t} \left\|X_{L;i}^{(s)}\right\|_2^2}} \\
    &\leq \E^{}_0 e^{t |\langle  n_s \overline{X^{(s)}} , {X}_{Li}^{(s)} \rangle|},
\end{align*}
where the last equality follows from the fact that conditional expectation contracts the \(L_1\)-norm.
Next, we bound \(\E^{X^{(s)}}_0  e^{t |\langle  n_s \overline{X^{(s)}} , X_{L;i}^{(s)} \rangle|}\). By the triangle inequality and Cauchy-Schwarz,
\begin{align*}
    \E^{X^{(s)}}_0 e^{t |\langle  n_s \overline{X^{(s)}} , {X}_{Li}^{(s)} \rangle|} &\leq \sqrt{\E^{X^{(s)}}_0 e^{2t |\langle \sum_{k \neq i}^{n_s} X_{Lk}^{(s)} , {X}_{Li}^{(s)} \rangle|}} \sqrt{\E^{X_{L;i}^{(s)}}_0 e^{2t |\langle  {X_{L;i}^{(s)}} , {X}_{Li}^{(s)} \rangle|}}.
\end{align*}
The random variable \(\langle  {X_{L;i}^{(s)}} , {X}_{Li}^{(s)} \rangle\) is the inner product of centered beta$(2,2)$ distributed random variables, so 
\begin{align*}
    \E^{X_{L;i}^{(s)}}_0 e^{2t |\langle  {X_{L;i}^{(s)}} , {X}_{Li}^{(s)} \rangle|} \leq e^{2^{L-1}t}.
\end{align*}
By Lemma~\ref{lem:inner_product_ind_gaussians},
\begin{align*}
    \E^{X^{(s)}}_0 e^{2t |\langle \sum_{k \neq i}^{n_s} X_{Lk}^{(s)} , {X}_{Li}^{(s)} \rangle|} \leq e^{{c t^2(n_s-1)}2^L},
\end{align*} 
for some $c > 0$. By the fact that \(G_i^+ \leq |G_i|\) and Markov's inequality, 
\begin{align*}
    \P^s(G_i^+ > T) \leq \P^s(|G_i| > T) \leq e^{-tT}\E^s e^{t|G_i|}, \text{ for all } T,t>0.
\end{align*}
Combining this with the bound for the moment generating function derived above means that for \(\delta_s = 0\), the result follows from letting \(T \to \infty\). If \(\delta_s > 0\), take \(T = 32 (2^L \vee \sqrt{2^L n_s})  \log(1/\delta_s)\) to obtain that
\begin{align*}
    \int_T^\infty \P^s \left( G_i^+ \geq t \right) dt \leq  e^{ -  \log(1/\delta_s)}.
\end{align*} 
It is easy to see that the same bound applies to \(\int_T^\infty \P^s_0 \left( \breve{G}_i^- \geq t \right) dt\). We obtain that
\begin{align*}
    \underset{i=1}{\overset{n_s}{\sum}}  \E^{s} G_i &\leq 2 n_s \varepsilon_s \sqrt{\text{Tr}(I^{X^{(s)}}_f)} + 64 \delta_s (2^L \vee \sqrt{2^L n_s}) \log(1/\delta_s) + 2 n_s \delta_s.
\end{align*}
If \(\sqrt{\text{Tr}(I^{X^{(s)}}_f)} \leq n_s \varepsilon_s\), the lemma holds (there is nothing to prove). So assume instead that \(\sqrt{\text{Tr}(I^{X^{(s)}}_f)} \geq n_s \varepsilon_s\). We find that
\begin{align*}
    \sqrt{\text{Tr}(I^{X^{(s)}}_f)}  &\leq {2 n_s \varepsilon_s} +   64 \delta_s \frac{2^L \vee \sqrt{2^L n_s}}{n_s \varepsilon_s} \log(1/\delta_s) +  \frac{2}{\varepsilon_s}  \delta_s.
\end{align*}
Since \(x^p \log(1/x)\) tends to \(0\) as \(x \to 0\) for any \(p>0\), the result follows for \(\delta_s \leq  \left( \left(\frac{n_s}{2^L} \wedge \frac{n_s^{1/2}}{\sqrt{2^L}}  \right) \varepsilon_s^2 \right)^{1+p}\) for some \(p>0\) as this implies that the last two terms are \(O(n_s \varepsilon_s)\).

\end{proof}

\subsubsection{Auxiliary lemmas}

The following lemmas are folkore that involve straightforward calculations, but included for completeness.

\begin{lemma}\label{lem:chi-square-subgaussianity_lemma}
	Let $Z$ be $N(0,1)$, $0 \leq \lambda \leq 1/4$. Then, 
	\begin{equation*}
	\E e^{\lambda (Z^2-1)} \leq e^{2\lambda^2}.
	\end{equation*}
	\end{lemma}
	\begin{proof}
	Using the change of variables $u = z \sqrt{1- 2\lambda}$,
	\begin{align*}
	\E e^{\lambda (Z^2-1)} &= \frac{1}{\sqrt{2\pi}} \int e^{\lambda (z^2 - 1)} e^{-\frac{1}{2} z^2} dz  \\
	 &= \frac{e^{-\lambda}}{\sqrt{2\pi(1-2\lambda)}} \int  e^{-\frac{1}{2} z^2} dz = \frac{e^{-\lambda}}{\sqrt{(1-2\lambda)}}.
	\end{align*}
	The MacLaurin series of $-\frac{1}{2} \log (1 - 2\lambda)$ reads
	\begin{equation*}
	 \frac{1}{2} \underset{k=1}{\overset{\infty}{\sum}} \frac{(2\lambda)^k}{k}, 
	\end{equation*}
	which yields that the second last display equals
	\begin{equation*}
	\exp \left(\frac{3}{2} \lambda^2 + \frac{1}{2} \underset{k=3}{\overset{\infty}{\sum}} \frac{(2\lambda)^k}{k} \right).
	\end{equation*}
	If $\lambda \leq 1/4$,
	\begin{equation*}
	\underset{k=3}{\overset{\infty}{\sum}} \frac{(2\lambda)^k}{k} \leq \frac{(2\lambda)^3}{1-2\lambda} \leq \lambda^2,
	\end{equation*}
	from which the result follows.
	\end{proof}

	\begin{lemma}\label{lem:inner_product_ind_gaussians}
	Let $a \in \R$ and let $Z,Z' \overset{\text{i.i.d.}}{\sim} N(0,I_d)$ for $d \in \N$. 
	
	Then, $ a \langle  Z , Z' \rangle$ is $C a \sqrt{d}$-sub-exponential for a universal constant $C>0$ and
	\begin{equation*}
	\E  e^{t |a \langle  Z , Z' \rangle|} \leq 2 e^{ {t^2 a^2 d}},
	\end{equation*}
	whenever $|t| \leq (2a^2)^{-1}$.
	\end{lemma}
	\begin{proof}
	Since $\langle Z, Z' \rangle |Z' \sim N(0,\|Z'\|_2)$,
	\begin{align*}
	\E e^{t a \langle  Z , Z' \rangle} &= \E^{Z'} \E^{Z|Z'} e^{t a \langle  Z , Z' \rangle} = \E^{Z'} e^{ \frac{t^2 a^2}{2} \|Z'\|_2^2}.
	\end{align*}
	By Lemma~\ref{lem:chi-square-subgaussianity_lemma}, the latter is further bounded by
	\begin{equation*}
	e^{ \frac{t^2 a^2 d}{2} + \frac{t^4 a^4 d}{2}} \leq e^{{t^2 a^2 d}},
	\end{equation*}
	whenever $t^2 a^2 \leq 1/2$. The conclusion then follows by e.g.\ Proposition 2.7.1 in~\cite{vershynin_high-dimensional_2018}, since $\langle Z, Z' \rangle$ is mean zero. For the last statement,
	\begin{equation*}
	\langle Z, Z' \rangle |Z' \overset{d}{=} - \langle Z, Z' \rangle |Z'.
	\end{equation*}
	Consequently,
	\begin{align*}
	\E^{Z|Z'} e^{t |a \langle  Z , Z' \rangle|} &= \E^{Z|Z'} \mathbbm{1}_{\{ \langle  Z , Z' \rangle > 0 \}} e^{t a \langle  Z , Z' \rangle} + \E^{Z|Z'} \mathbbm{1}_{\{ \langle  Z , Z' \rangle \leq 0 \}} e^{- t a \langle  Z , Z' \rangle} \\
	&\leq 2 \E^{Z|Z'} e^{t a \langle  Z , Z' \rangle},
	\end{align*}
	and the proof follows by what was shown above.
	\end{proof}

\subsection{Lower Bound Common Design Auxilliary Lemmas}  

\begin{lemma}\label{lem:beta-fisher-information}
	Let \(B \sim \text{Beta}(2,2)\). Define \(U = \tfrac{1}{2} - B\). Let \(\mu \in \mathbb{R}\) and consider the random variable
	\[
	X = \mu + U.
	\]
	The Fisher information for the parameter \(\mu\) is given by a positive constant independent of \(\mu\).
	\end{lemma}
	
	\begin{proof}
	The probability density of $B$ is given by
	\[
	f_B(b) = 6b(1-b), \quad b \in [0,1].
	\]
	Hence, the density $f_U$ of \(U\) satisfies
	\[
	f_U(u) = f_B\left(\tfrac{1}{2}-u\right) = 6\left(\tfrac{1}{2}-u\right)\left(\tfrac{1}{2}+u\right) = 6\left(\tfrac{1}{4}-u^2\right).
	\]
	and we find that probability density function of \(X\) is 
	\[
	f_X(x;\mu) = f_U(x-\mu) = 6\left(\tfrac{1}{2} - (x-\mu)\right)\left(\tfrac{1}{2}+(x-\mu)\right).
	\]
	The differentiating the log-likelihood, we have
	\[
	\frac{\partial}{\partial \mu}\log f_X(x;\mu) 
	= \frac{\partial}{\partial \mu}\log f_U(u)\bigg|_{u=x-\mu}
	= -\frac{f_U'(u)}{f_U(u)},
	\]
	where \(u = x-\mu\). Hence, the Fisher information for \(\mu\) is
	\[
	I(\mu) = \mathbb{E}\left[\left(\frac{\partial}{\partial \mu}\log f_X(X;\mu)\right)^2\right] = \int_{-1/2}^{1/2} \frac{f_U'(u)^2}{f_U(u)}\,du,
	\]
	since \(X-\mu \overset{d}{=} U\). We have $\frac{(f_U'(u))^2}{f_U(u)}  = \frac{24u^2}{\tfrac{1}{4}-u^2}$, so the integral is positive and independent of \(\mu\).
	\end{proof}

\end{document}